\documentclass[aap]{imsart}

\RequirePackage{amsthm,amsmath,amsfonts,amssymb}
\RequirePackage[numbers]{natbib}

\startlocaldefs
\usepackage[english]{babel}
\usepackage{hyperref}
\usepackage{amsfonts}
\usepackage{mathtools}
\usepackage{stmaryrd}
\usepackage{fancyhdr}
\usepackage{caption}
\usepackage[ruled,vlined]{algorithm2e}

\usepackage{graphicx,color}
\usepackage{epstopdf}
\usepackage{mathrsfs}  
\usepackage{subcaption}
\usepackage{layouts}
\allowdisplaybreaks

\newtheorem{theorem}{Theorem}[section]
\newtheorem{prop}[theorem]{Proposition}
\newtheorem{lemma}[theorem]{Lemma}
\newtheorem{corollary}[theorem]{Corollary}

\theoremstyle{definition}

\newtheorem{hypothesis}[theorem]{Inductive Hypothesis}

\newtheorem{conjecture}[theorem]{Conjecture}
\newtheorem{coupling}[theorem]{Coupling}

\newtheorem{assumption}[theorem]{Assumption}

\theoremstyle{remark}
\newtheorem{remark}[theorem]{Note}

\newtheorem{example}[theorem]{Example}

\newcommand{\leqnomode}{\tagsleft@true}
\newcommand{\reqnomode}{\tagsleft@false}

\newcommand{\C}{\mathcal{C}}
\newcommand{\N}{\mathbb{N}}

\newcommand{\cL}{\mathcal{L}}
\newcommand{\cU}{\mathcal{U}}
\newcommand{\nuU}{\nu_\cU}

\newcommand{\cP}{\mathcal{P}}
\newcommand{\Xbar}{\overline{X}}
\newcommand{\Vbar}{\overline{V}}
\newcommand{\Qbar}{\overline{Q}}
\newcommand{\Pbar}{\overline{P}}

\newcommand{\Zbar}{\overline{Z}}

\newcommand{\R}{\mathbb{R}}

\newcommand{\vf}{\Phi}
\newcommand{\vfd}{\mathbf{\Phi}}
\newcommand{\ic}{\varphi}

\newcommand{\gE}{E}
\newcommand{\lf}{\overline{G}}
\newcommand{\pot}{\psi}

\newcommand{\ptau}[1]{p_{\overline{\tau}}^{z,#1}}
\newcommand{\taub}{{\overline{\tau}_{n+1}}}
\newcommand{\phibar}{\overline{\varphi}}
\newcommand{\Fbar}{\overline{F}}
\newcommand{\lambdabar}{\overline{\lambda}}
\newcommand{\multi}{\mathrm{Discrete}}
\newcommand{\Eeq}{E_=}
\newcommand{\Unif}{\mathrm{Unif}}
\newcommand{\dd}{\mathrm{d}}

\newcommand{\review}[1]{#1}

\usepackage{booktabs} 
\usepackage{array} 
\usepackage{paralist} 
\usepackage{verbatim} 
\usepackage{fancyvrb}
\usepackage{float}
\usepackage{caption}
\usepackage{subcaption}
\usepackage{bbm} 

\usepackage{url}
\usepackage[toc,page]{appendix}
\usepackage{color}

\usepackage[shortlabels]{enumitem} 

\DeclareMathOperator*{\argmin}{arg\,min}

\endlocaldefs

\begin{document}

\begin{frontmatter}
\title{Approximations of Piecewise Deterministic Markov Processes and their convergence properties}
\runtitle{Approximations of PDMPs}

\begin{aug}
\author[A]{\fnms{Andrea} \snm{Bertazzi}\ead[label=e1,mark]{a.bertazzi@tudelft.nl}},
\author[A]{\fnms{Joris} \snm{Bierkens}\ead[label=e2,mark]{joris.bierkens@tudelft.nl}}
\and
\author[B]{\fnms{Paul} \snm{Dobson}\ead[label=e3,mark]{pdobson@ed.ac.uk}}
\address[A]{Delft Institute of Applied Mathematics, Delft University of Technology, \printead{e1,e2}}

\address[B]{University of Edinburgh, \printead{e3}}
\end{aug}

\begin{abstract}
Piecewise deterministic Markov processes (PDMPs) are a class of stochastic processes with applications in several fields of applied mathematics spanning from mathematical modeling of physical phenomena to computational methods. A PDMP is specified by three characteristic quantities: the deterministic motion, the law of the random event times, and the jump kernels. The applicability of PDMPs to real world scenarios is currently limited by the fact that these processes can be simulated only when these three characteristics of the process can be simulated exactly. In order to overcome this problem, we introduce discretisation schemes for PDMPs which make their approximate simulation possible. In particular, we design both first order and higher order schemes that rely on approximations of one or more of the three characteristics. For the proposed approximation schemes we study both pathwise convergence to the continuous PDMP as the step size converges to zero and convergence in law to the invariant measure of the PDMP in the long time limit. Moreover, we apply our theoretical results to several PDMPs that arise from the computational statistics and mathematical biology literature.
\end{abstract}



\end{frontmatter}

\tableofcontents

\section{Introduction}

Piecewise Deterministic Markov Processes (PDMPs) \cite{Davis1984, Davis1993} are nowadays widely used in  mathematical modelling in fields such as mathematical biology \cite{berg1972chemotaxis,Cloez_etal,rudnicki2017piecewise}, biochemistry \cite{singh2010stochastic}, insurance risk theory \cite{dassios_insurance,embrechts_schmidli_1994}, materials science \cite{pdmp_crack_propagation}, neuroscience \cite{pakdaman_neuron}, and neutron transport \cite{Horton}. Mathematical properties of PDMPs such as stability and stationarity have been extensively investigated in the mathematics community, see e.g.  \cite{Qualitative_prop_pdmp,costadufour2008,pdmp_inv_meas}.
Moreover, in recent years these processes have also quickly gained in popularity for purposes of Monte Carlo computation in statistical physics \cite{MichelKapferKrauth2014,PetersDeWith2012,Turitsyn_2011} and in Bayesian statistics \cite{fearnhead2018,vanetti2017piecewisedeterministic}, for example in the form of the Bouncy Particle Sampler (BPS) and the Zig-Zag Sampler (ZZS) \cite{BPS, ZZ}. Several papers have further investigated the use of PDMPs in this area, e.g. \cite{andrieu2019hypocoercivity,AndrieuLivingstone,bertazzi2020adaptive,Bierkensergodicity,bierkens2018highdimensional,BPS_Durmus,guillin_lowtemp,Lu2020}.

PDMPs are continuous time Markov processes which move along deterministic trajectories (typically in Euclidean space) on a time interval of random length, after which a (possibly random) transition occurs to a new state, followed by another deterministic motion, etc. The deterministic motion is prescribed by the integral curves, $\varphi_t$, of a vector field $\vf : \R^d \rightarrow \R^d$, the length of the random time intervals between transitions is governed by a transition rate $\lambda :\R^d \rightarrow [0,\infty)$, and the transitions are described by a Markov kernel $Q : \R^d \times \mathcal B(\R^d) \rightarrow [0,1]$. Together the vector field $\Phi$, transition rate $\lambda$ and transition kernel $Q$ comprise the \emph{characteristics} of the PDMP.


These processes are relatively easy to understand from a conceptual point of view and in some special cases their simulation can be performed exactly. In particular, if (i) the vector field $\Phi$ is explicitly integrable, (ii) it is possible to generate random times exactly as prescribed by $\lambda$, and (iii) it is possible to simulate from the transition kernel $Q$, then the iterative computation of trajectories of the associated PDMP is relatively straightforward. 

Simulation of trajectories becomes problematic if one or more of these conditions are not met.  Let us discuss possible problems that may arise. Concerning (i), the vector field $\Phi$, as is well known in the field of differential equations, explicit solutions to the ODE $\dot \varphi_t = \Phi(\varphi_t)$ are only available in special cases, for example when $\Phi$ is affine, or when it is has some other special structure or symmetry. Concerning (ii), the transition rate $\lambda$, it is easy to simulate when the rate $\lambda$ is constant or globally bounded. If $\lambda$ is constant, then the random times between transitions are simply Exponential($\lambda$)-distributed and thus easily simulated. If $\lambda$ is globally bounded, say by a constant $M$, we may use a technique called \emph{Poisson thinning}  \cite{lewis_shedler_thinning}, which allows us to first simulate the random times according to an Exponential($M$)-distribution and then accept a proposed transition time  as a true transition with a probability governed by the ratio between $\lambda(\cdot)$ and $M$. The use of Poisson thinning may be extended to cases with non-constant bounds $M(s)$ along trajectories under the condition that it is simple to simulate from an inhomogeneous Poisson process with rate $M(s)$. However, finding a sharp bound $M(s)$ can be an extremely challenging problem in most practical settings. Moreover, the looser the bound the greater the computational cost of the simulation of the PDMP. For more extensive descriptions of Poisson thinning we refer to, e.g., \cite{BPS, ZZ}. Finally problems with (iii), the simulation of transitions according to $Q$, may arise in various ways. For instance it may be interesting to approximate the transition kernel of the BPS (see Section 2.4 of \cite{sherlock2021discrete}).

In this paper we propose several schemes to approximate a PDMP in cases that are otherwise not straightforward to simulate, and we accompany these schemes by a detailed analysis of the convergence of the approximate process towards its exact, theoretical counterpart as the parameter governing the numerical precision, $\delta$, converges to zero. Moreover, in the setting in which the PDMP is geometrically ergodic with a specified invariant measure, we investigate  the theoretical convergence of the law of the approximate scheme to the invariant measure of the PDMP.

We introduce the \emph{Fully Discrete PDMP} (FD-PDMP) Algorithm, the \emph{Partially Discrete PDMP} (PD-PDMP) Algorithm and the Higher Order Partially Discrete PDMP Algorithm (Algorithms~\ref{alg:basic_EPDMP}, \ref{alg:advanced_EPDMP} and \ref{alg:secondorder_EPDMP}, respectively). 
The FD-PDMP algorithm (Algorithm~\ref{alg:basic_EPDMP}) defines a Markov chain $\{\Zbar_{t_n}\}_{n\in\N}$ on a mesh $0=t_0< t_1 < t_2 <\dots$ that moves deterministically between time steps, and a random event may occur at each of the mesh points with suitable probability. The PD-PDMP algorithm (Algorithm~\ref{alg:advanced_EPDMP}) defines a Markov chain that moves deterministically with exception of at most one random event in each interval of the form $[t_{n-1},t_n]$. In contrast to the FD-PDMP the random event does not need to occur at mesh points. This difference motivates the choice of name of the two algorithms. By allowing at most $p$ random events per time step, the higher order algorithm (Algorithm~\ref{alg:secondorder_EPDMP}) constructs an approximation of the PDMP of order $p$. 

Naturally these algorithms are designed to be straightforward to simulate. Both the FD-PDMP and the PD-PDMP algorithms rely on first order approximations of the characteristics of the PDMP. A wide range of approximations for $\varphi_t,\lambda,Q$ is allowed, see Assumptions \ref{ass:integrator},  \ref{ass:approx_jump_kernel}, \ref{ass:lambda_wasserstein} for the formal requirements. As a simple yet important example, consider the case in which we are interested in simulating a PDMP for which the event times are hard to obtain. With an Euler-type approach, we can use an approximation of $\lambda$ that is constant between mesh points, based on the state of the process at the initial point of each time step. For such approximation, the next event time in the case of PD-PDMP is simply exponentially distributed with constant rate, which is straightforward to simulate. Similarly, in the case of the FD-PDMP a random event takes place at the end of the time interval according to a Bernoulli distributed random variable. In comparison to the simulation of the continuous time PDMP, both algorithms do not require an upper bound to the switching rates, which is required to apply Poisson thinning. In a similar fashion, simple approximations of $\varphi_t$ and $Q$ can be employed. We refer to Section \ref{sec:algorithms} for a detailed description of the algorithms.
 
We study convergence of these algorithms as a function of the step size and of the time horizon. Under very broad assumptions on the approximation, in particular allowing for approximations of all three $\lambda$, $\varphi_t$, and $Q$, in Theorem~\ref{thm:strong_error_pdmp} we are able to show convergence in a Wasserstein distance to the PDMP as the step size tends to zero. In the case in which it is possible to simulate $\varphi_t$ and $Q$ exactly, we obtain convergence of the PD-PDMP algorithm in the stronger metric of total variation (see Theorem~\ref{thm:tv_distance}). In this setting weaker assumptions on the continuous time PDMP are required. For instance we show in Examples \ref{ex:BPS_wass} and \ref{ex:bps_tv} that BPS satisfies the assumptions of Theorem~\ref{thm:tv_distance} but not those of Theorem~\ref{thm:strong_error_pdmp}. Moreover, both Theorems establish convergence of order $p$ as long as the approximations of $\varphi_t$, $\lambda$, and $Q$ are of order $p$. The proofs of both these theorems rely on couplings of the continuous time PDMP with its approximation and are described respectively in Couplings \ref{coup:wass_gen_pdmp} and \ref{coup:error_tv_distance}. 

In many areas it is important to understand the long time behaviour of the approximation schemes. 
In the field of Markov chain Monte Carlo (MCMC) algorithms the goal is to simulate a process that converges in law to the correct probability measure, which is the posterior distribution in Bayesian statistics and the Boltzmann-Gibbs distribution in statistical physics. In this context, such a probability measure is the invariant distribution of the PDMP.
In Theorem~\ref{thm:weakerror} we prove uniform in time convergence of the weak error between the PDMP and the approximations given by the FD-PDMP or the PD-PDMP algorithms. In particular, we obtain convergence in law of the approximation and its time average to the invariant measure of the PDMP in the joint limit as time tends to infinity and step size tends to zero (see Corollary~\ref{cor:convergence-to-stationary-measure}).

We confirm the applicability of our theorems on a variety of examples. ZZS and BPS are instances of PDMPs for which exact simulation of the random event times is not always possible. In Example \ref{ex:ZZ} we discuss how to approximate the ZZS, and in Examples \ref{ex:zzs_wass}, \ref{ex:zzs_weakerror} we show that our Theorems apply to the proposed approximation. An attractive feature of ZZS is that it allows for exact \emph{subsampling} (see \cite{ZZ}), which means that in a Bayesian statistics setting for each ``iteration" of the algorithm only a subset of the data has to be accessed. In Example \ref{ex:zzs_subsampling} we propose an approximation of ZZS with subsampling which also has the property of accessing a batch of the data over each time step and prove convergence in total variation as the step size tends to zero. For BPS we construct an approximation and prove convergence as step size tends to zero and time tends infinity in Examples \ref{ex:BPSintro}, \ref{ex:bps_tv} and \ref{ex:bps_weakerror}. In contrast to ZZS and BPS, randomised Hamiltonian Monte Carlo (RHMC) (see \cite{bou2017randomized}) is an example of a PDMP in which it is typically not possible to simulate the flow $\varphi_t$ exactly. In Examples \ref{ex:randHMC}, \ref{ex:rhmc_wass}, and \ref{ex:rhmc_weakerror} we discuss approximations of RHMC and show convergence as the step size tends to zero and time tends to infinity. We also discuss continuous time approximation schemes of a PDMP in Examples \ref{ex:continuous_approx_pdmp} and \ref{ex:continuous_approx_weakerror}.

\subsection*{Related works}
Whereas discretisations of stochastic differential equations such as the Langevin equation have been studied extensively in the literature, see e.g. the book \cite{kloeden1992stochastic} or recent papers \cite{Durmus_Moulines_NAunadj,Durmus_Moulines_highdim,sanzserna2021wasserstein}, the same has not been done for PMDPs.
Here we give a brief overview of works that are to some extent related to the present manuscript. 

An approximation scheme for PDMPs suitable for a specific setting was proposed in \cite{lemaire_thieullen_thomas_2020}. The authors consider the case in which the ODE describing the deterministic motion can only be solved numerically, a global upper bound for the switching rates is available, and the kernels $Q$ can be simulated exactly. Their proposal is to move deterministically according to a numerical integrator and to draw a proposal for the following event time according to the upper bound of the switching rates and then accept or reject it by Poisson thinning. The framework we propose in Algorithms \ref{alg:advanced_EPDMP} and \ref{alg:secondorder_EPDMP} is more general as approximations of all characteristics are possible. Moreover, the approximations in this manuscript do not require existence or knowledge of an upper bound for the switching rates. As discussed in Example \ref{ex:Morris-Lecar} it is possible to closely resemble the proposal of \cite{lemaire_thieullen_thomas_2020} using our framework. Moreover, we obtain similar finite time strong and weak error results as in \cite{lemaire_thieullen_thomas_2020} by applying Theorem~\ref{thm:strong_error_pdmp}. 

In \cite{vanetti2017piecewisedeterministic} the authors focus on how to design a discrete time PDMP with a specific invariant measure. This is a fundamentally different approach to the focus of this paper. A related work is \cite{sherlock2021discrete}, which defines a discrete time chain that resembles a BPS.

The book \cite{cocozzamarkov} discusses approximations of PDMP based on finite volume schemes for the Chapman-Kolmogorov type equations. Such schemes approximate the law of the process and are thus different in nature compared to this manuscript.

Finally, we discuss  papers that deal with continuous time approximations of the ZZS. In \cite{cotter2020nuzz} the authors propose an approximation that relies on an integrator and a root finding method to generate the random event times. The paper \cite{huggins2017quantifying} discusses the effect of approximate switching rates $\{\Tilde{\lambda}\}_{i=1}^d$ on the stationary measure of the ZZS. This approximation relies on the availability of suitable $\{\Tilde{\lambda}\}_{i=1}^d$ for which it is possible to (efficiently) simulate the corresponding ZZS. In Examples \ref{ex:continuous_approx_pdmp} and \ref{ex:continuous_approx_weakerror} we discuss applications of our theory to these approximation schemes. \review{A similar setting is considered in \cite[Theorems 11 and 25]{pdmp_inv_meas}, where the authors establish bounds in total variation distance between (the invariant measures of) two PDMPs with same deterministic dynamics, but different switching rates and jump kernels. The authors prove such bounds by a coupling of the two continuous time PDMPs that is similar in spirit to our Coupling \ref{coup:error_tv_distance}. In this paper, in particular in Section \ref{sec:tv_distance}, we bound the TV distance between a PDMP and a discrete time approximation. Thus the statements and proofs differ from \cite{pdmp_inv_meas} in this sense.}

\subsection*{Organisation of the paper}
The paper is organised as follows. In Section \ref{sec:notation} we define notation that we use throughout the paper. In Section \ref{sec:algorithms} we describe the setting and the proposed algorithms. In particular in Section \ref{sec:algorithms_firstorder} we discuss first order schemes and in Section \ref{sec:higherorder} we consider higher order schemes. Section \ref{sec:main_results} contains the main results together with the required assumptions. This section is divided into three parts. Section \ref{sec:Wass_bounds} is devoted to convergence in Wasserstein distance, which is established in Theorem~\ref{thm:strong_error_pdmp}. Section \ref{sec:tv_distance} concerns convergence in total variation as stated in Theorem~\ref{thm:tv_distance}. Section \ref{sec:weakerror} gives conditions for uniform in time convergence of the weak error, as expressed by Theorem~\ref{thm:weakerror}. \review{In Section \ref{sec:examples} we gather examples to demonstrate the when the assumptions of the main theorems are satisfied.} 
The proofs of the three main theorems can be found respectively in Section \ref{sec:proof_wass_theorem}, Section \ref{sec:proofs_tv} and Section \ref{sec:proof_weakerror_thm}. All other results as well as all auxiliary lemmas from Sections \ref{sec:Wass_bounds}, \ref{sec:tv_distance}, and \ref{sec:weakerror} can be found respectively in Appendix \ref{sec:proofs_wasserstein_appendix}, Appendix \ref{sec:proofs_tv_appendix}, and Appendix \ref{sec:proofs_weakerror_appendix}.

\section{Notation}\label{sec:notation}
We denote the semigroup of the continuous time PDMP, $\{Z_t\}_{t\geq 0}$, as $\cP_t$ which acts on suitable functions by
\begin{equation*}
    \cP_tf(z) = \mathbb{E}_z[f(Z_t)].
\end{equation*}
Here the subscript $z$ denotes that the process $Z_t$ has initial position $Z_0=z$. Note that the semigroup is related to the transition probability of $Z_t$, which is denoted by $\cP_t(z,A)$. These concepts are related for any function $f$ and measurable set $A\subseteq E$ by 
\begin{equation*}
    \cP_tf(z)=\int f(y) \cP_t(z,\dd y), \quad \cP_t(z,A)=(\cP_t\mathbbm{1}_A)(z).
\end{equation*}
Similarly we denote the transition probability of the approximation processes $\{\Zbar_{t_n}\}_{n\in\N}$ as $\overline{\cP}_{t_n}$. 

Consider a metric $d:\mathbb{R}^d\times\mathbb{R}^d \to \mathbb{R}_+$ and let $P,Q$ be probability measures on $\mathbb{R}^d$. Then we define the Wasserstein distance of order $1$ with respect to the metric $d$ as
\begin{equation}\label{eq:wasserstein}
    \mathcal{W}_1(P,Q) = \inf_{R\in \Pi(P,Q)} \left\{ \int_{\mathbb{R}^d\times\mathbb{R}^d} d(x,y)  R(\dd x,\dd y) \right\},
\end{equation}
where $\Pi(P,Q)$ is the set of couplings of the two probability measures $P,Q$, that is the set of probability measures $R$ on $\mathbb{R}^d\times\mathbb{R}^d$ such that $R(A,\mathbb{R}^d)=P(A)$ and $R(\mathbb{R}^d,B)=Q(B)$.

We will denote a norm by $\lVert\cdot \rVert$. The maximum between $a\in\mathbb{R}$ and $0$ is denoted by $(a)_+=\max\{a,0\}$.

Let us define the space $\C^k$ to be the set of functions $f:\R^d \to \R$ which are $k$ times continuously differentiable. $\C^k_b$ ($\C_c^k$ respectively) denotes the subset of $\C^k$ to functions which are bounded (resp. have compact support) with bounded and continuous derivatives up to order $k$. We endow the space $\C_b$ with the supremum norm $\lVert\cdot\rVert_\infty$ and the space $\C_b^1$ is endowed with the norm
\begin{equation*}
    \lVert f\rVert_{\C_b^1} = \lVert f\rVert_\infty + \sum_{i=1}^d\lVert \partial_i f\rVert_\infty.
\end{equation*}

Consider a random variable $I$ with values in $\{1,\dots,m\}$ such that $\mathbb{P}(I=i)=w_i$ for $i=1,\dots,m$. Then we say $I$ has a discrete distribution with probabilities $w_i$ and we denote this as $I\sim \multi(\{w_i\}_{i=1}^m)$.

Given a measure $\pi$ we define for any $f\in L^1_\pi$
\begin{equation*}
    \pi(f)=\int f(z) \pi(\dd z).
\end{equation*}
Similarly for a probability kernel $Q(x,\dd y)$ we  write
\begin{equation}\label{eq:kernelnotation}
    Qf(x) = \int f(y) Q(x,\dd y)
\end{equation}
for $f$ measurable and integrable with respect to $Q(x,\dd y)$ for all $x$.
Note that \eqref{eq:kernelnotation} allows us to consider a probability kernel as a map from the space of bounded and measurable functions, $B_b$, to $B_b$.

Let us define the \emph{total variation distance} between two probability measures $\mu$, and $\nu$ as
\begin{equation*}
    \lVert \mu-\nu\rVert_{TV} = \sup_{f\in \C_b:\lVert f\rVert_\infty\leq 1}\lvert \mu(f)-\nu(f)\rvert.
\end{equation*}

Given a vector field $\vf$ we can view this as a map, $\vfd$, which acts on $\C^1$ functions as
\begin{equation*}
    \vfd(f)(x) = \vf(x)^T\nabla f(x), \quad \text{ for } f\in \C^1.
\end{equation*}

Given two maps $X,Y: \C^\infty \to \C^\infty$ we shall define the commutator of $X$ and $Y$, $[X,Y]$ to be the map $[X,Y]:\C^\infty\to \C^\infty$ by
\begin{equation*}
    [X,Y]f=XYf-YXf, \quad \quad \text{ for } f\in \C^\infty.
\end{equation*}
We will use this with the maps $\vfd$ and $Q$. If we assume that $Q$ preserves $\C^1$ and $\vfd$ is bounded then we can view $[\vfd,Q]:\C^1_b\to B_b$ defined by 
\begin{equation*}
    [\vfd,Q]f=\vfd (Q f)-Q\vfd (f), \quad \quad \text{ for } f\in \C^1_b.
\end{equation*}
\review{Note that although the commutator was defined for smooth vector fields the above definition makes sense for all $C_b^1$-functions since $\vfd$ is a map from $\C_b^1$ to $\C_b$ and we have $Q(\C_b^1)\subseteq \C_b^1$, $Q(\C_b)\subseteq \C_b$ so both the operations $Q(\vfd f)$ and $\vfd (Qf)$ are well defined for $f\in \C_b^1$.}

\section{Algorithms}\label{sec:algorithms}
Consider a PDMP $(Z_t)_{t\geq 0}$ taking values on a state space $E$, which is a subset of a finite dimensional vector space. Examples are $E=\mathbb{R}^d\times \mathbb{R}^d$ or $E=\mathbb{R}^d\times \{-1,+1 \}^d$. The dynamics of the process are described by the generator $\cL$, which applied on a function \review{in the domain of the extended generator} gives
\begin{equation}\label{eq:genPDMPmanykernels}
\cL f(z) = \langle \vf(z), \nabla_zf(z)\rangle + \sum_{i=1}^m\lambda_i(z)\int_E (f(y)-f(z)) Q_i(z,\dd y).
\end{equation}
\review{The generator here is understood to be the extended generator, see \cite[Theorem 26.14]{Davis1993} for the exact description of the domain of the extended generator. Note, in particular that functions that are differentiable in the direction $\Phi$ and bounded are included in the domain.} Here $\Phi$ is a smooth and globally Lipschitz vector field, $\lambda_i:E\to [0,\infty)$ are continuous functions and $Q_i$ are probability kernels. 
Let $\varphi_t$ denote the integral curve of $\Phi$, i.e. the solution to the following ordinary differential equation (ODE)
\begin{equation*}
\frac{\dd }{\dd t}\varphi_t(z)=\vf(\varphi_t(z)), \quad \varphi_0(z)=z, \quad \text{ for all } t\geq 0, z\in E.
\end{equation*}
Note that $\varphi_t$ exists since $\vf$ is globally Lipschitz. We assume that $\varphi_t$ leaves $E$ invariant. Define the total switching rate
\begin{equation*}
    \lambda(z) = \sum_{i=1}^m \lambda_i(z).
\end{equation*} 
\review{As shown in \cite[Section 26]{Davis1993} \eqref{eq:genPDMPmanykernels} corresponds to a PDMP where} the next event time is distributed as
\begin{equation}\label{eq:switching_time_pdmp}
    \mathbb{P}_z(\tau \leq t) = 1-\exp\left(-\int_0^t \lambda(\varphi_s(z))\dd s\right),
\end{equation}
and that between two random events the process follows the flow-map $\varphi_t$, i.e. $Z_t=\varphi_t(z)$. At event time, $\tau$, the process jumps according to probability kernel $Q_I$, where $I$ is distributed according to the following discrete distribution $$I\sim\multi\left(\left\{\frac{\lambda_i(\varphi_\tau(z))}{\lambda(\varphi_\tau(z))}\right\}_{i=1}^m\right).$$ 
Algorithm \ref{alg:pdmp} describes the simulation procedure for a PDMP with generator \eqref{eq:genPDMPmanykernels}.

\begin{algorithm}[t]
\SetAlgoLined
\SetKwInOut{Input}{Input}\SetKwInOut{Output}{Output}
\Input{Time horizon $T$, initial condition $z$.}
 Set $t = 0$, $Z_0 = z$\;
 \While{$t < T$}{
  simulate next event time as 
  $$\tau = \inf\left\{ r>0: 1-\exp\left(-\int_0^r \lambda(\varphi_s(Z_t))\dd s\right) \geq U \right\} $$ where $U \sim \text{Unif}[0,1]$ \;
  simulate $Z_{t+s}=\varphi_s(Z_t)$ for $s\in(0,\tau)$\;
  draw $I\sim \multi\big(\big\{\frac{\lambda_i(Z_{t+\tau-})}{ \lambda(Z_{t+\tau-})}\big\}_{i=1}^m\big)$\; 
  simulate $Z_{t+ \tau} \sim Q_I(Z_{t+\tau-},\cdot)$\;
  set $t=t+\tau$\;
 }
 \caption{Pseudo-code for the simulation of a PDMP}
 \label{alg:pdmp}
\end{algorithm}

\begin{remark}\label{note:manykernels}
It is possible to rewrite \eqref{eq:genPDMPmanykernels} to the form 
\begin{equation}\label{eq:genPDMPonekernel}
\cL f(z) = \langle \vf(z), \nabla_zf(z)\rangle + \lambda(z)\int_E (f(y)-f(z)) Q(z,\dd y)
\end{equation}
for some continuous function $\lambda: E\to [0,\infty)$ and probability kernel $Q$. Indeed this can be achieved by setting 
\begin{align}\label{eq:defoofQandlambda}
    \lambda(z) = \sum_{i=1}^m \lambda_i(z), \quad 
     Q(z,\dd y)=\sum_{i=1}^m \frac{\lambda_i(z)}{\lambda(z)} Q_i(z,\dd y).
\end{align}
Therefore there is no loss of generality for the PDMP to take $m=1$. However we will see in Section \ref{sec:Wass_bounds} that allowing $m\geq 1$ leads to weaker assumptions for our convergence results, in particular we will see in Example \ref{ex:zzs_wass} a case where the assumptions are satisfied with $m>1$ but would not be satisfied when written in the form \eqref{eq:genPDMPonekernel}. 
\end{remark}

The focus of this paper is to define and analyse approximations of PDMPs that can be employed in settings where their simulation cannot be performed exactly. As explained in the introduction, there are three quantities which characterise a PDMP and may be difficult to simulate. These are the flow map $\varphi_t$, the random event times with rates $\lambda_i$, and the Markov kernels $Q_i$. The idea is then to introduce $p$-th order approximations of the three characteristics for some $p\geq 1$. Precise conditions on the approximations are given in Assumptions \ref{ass:integrator}, \ref{ass:approx_jump_kernel}, \ref{ass:lambda_wasserstein}, but here we provide a heuristic description.
The flow map $\varphi_t(z)$ can be approximated with a numerical integrator, which is denoted as $\phibar_t(z;\delta,p)$. The parameters $\delta,p$ have the meaning that for $s\in[0,\delta]$ we have that $\phibar_s(z;\delta,p)$ is an approximation of order $\delta^p$ of $\varphi_s(z)$. Classical examples of numerical integrators from the ODE literature include the Euler discretisation, the leap frog scheme, and higher order numerical schemes. Then we want to approximate the switching rates in such a way that the random times \eqref{eq:switching_time_pdmp} can be simulated easily at the cost of a small error. This can be done by using order $\delta^p$ approximations of $\lambda(\varphi_s(z))$, i.e. the switching rate along the deterministic flow. We denote the corresponding approximation as $\lambdabar(z,s;\delta,p):E\times[0,\infty)\to [0,\infty)$.
The motivation is to ensure the following as an approximation of order $\delta^p$ for $t\leq \delta$:
\begin{equation}\notag
    \mathbb{P}_z(\tau \leq t) \approx 1-\exp\left(-\int_0^t \lambdabar(z,s;\delta,p)\dd s\right).
\end{equation}
Here
\begin{equation*}
    \lambdabar(z,s;\delta,p)=\sum_{i=1}^m \lambdabar_i(z,s;\delta,p).
\end{equation*}
Let us give some examples with $p=1$.
A possible choice is to ``freeze'' the switching rate, thus taking $\overline{\lambda}_{i}(z,s;\delta,1) = \lambda_i(z)$. This is supported by the intuition that $\lambda(\varphi_s(z))\approx \lambda(z)$ for small $s$. In this case  $\mathbb{P}_z(\tau \leq t)$ is approximately equal to $1-\exp\left(-t \lambda(z)\right)$, which is the cumulative distribution function of the exponential distribution with constant rate $\lambda(z)$. We refer to the $\lambdabar_i$ as \textit{frozen switching rates} and to the corresponding approximation process as \emph{Euler approximation}. Alternatively one could take $\lambdabar_i(z,s;\delta,1) = \lambda_i(\varphi_\delta(z))$, or the switching rates along the trajectory given by the numerical integrator $\lambdabar_i(z,s;\delta,1) = \lambda_i(\phibar_s(z;\delta,1))$, or more generally $\lambdabar_i(z,s;\delta,p) = \lambda_i(\phibar_s(z;\delta,p))$. 
Finally, consider the Markov kernels $Q_i$. We define a function $F_i$ which describes a choice of implementation of $Q_i$. Let $F_i:E\times\mathcal{U}\to E$ be a deterministic map such that $F_i(z,U)$ is distributed according to $Q_i(z,\cdot)$ when $U$ is distributed according to a probability distribution $\nuU$. We can then approximate each map $F_i$ by a map $\Fbar_{i}(\cdot;\delta,p):E\times\mathcal{U}\to E$, where once again $\delta,p$ denotes the order of accuracy of our estimate. To simplify the notation, when we consider first order schemes, i.e. $p=1$, we shall suppress the $p$-dependence and write $\phibar_s(z;\delta)$, $\lambdabar_i(z,s;\delta)$, $\Fbar_i(z,U;\delta)$.

Now that we have introduced the problem and the various approximations we wish to exploit, we illustrate how to design first order and higher order approximation schemes for PDMPs. 
By an \emph{order $p$ scheme} we mean an approximation process for which the \emph{local error}, i.e. the error between the PDMP and the approximation over a step of size $\delta$ with identical initial conditions, is proportional to $\delta^{p+1}$.
Therefore after $n$ steps of size $\delta$ the \emph{global error} is proportional to $t_n\delta^p$ where $t_n=n\delta$, which motivates the term order $p$ scheme.

\subsection{First order schemes}\label{sec:algorithms_firstorder}
Let us introduce a mesh $\{t_n\}_{n\in\N}$ for the time variable where $t_n=\sum_{\ell=1}^n\delta_\ell$, and $\delta_\ell$ are step sizes. For example if the step size is constant $\delta_\ell=\delta$ then $t_n=n\delta$ for all $n\in \N$. In this section we introduce two alternative first order schemes: the FD-PDMP algorithm and the PD-PDMP algorithm. We define the FD-PDMP approximation $\{\Zbar_{t_n}\}$ on the mesh $\{t_n\}_{n\in\N}$ by setting $\Zbar_{0}=z$ and then following the procedure
\begin{align*}
\tilde{Z}_{t_{n+1}}&=\phibar_{\delta_{n+1}}(\Zbar_{t_n};\delta_{n+1}),\\
\Zbar_{t_{n+1}}&=\alpha_{n+1}\Fbar_{\bar{I}_{n+1}}(\tilde{Z}_{t_{n+1}},U_{n+1};\delta_{n+1})+(1-\alpha_{n+1})\tilde{Z}_{t_{n+1}} .
\end{align*}
Here we have $U_{n+1}\sim\nuU$. The value of $\alpha_{n+1}$ is determined as follows. 
We simulate $\bar{\tau}$ which has distribution conditional on $\Zbar_{t_n}$ given by
\begin{equation*}
    \mathbb{P}_z(\bar{\tau} \leq t\rvert \Zbar_{t_n}) = 1-\exp\left(-\int_0^t \lambdabar(\Zbar_{t_n},s;\delta_{n+1})\dd s\right). 
\end{equation*}
Then $\alpha_{n+1}=1$ if and only if $\bar{\tau}\leq \delta_{n+1}$, otherwise $\alpha_{n+1}=0$. 
We then draw
\begin{equation}\label{eq:multinomial_draw}
    \bar{I}_{n+1} \sim \multi\left(\left\{\frac{\lambdabar_i(\Zbar_{t_n},\bar{\tau};\delta_{n+1})}{\lambdabar(\Zbar_{t_n},\bar{\tau};\delta_{n+1})}\right\}_{i=1}^m\right).
\end{equation}
The resulting Markov chain $\Zbar_{t_n}$ is thus updated by first following the approximate flow map  and then establishing whether a random event takes place at the end of the current time interval. This procedure is written in pseudo-code form in Algorithm \ref{alg:basic_EPDMP}.
Note that if $\lambdabar(z,s;\delta_{n+1})$ is independent of $s$, i.e. $\lambdabar(z,s;\delta_{n+1})= \lambdabar(z;\delta_{n+1})$, then we do not need to simulate $\taub$ and we have that $\alpha_{n+1}$ is a Bernoulli random variable with success rate $1-\exp(-\delta_{n+1}\lambdabar(z;\delta_{n+1}))$ and $\bar{I}_{n+1}$ is distributed as
\begin{equation*}
    \bar{I}_{n+1} \sim \multi\left(\left\{\frac{\lambdabar_i(\Zbar_{t_n};\delta_{n+1})}{\lambdabar(\Zbar_{t_n};\delta_{n+1})}\right\}_{i=1}^m\right).
\end{equation*}
This is for instance the case of frozen switching rates.
\begin{algorithm}[t]
\SetAlgoLined
\SetKwInOut{Input}{Input}\SetKwInOut{Output}{Output}
\Input{Number of iterations $N$, initial condition $z$, step sizes $(\delta_n)_{n=0}^N$.}
\Output{Chain $(\Zbar_{t_{n}})_{n=0}^N$.}
 Set $n=0$, $\Zbar_0 = z$\;
 \While{$n < N$}{
  simulate $\tilde{Z} = \phibar_{\delta_{n+1}}(\Zbar_{t_n};\delta_{n+1})$\;
  simulate 
  $$\bar{\tau} = \inf\left\{ r>0: 1-\exp\left(-\int_0^r \lambdabar(\Zbar_{t_n},s;\delta_{n+1})\dd s\right) \geq U \right\} $$ where $U \sim \text{Unif}[0,1]$ \;
   \If{$\bar{\tau}\leq \delta_{n+1}$}{
   draw $U_{n+1} \sim \nuU$ and $\bar{I}_{n+1}\sim \multi\Big(\Big\{\frac{\lambdabar_i(\Zbar_{t_n},\Bar{\tau};\delta_{n+1})}{\lambdabar(\Zbar_{t_n},\Bar{\tau};\delta_{n+1})}\Big\}_{i=1}^m\Big)$\;
   set $\tilde{Z} = \Fbar_{\bar{I}_{n+1}} (\tilde{Z},U_{n+1}; \delta_{n+1})$\;
  }
  set $\Zbar_{t_{n+1}}=\tilde{Z}$\;
  set $n=n+1$\;
 }
 \caption{Fully Discrete Approximation of a PDMP}
 \label{alg:basic_EPDMP}
\end{algorithm}

A different approach is shown in Algorithm \ref{alg:advanced_EPDMP}, which describes the PD-PDMP approximation. Here the idea is to simulate the switching time $\bar{\tau}$ with rate $\lambdabar(\Zbar_{t_n},s;\delta_{n+1})$, then if $\bar{\tau}$ is before the end of the current time step set $t=t_n+\bar{\tau}$, draw $\overline{I}_{n+1}$ as in \eqref{eq:multinomial_draw}, and follow the procedure below:
\begin{alignat*}{2}
& \Zbar_{t} &&=\overline{F}_{\bar{I}_{n+1}}(\tilde{Z}_{t},U_{n+1};\delta_{n+1}) , \quad \text{where }  \tilde{Z}_{t} =\phibar_{\bar{\tau}}(\Zbar_{t_n};\delta_{n+1}),\\
& \Zbar_{t_{n+1}} &&= \phibar_{t_{n+1}-t}(\Zbar_{t};\delta_{n+1}).
\end{alignat*}
On the other hand, when $\bar{\tau}>\delta_{n+1}$ the process is simply moving deterministically according to the approximate flow map, i.e. $\Zbar_{t_{n+1}}= \phibar_{\delta_{n+1}}(\Zbar_{t_n};\delta_{n+1})$.
\begin{algorithm}[t]
\SetAlgoLined
\SetKwInOut{Input}{Input}\SetKwInOut{Output}{Output}
\Input{Number of iterations $N$, initial condition $z$, step sizes $(\delta_n)_{n=0}^N$.}
\Output{Chain $(\Zbar_{t_{n}})_{n=0}^N$.}
 Set $n=0$, $\Zbar_0 = z$\;
 \While{$n < N$}{
  simulate 
  $$\bar{\tau} = \inf\left\{ r>0: 1-\exp\left(-\int_0^r \lambdabar(Z_t,s;\delta_{n+1})\dd s\right) \geq U \right\} $$ where $U \sim \text{Unif}[0,1]$ \;
  \eIf{$\bar{\tau}< \delta_{n+1}$}{
   set $t=t_n+\bar{\tau}$\;
   simulate $\tilde{Z}_{t} = \phibar_{\bar{\tau}}(\Zbar_{t_n};\delta_{n+1})$\;
   draw $U_{n+1} \sim \nuU$ and $\bar{I}_{n+1} \sim \multi\Big(\Big\{\frac{\lambdabar_i(\Zbar_{t_n},\Bar{\tau};\,\delta_{n+1})}{\lambdabar(\Zbar_{t_n},\Bar{\tau};\,\delta_{n+1})}\Big\}_{i=1}^m\Big)$\;
   set $\Zbar_t = \Fbar_{\bar{I}_{n+1}} (\tilde{Z}_t,U_{n+1}; \delta_{n+1})$\;
   simulate $\Zbar_{t_{n+1}} = \phibar_{t_{n+1}-t}(\Zbar_{t};\delta_{n+1})$\;
   }{
   simulate $\Zbar_{t_{n+1}} = \phibar_{\delta_{n+1}}(\Zbar_{t_n};\delta_{n+1})$\;
  }
  set $n=n+1$\;
 }
 \caption{Partially Discrete Approximation of a PDMP}
 \label{alg:advanced_EPDMP}
\end{algorithm}
Only one random jump per time step is allowed, and in this case it happens at time $\bar{\tau}$ instead of at the end of the time step. This choice comes with advantages and disadvantages if compared to Algorithm \ref{alg:basic_EPDMP}. As we shall see in Sections \ref{sec:tv_distance} and \ref{sec:weakerror} it is possible to obtain stronger results under weaker assumptions on the PDMP in the setting of Algorithm \ref{alg:advanced_EPDMP} compared to Algorithm \ref{alg:basic_EPDMP}. However this may come at a larger computational cost (see e.g. Example \ref{ex:BPSintro}).

\subsection{Examples}\label{sec:algorithms_examples}
In this section we introduce several examples, which will be revisited as illustrative applications of our results. In the first three examples, i.e. Examples \ref{ex:ZZ}, \ref{ex:BPSintro}, \ref{ex:randHMC}, we discuss MCMC samplers which target a probability measure with density $\pi(x)\propto \exp(-\pot(x))$ for $x\in \mathbb{R}^d$.
\begin{example}[Zig-Zag sampler \cite{ZZ}]\label{ex:ZZ}
	Let $E=\R^d\times\{+1,-1\}^d$, and for any $z\in E$ we write $z=(x,v)$ for $x\in \R^d$, $v\in\{+1,-1\}^d$. Set  $\Phi(x,v)=(v,0)^T$, $m=d$, $\lambda_i(x,v)=(v_i\partial_i\pot(x))_++\gamma_i(x,v)$, and $Q_i((x,v),(\dd y,\dd w))=\delta_{(x,R_iv)}(\dd y,\dd w)$, where $\delta_z$ denotes the Dirac Delta measure and $R_iv=(v_1\ldots,v_{i-1},-v_i,v_{i+1},\ldots,v_d)$.  The ZZS is described by its generator
	\begin{equation}\label{eq:ZZgen}
	\cL f(x,v) = \langle v, \nabla_x f(x)\rangle + \sum_{i=1}^d\lambda_i(x,v)[f(x,R_iv)-f(x,v)]. 
	\end{equation}
	Simulating the event times with rates of this form is in general a very challenging problem as the integral in \eqref{eq:switching_time_pdmp} cannot be computed for general potentials $\pot$. 
	
	We can apply Algorithm \ref{alg:basic_EPDMP} to the ZZS to obtain $(\overline{X}_{t_{n+1}},\overline{V}_{t_{n+1}})$ given the previous state by first simulating the next switching time $\bar{\tau}$ with rate $\overline{\lambda}((\Xbar_{t_n},\Vbar_{t_n}),s;\delta_{n+1})$ and then
	\begin{align*}
	\overline{X}_{t_{n+1}}&\coloneqq\overline{X}_{t_n}+\overline{V}_{t_n}\delta_{n+1} \\
	\overline{V}_{t_{n+1}}&\coloneqq\begin{cases} R_{\bar{I}_{n+1}}\overline{V}_{t_n} & \text{ if } \bar{\tau}\leq \delta_{n+1},\\
	\overline{V}_{t_n} & \text{ if } \bar{\tau}> \delta_{n+1},
	\end{cases}
	\end{align*} 
	where $\overline{\lambda}(z,s;\delta)=\sum_{i=1}^d \overline{\lambda}_{i}(z,s;\delta)$, and $$\bar{I}_{n+1}\sim \multi\left(\left\{\frac{\lambdabar_{i}((\Xbar_{t_n},\Vbar_{t_n}),\bar{\tau};\delta_{n+1})}{\lambdabar((\Xbar_{t_n},\Vbar_{t_n}),\bar{\tau};\delta_{n+1})}\right\}_{i=1}^d\right).$$ 
	 The only approximation concerns the switching rates, whereas it is straightforward to simulate the linear dynamics and the jumps at event times. As mentioned above, a simple choice is to take $\overline{\lambda}_{i}((x,v),s;\delta) = \lambda_i(x,v)$, which results in an Euler approximation of the ZZS.  An alternative choice is
    \begin{equation}\label{eq:lambda_gradfree_zzs}
        \overline{\lambda}_{i} ((x,v),s;\delta) = \frac{1}{\delta}\left( \pot(x+ v_i e_i \delta)-\pot(x) \right)_++\gamma_i(x,v),
    \end{equation}
    which is obtained by a finite difference scheme approximation for $\partial_i\pot$. Here $e_i$ is the $i$-th vector of the canonical basis. Observe that with this choice of $\overline{\lambda}_{i}$ the approximation is gradient free, as it does not require computing $\nabla \pot$. An approximation given by Algorithm \ref{alg:advanced_EPDMP} may be introduced analogously.
\end{example}

\begin{example}[Bouncy Particle Sampler \cite{BPS,PetersDeWith2012}]\label{ex:BPSintro}
	Let $E=\R^d\times\R^d$, and for any $z\in E$ we write $z=(x,v)$ for $x\in \R^d$, $v\in\R^d$. Set $\Phi(x,v)=(v,0)^T$, $m=2$, 
	$$Q_1((x,v),(\dd y,\dd w)) = \delta_{(x,R(x)v)}(y,w), \quad Q_2((x,v),(\dd y,\dd w)) =  \delta_x(\dd y)\nu(\dd w),$$ 
	with $\lambda_1(x,v) =(v^T\nabla_x\pot(x))_+$, $\lambda_2= \lambda_{r}$ for $\lambda_{r}>0$,  $\nu$ is the Gaussian measure, and finally 
	$$R(x)v=v-2\frac{\langle v,\nabla_x\pot(x)\rangle}{\lVert\nabla_x\pot(x)\rVert^2} \nabla_x\pot(x).$$
	In this example $\lVert\cdot\rVert$ denotes the Euclidean norm.
	The BPS has generator 
	\begin{equation}\notag
	\cL f(x,v)\! =\! \langle v, \nabla_x f(x)\rangle + \lambda_1(x,v)[f(x,R(x)v)-f(x,v)]+\lambda_2 \!\int\!\! \big(f(x,w) - f(x,v)\big) \nu(\dd w).
	\end{equation}
	For the same reasons as for the ZZS, simulating the event times can be very challenging for the BPS.
	
	For this process we introduce an approximation based on Algorithm \ref{alg:advanced_EPDMP}. Let $U_{n+1}=(\mathscr{Z}_{n+1},\mathscr{U}_{n+1})$ for $\mathscr{Z}_{n+1}$ distributed according to the standard Gaussian distribution $\nu$ and $\mathscr{U}_{n+1}\sim\textnormal{Unif}([0,1])$ is an independent uniform random variable. For $n\geq 0$ we define the next state of the approximation $(\overline{X}_{t_{n+1}},\overline{V}_{t_{n+1}})$ given the previous state by first simulating $\bar{\tau}$ with distribution $\mathbb{P}_z(\bar{\tau}> t) =\exp(-\int_0^t \overline{\lambda}((\Xbar_{t_n},\Vbar_{t_n}),s;\delta_{n+1})\dd s)$ and then 
	\begin{align*}
	\overline{X}_{t_{n+1}}&\coloneqq\overline{X}_{t_n}+\bar{\tau}\overline{V}_{t_n} + (\delta_{n+1}-\bar{\tau}) \overline{V}_{t_{n+1}}, \\
	\overline{V}_{t_{n+1}}&\coloneqq\begin{cases} F((\overline{X}_{t_{n} + \bar{\tau} },\overline{V}_{t_n}),U_{n+1}) & \text{ if } \bar{\tau} \leq\delta_{n+1},\\
	\overline{V}_{t_n} & \text{ if } \bar{\tau} >\delta_{n+1}.
	\end{cases}
	\end{align*}
	Here $\overline{\lambda}((x,v),t;\delta_{n+1}) = \overline{\lambda}_{1}((x,v),t;\delta_{n+1})+\lambda_r$ where $\overline{\lambda}_{1}((x,v),t;\delta_{n+1})$ approximates $\lambda_1(x+vt,v)$ and
	\begin{equation*}
	    F((\overline{X}_{t_{n} + \overline{\tau} },\overline{V}_{t_n}),U_{n+1}) = \begin{cases}
	        R(\overline{X}_{t_{n}+\overline{\tau}})\overline{V}_{t_n}   & \text{ if } \mathscr{U}_{n+1} > \frac{\lambda_r}{\overline{\lambda}((\overline{X}_{t_n},\overline{V}_{t_n}),\bar{\tau};\delta_{n+1})}, \\
	        \mathscr{Z}_{n+1}  & \text{ if } \mathscr{U}_{n+1} \leq \frac{\lambda_r}{\overline{\lambda}((\overline{X}_{t_n},\overline{V}_{t_n}),\bar{\tau};\delta_{n+1})}.
	    \end{cases}
	\end{equation*}
	It is thus clear that applying Algorithm \ref{alg:advanced_EPDMP} rather than Algorithm \ref{alg:basic_EPDMP} can be more computationally expensive in the case of BPS, as when an event takes place $\nabla\pot$ has to be evaluated at some midpoint $\overline{X}_{t_{n}+\bar{\tau}}$ in order to compute the reflection operator. In contrast, $\nabla\pot$ has to be computed only at gridpoints in Algorithm \ref{alg:basic_EPDMP}. We shall see in Section \ref{sec:main_results} that our theoretical results can only be applied to approximations of the BPS based on Algorithm \ref{alg:advanced_EPDMP}, motivating the need for that algorithm.

	Similarly to the case of the ZZS described in Example \ref{ex:ZZ}, possible approximations of $\lambda_1(x+vt,v)$ are $\overline{\lambda}_{1}((x,v),t;\delta) = \lambda(x,v)$ or 
	\begin{equation}\notag
        \overline{\lambda}_1 ((x,v),t;\delta) = \frac{1}{\delta}\left( \pot(x+ v \delta)-\pot(x) \right)_+.
    \end{equation}
    The latter choice is not enough to not make the simulation of $(\Xbar_t,\Vbar_t)$ gradient free because $\nabla \pot$ is needed in the computation of the reflection operator.
\end{example}

\begin{example}[Randomized Hamiltonian Monte Carlo algorithm]\label{ex:randHMC}
The randomized Hamiltonian Monte Carlo algorithm (see \cite{bou2017randomized}) is defined on $E=\R^d\times\R^d$ by the generator
\begin{equation*}
    \cL f(q,p) = \langle p,\nabla_q f(q,p)\rangle - \langle \nabla_q \pot(q),\nabla_p f(q,p) \rangle + \lambda_r \int \big(f(q,p') - f(q,p)\big) \nu(\dd p'),
\end{equation*}
where $\nu$ is a Gaussian measure on $\R^d$. The Hamiltonian flow cannot be simulated exactly in most cases, and thus it becomes necessary to approximate it by a numerical integrator $\phibar_s$. Then  according to Algorithm \ref{alg:basic_EPDMP} we obtain the next state by first denoting $(\tilde{Q}_{t_{n+1}},\tilde{P}_{t_{n+1}})=\phibar_{\delta_{n+1}}(\Zbar_{t_{n+1}};\delta_{n+1})$ and thus
\begin{align*}
    (\Qbar_{t_{n+1}},\Pbar_{t_{n+1}}) =  \begin{cases}
        (\tilde{Q}_{t_{n+1}},\tilde{P}_{t_{n+1}}) \quad  &\textnormal{ with probability } \exp(-\lambda_r\delta_{n+1}),\\
        (\tilde{Q}_{t_{n+1}},\mathcal{Z}) \quad  &\textnormal{ with probability } 1-\exp(-\lambda_r\delta_{n+1}),
    \end{cases}
\end{align*}
where $\mathscr{Z}\sim \nu$. We remark that the most efficient implementation is to simulate the next refreshment time and then follow the numerical integrator until then, without drawing a new refreshment time at each iteration.
\end{example}

\begin{example}[Modelling the size of a cell]
Following Section 1.5 in \cite{rudnicki2017piecewise}, denote the size of a cell by $z\in \mathbb{R}$. The cell grows in time with deterministic flow $\varphi_t$, and splits into two daughter cells with division rate $\lambda(z)$. Then denote as $\tau_n$ the time when a cell from the $n$-generation splits. The size of a daughter cell is half of the parent cell, and thus $Z_{\tau_n} = \frac{1}{2}Z_{\tau_n-}$. We can characterise the resulting process with its generator:
\begin{equation*}
    \cL f(z) = \langle \vf(z),\nabla f(z)\rangle + \lambda(z) \left(f\left(\frac{z}{2}\right)-f(z)\right).
\end{equation*}
Therefore it may not be possible to simulate such a process if the desired $\varphi$ and $\lambda$ are complicated functions. An approximation can be obtained applying the ideas above introducing a numerical integrator $\phibar$ and approximate division rate $\lambdabar$.
\end{example}

\begin{example}[Chemotaxis in Escherichia coli]
 It was shown in \cite{berg1972chemotaxis} that the bacteria Escheria coli have two types of behaviour describing their motion, which are called ``runs'' and ``twiddles''. When the bacteria is ``running" it moves with near uniform speed. However when ``twiddling'' the bacteria changes direction very abruptly. We will describe this using the stochastic model as given in \cite{stroock1974some}.  We describe the bacteria by giving its position $x\in \R^3$ and velocity $v\in \mathbb{S}^2$ at each time, where $\mathbb{S}^2$ is the sphere in $\R^3$. Then there exists a function $\lambda:[0,\infty)\times \R^3\times \mathbb{S}^2\to (0,\infty)$ which describes the next time the bacteria twiddles; at such a twiddle the velocity changes according to some probability measure $\mu_v$ on $\mathbb{S}^2\setminus \{v\}$ where $v$ is the velocity before the twiddle. The dynamics of the bacteria are given as a PDMP described by the backward equation
 \begin{equation*}
     \frac{\partial u}{\partial t}(t,x,v) + \langle v, \nabla_x u(t,x,v)\rangle +\lambda(t,x,v)\int_{\mathbb{S}^2} [u(t,x,\eta)-u(t,x,v)]\mu_v(\dd \eta)=0.
 \end{equation*}
 Note if $\lambda$ is independent of $t$ then we can describe this process by writing a generator in the form \eqref{eq:genPDMPmanykernels}; otherwise we can extend the space to include a time variable and then write a corresponding generator in the form \eqref{eq:genPDMPmanykernels} which is given by
 \begin{equation*}
     \cL f(t,x,v) = \partial_t f(t,x,v) +  \langle v, \nabla_x f(t,x,v)\rangle +\lambda(t,x,v)\int_{\mathbb{S}^2} [f(t,x,\eta)-f(t,x,v)]\mu_v(\dd \eta).
 \end{equation*}
 We can introduce an approximation of this process by using frozen switching rates.
\end{example}

\subsection{Higher order schemes}\label{sec:higherorder}
A natural question is how to obtain higher order schemes. 
The first important observation is that that the probability that a PDMP has more than one jump in a time interval of length $\delta$ is of order $\delta^2$. Therefore in order to construct higher order schemes it is natural to allow multiple jumps in the same time step. 
\begin{algorithm}[t]
\SetAlgoLined
\SetKwInOut{Input}{Input}\SetKwInOut{Output}{Output}
\review{\Input{Number of iterations $N$, initial condition $z$, step sizes $(\delta_n)_{n=0}^N$.}
\Output{Chain $(\Zbar_{t_{n}})_{n=0}^N$.}
 Set $n=0$, $\Zbar_0 = z$\;
 \While{$n < N$}{
 set $\tilde{Z}=\Zbar_{t_n}$\;
  draw $U^1 \sim \text{Unif}[0,1]$ and simulate $$\bar{\tau}_1 = \inf\left\{ r>0: 1-\exp\left(-\int_0^r \lambdabar(\tilde{Z},s;\delta_{n+1},2)\dd s\right) \geq U^1 \right\} \;$$
  \eIf{$\bar{\tau}_1 < \delta_{n+1}$}{
   draw $U_{n+1}^1 \sim \nuU$ and $\bar{I}_1\sim \multi\Big(\Big\{\frac{\lambdabar_i(\tilde{Z},\bar{\tau}_1;\,\delta_{n+1},2)}{\lambdabar(\tilde{Z},\bar{\tau}_1;\,\delta_{n+1},2)}\Big\}_{i=1}^m\Big)$\;
   set $\tilde{Z} = \phibar_{\bar{\tau}_1}(\tilde{Z};\delta_{n+1},2)$\;
   set $\tilde{Z} = \overline{F}_{\Bar{I}_1}(\tilde{Z},U^1_{n+1};\delta_{n+1},2)$\;
  draw $U^2\sim \text{Unif}[0,1]$ and simulate $$\bar{\tau}_2 = \inf\left\{ r>0: 1-\exp\left(-\int_0^r \lambdabar(\tilde{Z},s;\delta_{n+1},1)\dd s\right) \geq U^2 \right\}\; $$
  \eIf{$\bar{\tau}_2 < t_{n+1}-\overline{\tau}_1$}{
  draw $U^2_{n+1} \sim \nuU$ and $\bar{I}_2\sim \multi\Big(\Big\{\frac{\lambdabar_i(\Tilde{Z},\Bar{\tau}_2;\,\delta_{n+1},1)}{\lambdabar(\tilde{Z},\Bar{\tau}_2;\,\delta_{n+1},1)}\Big\}_{i=1}^m\Big)$\;
  set $\tilde{Z} = \overline{F}_{\Bar{I}_2}(\tilde{Z},U^2_{n+1};\delta_{n+1},1)$\;
  simulate $\Zbar_{t_{n+1}} = \phibar_{t_{n+1}-\overline{\tau}_2-\overline{\tau}_1}(\Tilde{Z};\delta_{n+1},1)$\;
  }{
  simulate $\Zbar_{t_{n+1}} = \phibar_{t_{n+1}-\overline{\tau}_1}(\Tilde{Z};\delta_{n+1},1)$\;
  }
   }{
   set $\Zbar_{t_{n+1}} = \phibar(\tilde{Z};\delta_{n+1},2)$\;
   }
 }}
 \caption{Second order Partially Discrete Approximation of a PDMP}
 \label{alg:trulysecondorder_EPDMP}
\end{algorithm}
A detailed implementation of a higher order approximation scheme can be found in Algorithm \ref{alg:secondorder_EPDMP}.  Let us first describe a second order algorithm. Starting at state $\Zbar_{t_n}=z$, the proposed time for the first event is given by $\bar{\tau}$ where $$\mathbb{P}(\bar{\tau}>r)=\exp\left(-\int_0^r \overline{\lambda}(z,s;\delta_{n+1},2)\dd s \right).$$ If $\bar{\tau}  <\delta_{n+1}$, the process moves according to the numerical flow $\phibar_{s}(z;\delta_{n+1},2)$ for time $\bar{\tau}$, and at time $t_n+\bar{\tau} $  the random event takes place according to $\Fbar_{\Bar{I}}(\Zbar_{t_n+\bar{\tau}},\cdot;\delta_{n+1},2)$, where $\bar{I}$ has discrete distribution. In this case, a second jump is allowed in the current time step. The simulation of this event can be made using first order approximations $\lambdabar(\cdot,\cdot\,;\delta_{n+1},1), \phibar_s(\cdot\,;\delta_{n+1},1)$, and $\Fbar_i(\cdot,\cdot\,;\delta_{n+1},1)$.

\begin{algorithm}[t]
\SetAlgoLined
\SetKwInOut{Input}{Input}\SetKwInOut{Output}{Output}
\Input{Number of iterations $N$, initial condition $z$, step sizes $(\delta_n)_{n=0}^N$.}
\Output{Chain $(\Zbar_{t_{n}})_{n=0}^N$.}
 Set $n=0$, $\Zbar_0 = z$\;
 \While{$n < N$}{
 set $q=p$, $\tilde{Z}=\Zbar_{t_n}$\;
 set $t_{\mathrm{left}}=\delta_{n+1}$\;
 \While{$q >0$}{
  simulate $$\bar{\tau} = \inf\left\{ r>0: 1-\exp\left(-\int_0^r \lambdabar(\tilde{Z},s;\delta_{n+1},q)\dd s\right) \geq U \right\} $$ where $U \sim \text{Unif}[0,1]$ \;
  \eIf{$\bar{\tau} < t_{\mathrm{left}}$}{
   draw $U_{n+1} \sim \nuU$ and $\bar{I}\sim \multi\Big(\Big\{\frac{\lambdabar_i(\tilde{Z},\bar{\tau};\,\delta_{n+1},q)}{\lambdabar(\tilde{Z},\bar{\tau};\,\delta_{n+1},q)}\Big\}_{i=1}^m\Big)$\;
   set $\tilde{Z} = \phibar_{\bar{\tau}}(\tilde{Z};\delta_{n+1},q)$\;
   set $\tilde{Z} = \overline{F}_{\Bar{I}}(\tilde{Z},U_{n+1};\delta_{n+1},q)$\;
   set $q=q-1$ and $t_{\mathrm{left}}=t_{\mathrm{left}}-\bar{\tau}$\; 
   
   }{
   set $\tilde{Z} = \phibar_{\delta_{n+1}}(\tilde{Z};\delta_{n+1},q)$\;
   set $q=0$\;
   
   }
  }
  set $\Zbar_{n+1}=\tilde{Z}$, $n=n+1$\;
 }
 \caption{Order $p$ Partially Discrete Approximation of a PDMP}
 \label{alg:secondorder_EPDMP}
\end{algorithm}

Let us consider as an example how to obtain a second order approximation for smooth switching rates. For $s\leq\delta$ the first order Taylor approximation of $\lambda_i(\varphi_s(z))$ is given by
$$\lambdabar_i(z,s;\delta_{n+1},2) = \lambda_i(z) + s \langle \Phi(z),\nabla\lambda_i(z)\rangle.
$$
Because the integral in \eqref{eq:switching_time_pdmp} is with respect to $s$, this choice of $\lambdabar_i(z,s;\delta_{n+1},2)$ is such that computing the corresponding switching time is equivalent to computing the root of a second order polynomial. The downside is that an evaluation of the gradient of $\lambda_i$ is needed and may be unavailable or expensive to compute. However, we can further approximate the product $\langle \Phi(z),\nabla\lambda_i(z)\rangle$ with a finite difference scheme to obtain for $s\leq \delta_{n+1}$ the expression
\begin{equation}\label{eq:lambda_approx_2ndorder_gradfree}
    \lambdabar_i(z,s;\delta_{n+1},2) = \lambda_i(z) + \frac{s}{\delta_{n+1}} (\lambda_i(\varphi_{\delta_{n+1}}(z))-\lambda_i(z)),
\end{equation}
which is a second order approximation provided $\lambda$ is sufficiently smooth. \review{The algorithm for $p=2$ is given by Algorithm~\ref{alg:trulysecondorder_EPDMP}.}

Similarly, it is possible to obtain an order $p>2$ approximation. The simulation up to and counting the first event of each time step should be made according to approximations of order $\delta^p$ of the flow map, switching rates, and jump kernels. After the first event it is then possible to use approximations of order $p-1$, then of order $p-2$, and so on until one reaches the end of the current time interval, with the constraint that at most $p$ events take place. Finally, it is clearly possible to use approximations of order $\delta^p$ for the simulation of all events in the same time step, although such approximations can be in general more expensive to compute.

\section{Main results}\label{sec:main_results}
\subsection{Error bounds in Wasserstein distance}\label{sec:Wass_bounds}
The main result of this section is Theorem~\ref{thm:strong_error_pdmp}, which shows convergence of the Wasserstein distance between the approximation and the continuous process as the step size goes to $0$. We consider the Wasserstein distance of order $1$ with respect to any normed distance, that is we take $d(x,y)=\lVert x-y\rVert$ in  Equation~\eqref{eq:wasserstein} for any vector norm $\lVert \cdot\rVert$.
For convenience we assume that for all $n\in\N$ we have an upper bound $\delta_n\leq \delta_0$. 

Let us now state the assumptions on the process and on the various approximations that are required to show Theorem~\ref{thm:strong_error_pdmp}. We start with assumptions on the continuous time PDMP, and specifically from a condition on the deterministic dynamics. In particular, we require that $\vf$ is Lipschitz.
\begin{assumption}\label{ass:lipschitz_phi}
	For the vector field $\vf$ there exists a constant $C>0$ such that for all $z,z'\in E$ it holds that 
	\begin{equation*}
	\lVert \vf(z)-\vf(z')\rVert\leq C\lVert z-z'\rVert.
	\end{equation*}
\end{assumption}

We now shift our focus to the jump part of the process. In particular, we need the kernel $Q(z,\cdot)$ to satisfy the next conditions. 
\begin{assumption}\label{ass:F_new}
	There exist constants $D_1,D_2,D_3>0$ such that for $\tilde{U}\sim \nuU$ the following conditions hold for all $i\in \{1,\ldots,m\}$:
	\begin{enumerate}[label=(\alph*)]
	    \item For any $z\in E$
	        \begin{equation}\notag
	             \mathbb{E} [\lVert z - F_i(z,\tilde{U})\rVert] \leq D_1.
	        \end{equation}
	   \item For all $z,z'\in E$
	        \begin{equation}\notag
	            \mathbb{E} [\lVert F_i(z,\tilde{U})-F_i(z',\tilde{U})\rVert] \leq D_2\lVert z-z'\rVert.
	        \end{equation}
	   \item For all $z\in E$ and all $s\leq \delta\leq \delta_0$
	        \begin{equation}\notag
	            \mathbb{E} \left[\lVert \varphi_{\delta-s}(F_i(\varphi_s(z),\tilde{U})) - F_i(\varphi_\delta(z),\tilde{U}) \rVert \right]  \leq D_3 \delta.
	        \end{equation}
	\end{enumerate}
\end{assumption}
\noindent The first assumption asks that after a random jump the process is in expectation at bounded distance to its previous state, while condition (b) states that a Lipschitz condition with respect to the previous state holds for coupled jumps. Finally, condition (c) asks that the error committed by switching at the end of the time step or at an earlier time is of order $\delta$ if the two jumps are coupled.
Moreover, the following Lipschitz condition for the switching rates is required.
\begin{assumption}\label{ass:lambda_lipschitz}
	There exists $D_4>0$ such that for all $z,z'\in E$  and $i=1,\dots,m$
	\begin{equation}\notag
	    \lvert \lambda_i(z)-\lambda_i(z')\rvert \leq D_4\lVert z-z' \rVert .
	\end{equation}
\end{assumption}

Let us now focus on the required assumptions on the various approximations employed in the approximation process. We state the assumptions for a general order of accuracy $p\geq 1$, with $p\in\N$. Starting from the deterministic dynamics, we assume that the numerical integrator for the flow map is an approximation of order $p$. 
\begin{assumption}\label{ass:integrator}
    There exists $\Tilde{C}\geq 0$ such that for any $z\in E$ and any $0\leq s\leq \delta\leq \delta_0$ 
    \begin{equation}\notag
        \lVert \varphi_s(z)-\phibar_s(z;\delta,p)\rVert \leq \tilde{C}s^{p+1}.
    \end{equation}
\end{assumption}

\noindent In case the flow map can be simulated exactly, one can simply take $\phibar_s=\varphi_s$ and $\tilde{C}=0$. Next we focus on the approximate jump kernels $\Fbar_i$.
\begin{assumption}\label{ass:approx_jump_kernel}
The approximate jump kernels $\Fbar_i:E\times \cU\times [0,\delta_0] \to E$, satisfy for any $z \in E$ and $\delta\in(0,\delta_0]$
\begin{equation}\notag
    \mathbb{E}_z[\lVert \Fbar_i(z,\tilde{U};\delta,p) - F_i(z,\tilde{U})\rVert] \leq M_1 \delta^p
\end{equation}
for all $i=1,\dots,m$. 
\end{assumption}
\noindent Let us now state the requirement on the approximate switching rates $\lambdabar_i$.
\begin{assumption}\label{ass:lambda_wasserstein}
The following conditions hold:
    \begin{enumerate}[label=(\alph*)]
        \item 
        There exists $\overline{M}_2(z)$ such that for all $0\leq s\leq \delta\leq \delta_0$ and $i\in\{1,\dots,m\}$ 
        \begin{equation}\notag
            \lvert \lambdabar_i(z,s;\delta,p)-\lambda_i(\varphi_s(z))\rvert \leq \delta^p \overline{M}_2(z).
        \end{equation}
        \item For any $n\in\N$ there is a function $M_2(t,z)$ such that
        \begin{equation}\notag
            \mathbb{E}_z\left[\overline{M}_2(\Zbar_{t_n})\right]\leq M_2(t_n,z) <\infty.
        \end{equation}
    \end{enumerate}
\end{assumption}

\noindent As a final assumption, we require that both the continuous time PDMP and the approximation process have almost surely bounded norm for a finite time horizon. This assumption is verified for instance if the state space is compact, or if the processes travel with bounded velocity.
\begin{assumption}\label{ass:boundedPDMP}
	For any $t>0$ there exists $B(t,z)>0$ such that almost surely both $\lVert Z_t \rVert \leq B(t,z)$  and $\lVert \Zbar_t\rVert \leq B(t,z)$, where $Z_0=\Zbar_0 = z$.
\end{assumption}

\begin{remark}\label{note:wass_theorem}
Let us comment on these assumptions:
\begin{itemize}
    \item It is worth observing that conditions such as Assumption~\ref{ass:lambda_lipschitz} can be weakened to forms such as
$$\lvert \lambda_i(z)-\lambda_i(z')\rvert \leq D_4\lVert z-z' \rVert (1 +\lVert z\rVert^q + \lVert z'\rVert^{q'}),$$ for some $q,q'\in \N$. This is because by Assumption~\ref{ass:boundedPDMP} the norms at time $t$ of the two processes are bounded almost surely and therefore for some $M(t)$ we have $$(1 +\lVert Z_t\rVert^q + \lVert \Zbar_t\rVert^{q'}) \leq M(t) < \infty$$ almost surely. A similar reasoning can be applied to other assumptions that have this structure. For simplicity we will not consider this set of weakened assumptions in the proof of Theorem~\ref{thm:strong_error_pdmp}, but we remark that the extension is straightforward.
\item In both Example \ref{ex:ZZ} on the ZZS and Example \ref{ex:BPSintro} on the BPS we can write $\lambda$ of the form 
$$
\lambda(x,v) = f(r)
$$
where $r=\partial_i\pot(x)v_i$ for ZZS or $r=\langle \nabla\pot(x),v\rangle$ for BPS and $f(r)=r_+$. Note that it is possible to take a smooth function $f$ for which the process still has the desired invariant measure, see \cite{AndrieuLivingstone}. We will demonstrate some choices of $\lambdabar$ for ZZS which satisfy Assumption~\ref{ass:lambda_wasserstein}, and analogous choices hold for BPS. For smooth $\lambda$ we can use \eqref{eq:lambda_approx_2ndorder_gradfree} to obtain a second order approximation or similarly a $p$-th order finite difference scheme to have an order $p$ approximation. However if $\lambda(x,v)=(v_i\partial_i\pot(x))_+$ is only Lipschitz then this approximation is no longer valid; instead we can write
\begin{equation*}
    \lambdabar_i((x,v),s;\delta,p)= (\overline{\partial_i\pot}((x,v),s;\delta,p)\,v_i)_+
\end{equation*}
where $\overline{\partial_i\pot}((x,v),s;\delta,p)$ is a $p$-th order approximation in $s$ of $\partial_i\pot(x+sv)$ and can be obtained either by a truncated Taylor expansion or using a finite difference scheme. Then using that $(\cdot)_+$ is $1$-Lipschitz
\begin{equation*}
    \lvert \lambdabar_i((x,v),s;\delta,p)-\lambda_i(\varphi_s(x,v))\rvert \leq \lvert \overline{\partial_i\pot}((x,v),s;\delta,p) - \partial_i\pot(x+sv)\rvert \leq \overline{M}_2 \delta^p.
\end{equation*}
For example, for $\pot$ sufficiently smooth, we can take
\begin{equation*}
    \lambdabar_i((x,v),s;\delta,p) =\left(\sum_{q=0}^{p-1} \frac{(sv_i)^q}{q!}\Delta_{i,\delta,p-q}^{q+1} \pot(x)\right)_+,
\end{equation*}
where $\Delta_{\delta,p-q}^{q}\pot$ denotes the $\delta^{p-q}$-th order approximation of the $q$-th derivative of $\pot$ in the variable $x_i$.  
\end{itemize}

\end{remark}

We are ready to state the main result of this section. 
\begin{theorem}\label{thm:strong_error_pdmp}
	Let $p\geq 1$. Denote by $\{\cP_t\}_{t\geq 0}$ the semigroup of a PDMP with generator \eqref{eq:genPDMPmanykernels}, which satisfies Assumptions \ref{ass:lipschitz_phi}-\ref{ass:lambda_lipschitz}. Denote by $\overline{\cP}_t$ the transition probability of the Markov chain described by either Algorithms \ref{alg:basic_EPDMP} or \ref{alg:advanced_EPDMP} in the case $p=1$, or by Algorithm \ref{alg:secondorder_EPDMP} for $p>1$. Suppose that $\phibar_t(\cdot;\delta,q),\,\lambdabar(\cdot,\cdot;\delta,q),\, \Fbar_i(\cdot;\delta,q)$ satisfy Assumptions \ref{ass:integrator}-\ref{ass:boundedPDMP} for some $\delta_0>0$ and for every $1\leq q \leq p$ with $q\in\N$. Then for a fixed $T>0$  there exist $K_1=K_1(T)$, $K_2=K_2(T)$ such that for any mesh $0=t_0< t_1< \ldots < t_N=T$ with $\delta_n=t_n-t_{n-1}$ and $\delta_n\leq \delta_0$ for any $n\leq N$
	\begin{equation*}
	    \mathcal{W}_1(\cP_{T}(z,\cdot),\overline{\cP}_{T}(z,\cdot))
	    \leq K_2 \sum_{k=1}^{N} \delta^{p+1}_k \left(\prod_{\ell=k}^{N} (1+\delta_\ell K_1)\right).
	\end{equation*}
	If the step size is uniform, i.e. $\delta_n=\delta$ and $t_{n}=n\delta$, then 
	\begin{equation*}
	    \mathcal{W}_1(\cP_{T}(z,\cdot),\overline{\cP}_{T}(z,\cdot)) \leq \delta^p \left( e^{T K_1} -1 \right) \frac{K_2}{K_1}.
	\end{equation*}
\end{theorem}

\begin{proof}[Proof of Theorem~\ref{thm:strong_error_pdmp}]
The proof of Theorem~\ref{thm:strong_error_pdmp} can be found in Section \ref{sec:proof_wass_theorem}.
\end{proof}

We now give a setting in which Assumption~\ref{ass:F_new} simplifies. This is motivated by and includes the ZZS.
Let us now consider a PDMP $Z_t=(X_t,V_t)\in \mathbb{R}^n\times \mathcal{V}$, where $X_t$ and $V_t$ should be interpreted as the position and velocity at time $t$. Here $\mathcal{V}$ is some subset of Euclidean space.
Consider the case in which the deterministic dynamics with initial condition $(x,v)$ are of the form
\begin{equation}\notag
    \begin{cases}
        \dot{x} = \vf(v), \\
        \dot{v} = 0.
    \end{cases}
\end{equation}
Therefore the deterministic motion is $X_t=\varphi_t(x,v)$ and $V_t=v$ if $(X_0,V_0)=(x,v)$. Then assume that the random events affect only the velocity, and leave the position unchanged, i.e.  $F_i((x,v),U) = (x,F_i^v((x,v),U))$. This is the setting for example of the ZZS and BPS.
Consider the following assumption.
\begin{assumption}\label{ass:discr_velocities_wass}
The space $\mathcal{V}$ is such that for all $v,w\in \mathcal{V}$ with $v\neq w$ it holds that 
\begin{equation*}
    0<V_{min} \leq \lVert v-w\rVert \leq V_{max}<\infty.
\end{equation*}
Assume also that there exists $D>0$ such that for any $x,y \in \mathbb{R}^n$, $i\in \{1,\ldots,m\}$ and $v\in \mathcal{V}$
\begin{equation*}
    \mathbb{E}_{(x,v)}[\lVert F_i^v((x,v),U)-F_i^v((y,v),U)\rVert ] \leq D \lVert x-y\rVert.
\end{equation*}
\end{assumption}
The next corollary states that in this setting Assumption~\ref{ass:discr_velocities_wass} implies Assumption~\ref{ass:F_new}.
\begin{corollary}\label{cor:strong_error_pdmp}
Consider a PDMP of the particular form described above. Suppose Assumptions \ref{ass:lipschitz_phi}, \ref{ass:lambda_lipschitz}-\ref{ass:lambda_wasserstein}, as well as Assumption~\ref{ass:discr_velocities_wass} hold. Then Theorem~\ref{thm:strong_error_pdmp} applies.
\end{corollary}
\begin{proof}
The proof can be found in Appendix \ref{sec:proof_corollary_wass}.
\end{proof}

Finally, we consider the setting in which we have a deterministic upper bound for the switching rates, but the process is not almost surely bounded as was required by Assumption~\ref{ass:boundedPDMP}. This is the case for instance of the Randomized HMC algorithm \cite{bou2017randomized}. We shall show that in this case Theorem~\ref{thm:strong_error_pdmp} holds as long as for a finite time horizon the processes are bounded in expectation. The formal condition is the following.
\begin{assumption}\label{ass:bdd_lambda}
There exists a constant $\lambda_{max}>0$ such that  $\lambda(z)\leq \lambda_{max}$ for all $z\in E$. Moreover there exists $L(t,z)<\infty$ such that
\[
    \max\{\mathbb{E}_z[\lVert Z_t\rVert],\mathbb{E}_z[\lVert \Zbar_t\rVert] \} \leq B(t,z).
\]
\end{assumption}

\begin{prop}\label{prop:wass_rhmc}
Suppose Assumptions \ref{ass:lipschitz_phi}-\ref{ass:lambda_wasserstein} and \ref{ass:bdd_lambda} hold. Then Theorem~\ref{thm:strong_error_pdmp} applies.
\end{prop}

\begin{proof}
The proof is given in Appendix \ref{sec:proof_wass_rhmc}.
\end{proof}

\subsection{Error bounds in total variation distance}\label{sec:tv_distance}
In this section we show that a bound of order $\delta^p$ on the total variation distance between the approximation and the PDMP can be derived for Algorithm \ref{alg:secondorder_EPDMP} assuming it is possible to simulate exactly the flow $\varphi_t$ and the Markov kernels $Q_i$.  Interestingly, this result can be proved under considerably weaker assumptions on the PDMP compared to what is considered in Section \ref{sec:Wass_bounds}. We remark in particular that no assumption on the maps $F_i$ is needed, which was the case in Assumption~\ref{ass:F_new}. Moreover the process needs not be bounded almost surely for finite time horizons, as described in Assumption~\ref{ass:boundedPDMP}. The main result of this section is proved by coupling the event times of the PDMP and of the approximations via Poisson thinning. It follows that with a positive probability the processes, which are initialised at the same point, will remain together during a time step. 

Let us state the required assumptions on the switching rates and on the continuous time process. \review{Recall that for first order approximations of the characteristics we drop the specific order of accuracy, e.g. for switching rates we have $\lambdabar_i(z,s;\delta)=\lambdabar_i(z,s;\delta,1)$ for $i=1,\dots,m$. We distinguish the assumptions between the setting $p=1$ and $p>1$. In the case $p=1$ we impose the following assumption.} 
\begin{assumption}\label{ass:lambda_tv}
Each of the approximate switching rates $\overline{\lambda}_{i}(\cdot;\delta)$ for $i=1,\dots,m$ satisfies Assumption~\ref{ass:lambda_wasserstein}(a) with $p=1$ for some $\overline{M}_2(z)$. Furthermore for $z\in E$ and $s\geq 0$ define $\overline{\lambda}(z,s;\delta)= \sum_{i=1}^m \overline{\lambda}_{i}(z,s;\delta) $, and $\lambda_{tot}(z,s;\delta) = \lambda(z) + \overline{\lambda}(z,s;\delta)+m$.
Let $T>0$. Then there exist $L_1(T,z)$, $L_2(T,z)$, $L_3(T,z)<\infty$ such that for any mesh $0=t_0<t_1<\dots<t_N=T$ with $t_{k+1}-t_k=\delta_{k+1}$ and $N\in\N$ the following conditions hold:
\begin{align*}
    & \sup_{n\leq N} \sup_{i=1,\dots,m} \sup_{s\in [0,\delta_n]} \sup_{r\in [s,\delta_n]} \mathbb{E}_z \Big[  \lambda(\varphi_s(F_{i}(\varphi_r(Z_{t_{n-1}}),\Tilde{U}_n) \lambda_{tot}(Z_{t_{n-1}},s;\delta)  \Big]\leq L_1(T,z),\\
    & \sup_{n\leq N} \sup_{s\in [0,\delta_n]} \mathbb{E}_z \left[\overline{M}_2(Z_{t_{n-1}}) \lambda_{tot}(Z_{t_{n-1}},s;\delta)  \right] \leq L_2(T,z), \\
    & \sup_{n\leq N} \sup_{s\in [0,\delta_n]} \sup_{r\in [s,\delta_n]} \mathbb{E}_z \Big[  \left(\lambda(\varphi_r(Z_{t_{n-1}}))+ \lambdabar(Z_{t_{n-1}},r;\delta)\right) \lambda_{tot}(Z_{t_{n-1}},s;\delta)  \Big]\leq L_3(T,z).
\end{align*}
%
\end{assumption}

\review{For the case $p>1$ we make the following assumption. Recall in the case $p>1$ if in a single time step there have been $q$ jumps then we use $\lambdabar_i(\cdot;\delta,p-q)$ to simulate the next jump time. As the probability of there having been $q$ jumps in a time interval is order $\delta^q$ the conditions required on $\lambdabar_i(\cdot;\delta,q)$ are lessened, for this reason there are different requirements for each $q$.
\begin{assumption}\label{ass:lambda_tvp>1}
Each of the approximate switching rates $\overline{\lambda}_{i}(\cdot;\delta,q)$ for $i=1,\dots,m$ and $q=1,\dots,p$ satisfies Assumption~\ref{ass:lambda_wasserstein}(a) for some $\overline{M}_2(z)$. When $q=1$ the approximate switching rates $\overline{\lambda}_{i}(\cdot;\delta,1)$ for $i=1,\dots,m$ satisfy Assumption~\ref{ass:lambda_tv}. We make the additional moment bound for any $1\leq q \leq p$ 
\[
\sup_{n\leq N}  \sup_{s\in [0,\delta_n]}\mathbb{E}_z \left[(1+\overline{M}_2(Z_{t_{n-1}})) \lambda_{tot}(Z_{t_{n-1}},s;\delta,q) + \lambda_{tot}(Z_{t_{n-1}},s;\delta,q)^{q+1}   \right] \leq L_4(T,z).
\]
\end{assumption}}
\begin{remark} 
The moment bounds in Assumption~\ref{ass:lambda_tv} are rather technical, but also general. For instance Assumption~\ref{ass:lambda_tv} holds if Assumptions \ref{ass:lambda_wasserstein} and \ref{ass:boundedPDMP} hold, i.e. when the process has bounded norm for any finite time horizon. Furthermore, as Assumption~\ref{ass:lambda_tv} does not depend on moment bounds for the approximate process $\{\Zbar_{t_n}\}_{n\geq 1}$, one can verify Assumption~\ref{ass:lambda_tv} by finding a suitable Lyapunov function for the PDMP. Indeed if there exists a Lyapunov function which bounds the functions appearing in Assumption~\ref{ass:lambda_tv} then Assumption~\ref{ass:lambda_tv} holds with $L_1,L_2,L_3$ independent of $T$.
This is the case for instance of the ZZS and BPS, see Example \ref{ex:bps_tv}.
Alternatively, one can take advantage of Holder's inequality to reduce the problem to bounding polynomial moments of the various quantities. In Section \ref{sec:examples_tv} we show that the assumption holds for several examples. Finally we remark that in Assumption~\ref{ass:lambda_tv} it is possible to substitute $Z_{t_{n-1}}$ with $\Zbar_{t_{n-1}}$ and Theorem~\ref{thm:tv_distance} still holds.
\end{remark}

\begin{theorem}\label{thm:tv_distance}
    Denote as $\overline{\cP}_{t}(z,\cdot)$ the transition probability of the approximation process obtained by Algorithm \ref{alg:advanced_EPDMP} for $p=1$ or by Algorithm \ref{alg:secondorder_EPDMP} for $p>1$. Denote by $\{\cP_t\}_{t\geq 0}$ the semigroup of a PDMP with generator \eqref{eq:genPDMPmanykernels} satisfying Assumption~\ref{ass:lipschitz_phi}. Let $p\geq 1$ and \review{suppose the approximations $\lambdabar_i(z,s;\delta,q)$ for $q\leq p$ satisfy Assumption~\ref{ass:lambda_tv} if $p=1$ or Assumption~\ref{ass:lambda_tvp>1} if $p>1$}. Suppose the mesh $t_n=\sum_{i=1}^n \delta_n$ is such that $\delta_n<\delta_0$ for $\delta_0$ as in Assumption~\ref{ass:lambda_wasserstein}(a).  Suppose that $\phibar_s=\varphi_s$ and $\overline{F}_i=F_i$ for all $i=1,\dots,m$.
    Then for any $z\in E$ and any mesh $0=t_0< t_1< \ldots < t_N=T$ with $\delta_n=t_n-t_{n-1}$ and $\delta_n\leq \delta_0$ for any $n\leq N$
    \begin{equation*}
        \lVert \cP_{T}(z,\cdot)-\overline{\cP}_{T}(z,\cdot) \rVert_{TV} \leq \sum_{i=1}^N \delta_i^{p+1} D(T,z) \prod_{\ell=i+1}^N (1-D(T,z)\delta_\ell),
    \end{equation*}
    where $D(t,z)$ is a non-decreasing function of $t$.
    If $\delta_n = \delta $ for all $n\in\N$ then
    \begin{equation*}
        \lVert \cP_{T}(z,\cdot)-\overline{\cP}_{T}(z,\cdot) \rVert_{TV} \leq 1-e^{-D(T,z) T \delta^p} .
    \end{equation*}
\end{theorem}
\begin{proof}
The proof can be found in Section \ref{sec:proofs_tv}.
\end{proof}
\begin{remark}
    Let us for simplicity consider the constant step size case. If we fix a time horizon $t$, then the theorem shows that $\lVert \cP_{t}(z,\cdot)-\overline{\cP}_{t}(z,\cdot) \rVert_{TV}\to 0 \textnormal{ as } \delta \to 0.$
    \review{On the other hand, the upper bound tends to $1$ as $T\to\infty$ if the step size $\delta$ is fixed.}
    Moreover, because $1-\exp(-D(t_n,z) t_n \delta) \leq D(t_n,z) t_n \delta$ we have $$\lVert \cP_{t_n}(z,\cdot)-\overline{\cP}_{t_n}(z,\cdot) \rVert_{TV} \leq  D(t_n,z)\,t_n\, \delta^p$$ and therefore we have convergence of order $\delta^p$ as $\delta \to 0$. 
\end{remark} 

\review{
\begin{remark}
In a similar fashion to \cite{pdmp_inv_meas}, it is possible to obtain a bound as that in Theorem \ref{thm:tv_distance} also when the jump kernel is approximated. To prove such result it is sufficient to define a coupled jump kernel that keeps the two processes together with strictly positive probability if they are together right before the jump.
\end{remark}}

\subsection{Convergence to the invariant measure}\label{sec:weakerror}

In this section we give conditions for the approximation process $\{\Zbar_{t_n}\}_{n\geq 1}$ to converge to $\mu$, the invariant measure of the PDMP, which we shall assume to exist and be unique. We do this by showing convergence in law to the PDMP uniformly in time and requiring that the PDMP converges to its invariant measure. In the following we consider the case of geometric convergence as it is verified for a range of PDMPs, however convergence with any rate $r(t)$ which is integrable over $[0,\infty)$ is sufficient.

The strategy of this proof is inspired by \cite{CDO}, which uses derivative estimates to obtain uniform in time convergence of an Euler Scheme for an SDE. In that case the authors rely on having exponential decay of the derivatives of the semigroup for the SDE of interest, for which conditions are given in \cite{CrisanOttobre}.

\begin{assumption}\label{ass:geoerg}
Let $\{Z_t\}_{t\geq 0}$ be a PDMP with corresponding generator \eqref{eq:genPDMPonekernel}. Recall the definition of $Q$ given by \eqref{eq:defoofQandlambda}. We assume the following:
\begin{enumerate}[label=(\alph*)]
    \item There exists an invariant measure, $\mu$, for the PDMP, $\{Z_t\}_{t\geq 0}$, and $\mu$ is invariant under $Q$, that is \begin{equation*}
	    \mu(Qf)=\mu(f)
	\end{equation*}
	for any $f$ measurable and integrable.
	\item The Markov process $\{Z_t\}_{t\geq 0}$ is geometrically ergodic with invariant measure $\mu$. Specifically fix $\overline{G}:E\to [1,\infty)$ and define $\mathcal{G}=\{\text{measurable } g:E\to\R, \lvert g\rvert\leq \lf\}$. Assume that $\lf(Z_t)$ is integrable for all $t\geq 0$. For some $R_1>0$, $\omega>0$ 
	\begin{equation}\label{eq:geoerg}
	\sup_{g\in \mathcal{G}} \lvert \mathbb{E}_z[g(Z_t)]-\mu(g)\rvert \leq R_1e^{-\omega t} \lf(z).
	\end{equation}
\end{enumerate}
\end{assumption}
This Assumption has been shown in a variety of cases, for example for the $1$-dimensional Zig-Zag process this was shown in \cite[Theorem 5]{BierkensRoberts} and for higher dimensions in  \cite[Theorem 2]{Bierkensergodicity}. For BPS this was shown in \cite{BPS_Durmus, BPSexp_erg}. For RHMC see \cite[Theorem 3.9]{bou2017randomized}. 



The following assumption is required for Algorithm \ref{alg:basic_EPDMP}, but not for Algorithm \ref{alg:advanced_EPDMP}, for the reasons explained in Note \ref{note:derivativeest}. In general, derivative estimates on the semigroup are useful for proving convergence of approximations as they control the effect of a small error in the initial condition of a stochastic process. In this case we are not using explicitly a derivative estimate but instead the operator $[\vfd,Q]$. The role of this commutator is to describe the difference in the direction of the process over an infinitesimal time interval if the process first jumps then follows the flow map or first follows the flow map and then jumps.

\begin{assumption}\label{ass:derivativedeccommutatoralongflow} Let $\{\cP_t\}_{t\geq 0}$ denote the semigroup corresponding to the generator $\cL$ given by \eqref{eq:genPDMPonekernel}. Recall the notation $ [\vfd, Q]$ defined in Section \ref{sec:notation}.  
Let $\lf$ and $\mathcal{G}$ be given as in Assumption~\ref{ass:geoerg}. There exist some $R_2>0$, $\omega>0$ and set $\mathcal{G}_1\subseteq\mathcal{G}$, such that for all $t\geq 0$ we have
\begin{align}\notag
    \sup_{g\in \mathcal{G}_1}\sup_{\delta\in (0,\delta_0), s\in[0,\delta]} [\vfd, Q](\cP_tg\circ \varphi_{\delta-s})(\varphi_s(z)) \leq R_2 e^{-\omega t} \lf(z).
\end{align}
\end{assumption}
In Example \ref{ex:zzs_weakerror} we show that this assumption is satisfied for ZZS with a non-trivial set $\mathcal{G}_1$. 

Finally, we require the following moment bounds.

\begin{assumption}\label{ass:momentboundadvEulerWE}
Let $\{\Zbar_t\}_{t\geq 0}$ be the process described by Algorithm $i$ where $i$ is either \ref{alg:basic_EPDMP} or \ref{alg:advanced_EPDMP} and suppose Assumption~\ref{ass:geoerg} holds for the function $\lf$. Recall the definition of $\lambda$ and $Q$ given by \eqref{eq:defoofQandlambda}. Define for $i=2$ or $3$ (corresponding to Algorithm $i$)
\begin{equation}\label{eq:largerlf}
    \begin{aligned}
        \lf_i(z,r,s) &=  K_i(z,r,s)+ \lambda(\varphi_r(z))Q((Q\lf +\lf )\lambda)(\varphi_{s-r}(z))\\
    &\quad + \lambda(\varphi_r(z)) \lambda(\varphi_s(z))(Q\lf(\varphi_s(z)) +\lf(\varphi_s(z)) )
    \end{aligned}
\end{equation} 
where 
    \begin{align*}
        K_{\ref{alg:basic_EPDMP}}(z,r,s)&= \left( \lf(z) \lambdabar(z,s;\delta)  +   K_{\ref{alg:advanced_EPDMP}}(z,r,s)\right),\\
        K_{\ref{alg:advanced_EPDMP}}(z,r,s)&=\left( (Q\lf(\varphi_s(z))+\lf(\varphi_s(z))) ( \lambdabar(z,s;\delta)\lambdabar(z,r;\delta)  +\overline{M}_2(z))\right).
    \end{align*}
For $i=2$ or $3$ there exist a function $H_i(z)$ such that for any mesh $0=t_0\leq t_1\leq ...$ with $\delta_k=t_k-t_{k-1}<\delta_0$ for any $k$
\begin{equation*}
    \mathbb{E}_z\left[ \sup_{0\leq r<s\leq \delta_0}\lf_i(\Zbar_{t_k},r,s) \right] \leq CH_i(z).
\end{equation*}
\end{assumption}

\begin{remark}
    Observe that since $\lf$ is a Lyapunov function for the PDMP $\{Z_t\}$ we have that $\mathbb{E}_z[\lf(Z_t)]$ is bounded in $t$ for any $z$. Since $\{\Zbar_{t_n}\}_{n\geq 0}$ is designed to be a good approximation of $\{Z_t\}_{t\geq 0}$ we may hope that $\mathbb{E}_z[\lf(\Zbar_{t_n})]$ is also bounded in $n$. We confirm this for ZZS and BPS in 1 dimension in Lemma \ref{lem:1dLyapunovfunction} and test numerically in a higher dimensional setting.
    
    In each of the references discussed after Assumption~\ref{ass:geoerg} there is some freedom in the choice of parameters in the Lyapunov function. 
    By adjusting these parameters we can bound the terms in $\lf_i(z)$ appearing in Assumption~\ref{ass:momentboundadvEulerWE} by using a different choice of the parameters of the Lyapounov function. Confirming Assumption~\ref{ass:momentboundadvEulerWE}  then reduces to showing that, for a fixed Lyapunov function $\lf$ for the PDMP, we have
    \begin{equation*}
        \sup_{n}\mathbb{E}_z[\lf(\Zbar_{t_n})]<\infty.
    \end{equation*}
    
\end{remark}

\begin{theorem}\label{thm:weakerror}
Let $\{Z_t\}_{t\geq 0}$ be the PDMP with generator given by \eqref{eq:genPDMPonekernel}. Let $\{\Zbar_t\}_{t\geq 0}$ be the process described by Algorithm \ref{alg:basic_EPDMP} or  \ref{alg:advanced_EPDMP}, with $\phibar=\varphi$ and $\Fbar=F$. Suppose that Assumption~\ref{ass:lambda_wasserstein} (a), \ref{ass:geoerg},   \ref{ass:momentboundadvEulerWE} holds and that if $\{\Zbar_t\}_{t\geq 0}$ is described by Algorithm \ref{alg:basic_EPDMP} that Assumption~\ref{ass:derivativedeccommutatoralongflow} holds also. Let $\mathcal{G}_1\subseteq\mathcal{G}$ be given as in Assumption~\ref{ass:derivativedeccommutatoralongflow} if this assumption is required and $\mathcal{G}_1=\mathcal{G}$ otherwise.

Then there exists $K>0$ which depends only on $R_1, R_2$ and $C$ such that for any $g\in \mathcal{G}_{1}$, $n\in \N, z\in E$ we have
\begin{equation}\label{eq:WE}
    \left\lvert \mathbb{E}_{z}[g(Z_{t_n})] - \mathbb{E}_{z}[g(\Zbar_{t_n})] \right\rvert \leq KS_n H_i(z).
\end{equation}
Here
\begin{equation}\label{eq:Sn}
   S_n=\sum_{k=0}^{n-1}   \delta_{k+1}^2 e^{-\omega(t_n-t_{k+1})}.
\end{equation}
\end{theorem}

\begin{proof}[Proof of Theorem~\ref{thm:weakerror}]
The proof of this theorem is given in Section \ref{sec:proof_weakerror_thm}.
\end{proof}

The choice of the set $\mathcal{G}_1$ here determines the type of convergence that we obtain. For example if $\mathcal{G}_1$ contains the set of continuous functions with supremum norm bounded by $1$ then this corresponds to convergence in the total variation distance. On the other hand, if $\mathcal{G}$ contains the set of functions with Lipschitz constant less than $1$ then we have convergence in the Wasserstein distance of order $1$. Since we do not require Assumption~\ref{ass:derivativedeccommutatoralongflow} to hold when we use Algorithm \ref{alg:advanced_EPDMP} we can typically take $\mathcal{G}_1=\mathcal{G}$ in that case and hence we have convergence in a metric that is stronger than total variation. However for Algorithm \ref{alg:basic_EPDMP} we need an additional bound on the derivatives of the function so we have convergence in a weaker metric, see Example \ref{ex:zzs_weakerror}.

\begin{remark}\label{note:derivativeest}
An important estimate in the proof of Theorem~\ref{thm:weakerror} will be obtaining a bound between the law of the first jump of the PDMP, $\tau$, and of the approximation process, $\overline{\tau}$. This is done in Lemma \ref{lem:estlawforjump}. In this lemma we need to treat Algorithm \ref{alg:basic_EPDMP} differently to Algorithm \ref{alg:advanced_EPDMP}. In particular, we show convergence as $\delta \to 0$ by considering $\mathbb{E}[h(\tau)]-\mathbb{E}[h(\overline{\tau})]$ for a class $\C$ of test functions $h$. In the case of Algorithm \ref{alg:advanced_EPDMP} we use the set $\C=\C_b([0,\delta])$ of test functions whereas in the case of Algorithm \ref{alg:basic_EPDMP} we use the set $\C=\C^1_b([0,\delta])$. The result of using this weaker convergence is that we need a form of derivative estimate. The derivative estimate we require is given by Assumption~\ref{ass:derivativedeccommutatoralongflow} and is needed only if we are considering Algorithm \ref{alg:basic_EPDMP}.   
\end{remark}

\begin{remark}\label{note:numerical}
To simplify the exposition we have only considered the case when we can simulate the flow exactly. We can extend this proof to allow also for the use of a numerical integrator provided we have a suitable derivative bound. More precisely we require that for some $R_1>0$, $\omega>0$ and some set $\mathcal{G}_1\subseteq\mathcal{G}$ and for any $\delta\leq \delta_0, t>0, z\in \gE, i\in \{1,\ldots,m\}$
\begin{align}
    \sup_{g\in \mathcal{G}_1}\left\lvert \cP_tg(\phibar_{\delta}(z))-\cP_tg(\varphi_{\delta}(z))\right\rvert \leq \delta^2R_3e^{-\omega t} \lf(z),\label{eq:numintest}\\
    \sup_{g\in \mathcal{G}_1}\left\lvert Q_i\cP_tg(\phibar_{\delta}(z))-Q_i\cP_tg(\varphi_{\delta}(z))\right\rvert \leq \delta^2R_3e^{-\omega t} \lf(z) \notag.
\end{align}
\end{remark}

Now using the uniform in time weak error estimate \eqref{eq:WE} and exponential ergodicity \eqref{eq:geoerg} we can show convergence to the invariant measure of the PDMP.
\begin{corollary}
\label{cor:convergence-to-stationary-measure}
Suppose that the conclusion of Theorem~\ref{thm:weakerror} holds. Set $\delta_k=\delta$ for all $k\in \N$. Then for $g\in \mathcal{G}_1$ we have
\begin{align}
    \left\lvert \mathbb{E}_{z}[g(Z_{t_n})] - \mathbb{E}_{z}[g(\Zbar_{t_n})] \right\rvert &\leq C\delta H(z),\label{eq:WEfixeddelta}\\
    \left\lvert \mathbb{E}_{z}[g(\Zbar_{t_n})] -\mu(g)\right\rvert &\leq C H(z)(\delta +e^{-\omega t_n})\label{eq:weakerrortotarget}\\
     \left\lvert \frac{1}{N}\sum_{n=1}^N\mathbb{E}_{z}[g(\Zbar_{t_n})] -\mu(g)\right\rvert &\leq C H(z)\left(\delta+\frac{1}{t_N}\right).\label{eq:timeav}
    \end{align}
\end{corollary}

\begin{proof}[Proof of Corollary~\ref{cor:convergence-to-stationary-measure}]
The proof of this corollary is given in Appendix \ref{sec:proofs_weak_error_theorem_appendix}.
\end{proof}

\begin{corollary}\label{cor:convergence_variable_step_size}
Suppose that the assumptions of Theorem~\ref{thm:weakerror} holds. Assume that $\delta_k\to 0$ as $k\to \infty$ and $\sum_{k=1}^\infty\delta_k=\infty$. For any $g\in \mathcal{G}_{1}$ we have
\begin{equation}\notag
    \lim_{n\to\infty}\left\lvert \mu(g) - \mathbb{E}_{x,v}[g(\Xbar_{t_n},\Vbar_{t_n})] \right\rvert =0.
\end{equation}
\end{corollary}

\begin{proof}[Proof of Corollary~\ref{cor:convergence_variable_step_size}]
The proof of this corollary is given in Appendix \ref{sec:proofs_weak_error_theorem_appendix}.
\end{proof}

\section{\review{Examples}}\label{sec:examples}

\subsection{Examples for Section \ref{sec:Wass_bounds}}
\begin{example}[Zig-zag sampler continued]\label{ex:zzs_wass}
    We continue Example \ref{ex:ZZ} checking that the conditions of the previous section are satisfied. Let us check that approximations of the ZZS based on Algorithm \ref{alg:basic_EPDMP} or \ref{alg:advanced_EPDMP} satisfy Corollary~\ref{cor:strong_error_pdmp}. Assumption~\ref{ass:lipschitz_phi} clearly holds. 
    Assumption~\ref{ass:boundedPDMP} holds because the process travels with bounded velocity, so we can apply the reasoning in Note \ref{note:wass_theorem} to verify Assumption~\ref{ass:lambda_lipschitz}. In particular, Assumption~\ref{ass:F_new} holds as long as $\pot \in \C^2$ and $\gamma_i$ is locally Lipschitz for all $i\in\{1,\dots,m\}$. Assumptions \ref{ass:integrator} and \ref{ass:approx_jump_kernel} follow from the fact that we can simulate exactly the flow and the kernels. Assumption~\ref{ass:lambda_wasserstein}(a) is satisfied for $p=1$ both for $\lambdabar_i(z,s;\delta)=\lambda_i(z)$ and \eqref{eq:lambda_gradfree_zzs} for $\pot \in \C^2$ . Assumption~\ref{ass:lambda_wasserstein}(b) follows from Assumption~\ref{ass:boundedPDMP}. Finally Assumption~\ref{ass:discr_velocities_wass} clearly holds for any $D>0$. 
    
    Note that we could define the same algorithm with $m=1$ according to Note \ref{note:manykernels}. However, in this case neither Assumption~\ref{ass:F_new} nor Assumption~\ref{ass:discr_velocities_wass} hold as the function $F$ is not Lipschitz.
    
    In Figure \ref{fig:zzs_wass} we demonstrate numerically the difference between the ZZS with 50-dimensional Gaussian target and the approximation scheme corresponding to Algorithm \ref{alg:basic_EPDMP} with constant step size $\delta$ and frozen switching rates. \footnote{The codes for all experiments in this paper can be found in a dedicated GitHub repository at \url{https://github.com/andreabertazzi/Euler_PDMC_algorithms}} In this plot the two processes have been coupled according to Coupling \ref{coup:wass_gen_pdmp}, which is a synchronous coupling that is used in Section \ref{sec:proof_wass_theorem} to prove Theorem~\ref{thm:strong_error_pdmp}. In the figure we see that as $\delta$ tends to zero that the distance between the two processes converges to zero. We also observe there is an upper bound on how large the error can get, which roughly corresponds to the velocities having the opposite sign, i.e. $\Vbar_{t_n}=-V_{t_n}$. 
\end{example}

\begin{figure}[t]
\centering
\begin{subfigure}{0.49\textwidth}
    \includegraphics[width=\textwidth]{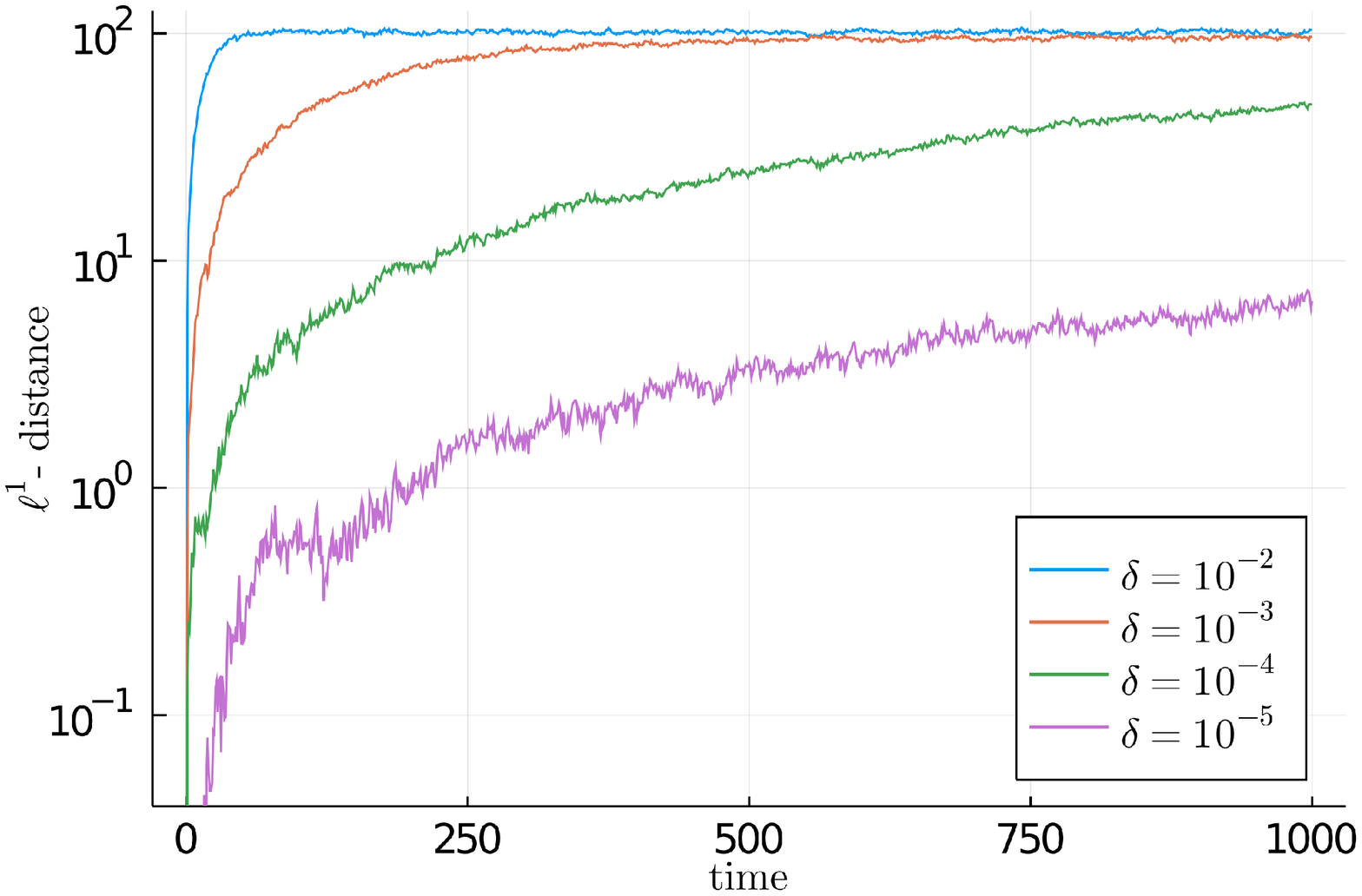}
    \caption{Results for the ZZS with $\gamma(x,v)=0$.}
    \label{fig:zzs_wass}
\end{subfigure}
\begin{subfigure}{0.49\textwidth}
    \includegraphics[width=\textwidth]{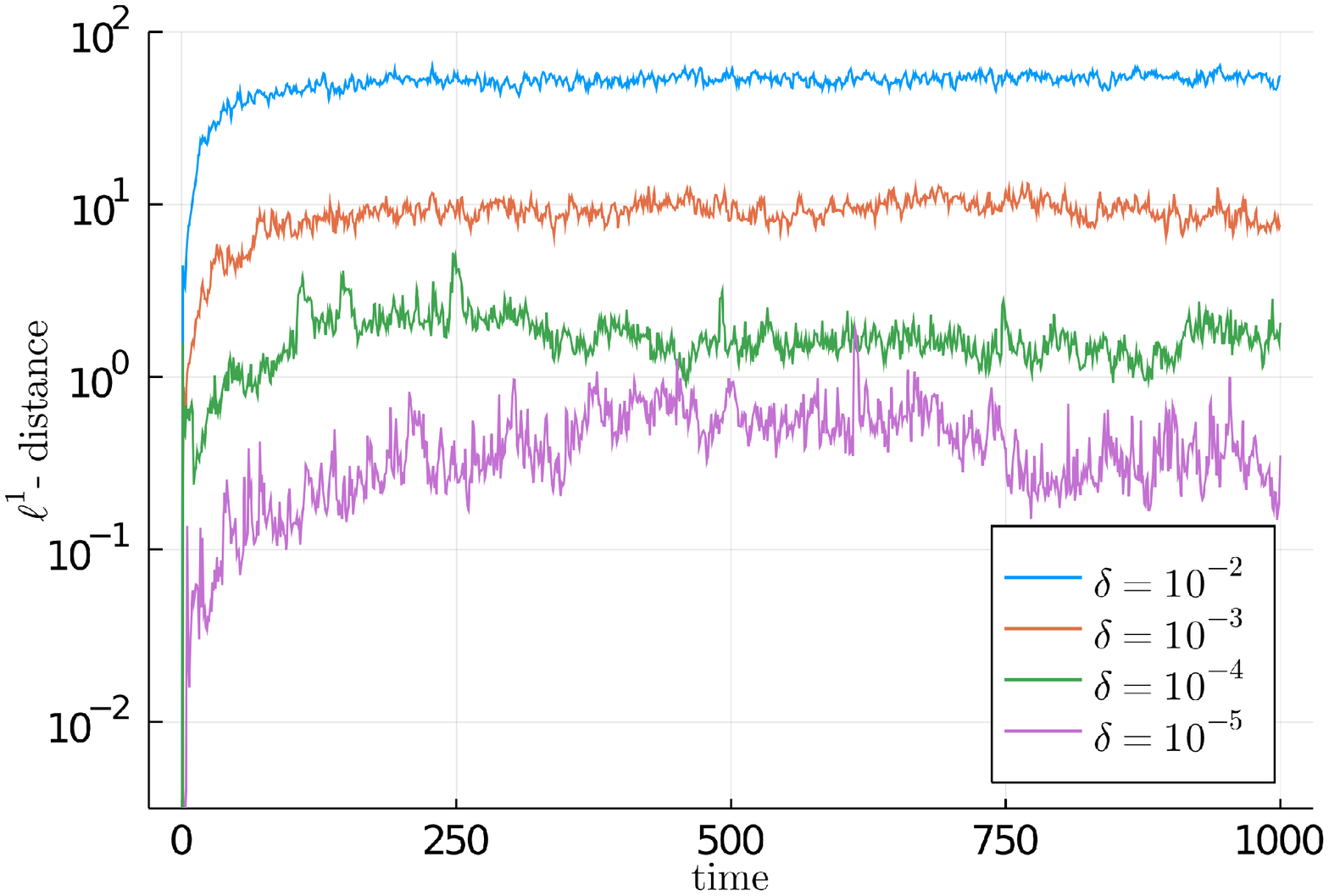}
    \caption{Results for the BPS  with $\lambda_r=1$.}
    \label{fig:bps_wass}
\end{subfigure}
\caption{Plots of the distance between the continuous time PDMPs and their approximations given by Algorithm \ref{alg:basic_EPDMP} for several values of the step size. \review{The $x$-axis shows continuous time units, i.e. the time of $\Zbar_{t_n}$ is $t_n=n\delta$.} The distance is $\lVert x-y\rVert_1=\sum_{i=1}^d\lvert x_i-y_i\rvert$ The plots show the average of $50$ experiments. The processes are coupled according to Coupling \ref{coup:wass_gen_pdmp}. The continuous PDMPs have a $50$-dimensional standard Gaussian as stationary measure. Here we choose $\lambdabar((x,v),s;\delta)=\lambda(x,v)$.}
\label{fig:truncation_error}
\end{figure}

\begin{example}[Bouncy Particle Sampler continued]\label{ex:BPS_wass} We continue Example \ref{ex:BPSintro} and discuss the assumptions of this section in this context. We show that $x\mapsto R(x)v$ need not be Lipschitz. Indeed, for a Gaussian example with $d>1$ fix $v\in\{1,-1\}^d$ and take $y\in \R^d$ orthogonal to $v$ then for any $s>0$ consider
\begin{align*}
    \lVert R(sv)v-R(y)v\rVert  = 2\left\lVert \frac{\langle v,sv\rangle}{\lVert sv\rVert^2} sv\right\rVert=2\lVert v\rVert.
\end{align*}
Letting $s$ tend to zero we see that $x\mapsto R(x)v$ is not Lipschitz at zero and hence Assumption~\ref{ass:F_new} does not hold.

In Figure \ref{fig:bps_wass} we demonstrate numerically the difference between the BPS with 50-dimensional Gaussian target, Gaussian refreshments with rate $1$ and the approximation scheme corresponding to Algorithm \ref{alg:basic_EPDMP} with constant step size $\delta$, and frozen switching rates when coupled according to Coupling \ref{coup:wass_gen_pdmp}. We see that although the assumptions of the theory do not hold the error appears to tend to zero as $\delta\to 0$. Indeed in Section \ref{sec:tv_distance} we obtain theory supporting this observation. Moreover, from the plot it appears that the error converges as $\delta\to 0$ uniformly in time. We will investigate this property further in Section \ref{sec:weakerror}.
\end{example}

\begin{example}[Randomized Hamiltonian Monte Carlo algorithm continued]\label{ex:rhmc_wass}
We continue Example \ref{ex:randHMC}. As long as $\nabla\pot$ is Lipschitz, Proposition~\ref{prop:wass_rhmc} can be applied to the approximations based on Algorithms \ref{alg:basic_EPDMP} and \ref{alg:advanced_EPDMP}.
\end{example}

\begin{example}[PDMP two-dimensional Morris-Lecar model \cite{lemaire_thieullen_thomas_2020}]\label{ex:Morris-Lecar}
Let us consider the PDMP defined on $E=\{0,\ldots,N_K\}\times \R$ whose characteristics are given by
\begin{align*}
    &\Phi(\theta,\nu)=\left(\begin{array}{c}
    0\\
 \frac{1}{C}\left(1-g_{\mathrm{Leak}}(\nu-V_{\mathrm{Leak}})-g_{\mathrm{Ca}}M_\infty(\nu)(\nu-V_{\mathrm{Ca}})-g_K\frac{\theta}{N_K}(\nu-V_K)\right) 
    \end{array}\right),\\
    &\lambda(\theta,\nu)=(N_K-\theta)\alpha_K(\nu)+\theta \beta_K(\nu),\\
    &Q((\theta,\nu),(\theta+1))=\frac{(N_K-\theta)\alpha_K(\nu)}{\lambda(\theta,\nu)},\quad  Q((\theta,\nu), \{\theta-1\}) =\frac{\theta\beta_K(\nu)}{\lambda(\theta,\nu)},\\
    &M_\infty(\nu)=(1+\tanh((\nu-V_1)/V_2))/2,\\ &\alpha_K(\nu)=\lambda_K(\nu)N_\infty(\nu), \quad \beta_K(\nu)=\lambda_K(\nu)(1-N_\infty(\nu)),\\ &N_\infty(\nu)=(1+\tanh((\nu-V_3)/4))/2, \quad  \lambda_K(\nu)=\bar{\lambda}_K\cosh((\nu-V_3)/2V_4).
\end{align*}
This model was given in \cite{lemaire_thieullen_thomas_2020} and is a PDMP version of the deterministic Morris-Lecar model introduced in \cite{MorrisLecar} to explain the dynamics of the barnacle muscle fibre. Here $\nu$ denotes the membrane potential, $\theta$ is number of open Potassium channels, $g_\mathrm{Leak}, g_\mathrm{Ca},g_K$ is maximum conductance value for leak, Calcium, and Potassium respectively, $C$ is the membrane capacitance, $V_\mathrm{leak}, V_{\mathrm{Ca}},V_K$ is the equilibrium potential of relevant ion channels, $M_\infty(\nu)$ ($N_\infty(\nu)$ respectively) is the fraction of open Calcium (Potassium resp.) channels at steady state. $V_1$ ($V_3$ respectively) is the potential at which $M_\infty=0.5$ ($N_\infty=0.5$ resp.). $V_2$ (respectively $V_4$) is the reciprocal is the slope of the voltage dependence of $M_\infty$ ($N_\infty$ resp.). 

We will consider a PD-PDMP approximation of this PDMP. Note that in this case the flow does not have an explicit solution so a numerical integrator is required. Therefore we will set $\phibar_t$ to be an Euler approximation of $\varphi_t$ and $\lambdabar((\theta,\nu),t;\delta)=\lambda(\phibar_t(\theta,\nu;\delta))=\lambda((\theta,\nu)+t\Phi(\theta,\nu))$. Note we can simulate jump times with this approximate rate using Poisson thinning. Since the kernel $Q$ can be simulated exactly we do not need to approximate this. This algorithm is very similar to the approximation proposed in \cite{lemaire_thieullen_thomas_2020} and we confirm their results in our framework. Indeed, one can verify that Assumptions \ref{ass:lipschitz_phi}, \ref{ass:lambda_lipschitz}-\ref{ass:lambda_wasserstein}, as well as Assumption~\ref{ass:discr_velocities_wass} hold so by Theorem~\ref{thm:strong_error_pdmp} we have that the approximation converges as $\delta\to 0$ to the PDMP. 
\end{example}

\subsection{Examples for Section \ref{sec:tv_distance}}\label{sec:examples_tv}
It is straightforward to verify Assumption~\ref{ass:lambda_tv} for either ZZS or BPS. Below we give details for BPS.
\begin{example}[Bouncy Particle Sampler continued]\label{ex:bps_tv}
We continue Examples \ref{ex:BPSintro} and \ref{ex:BPS_wass} and discuss when we may apply Theorem~\ref{thm:tv_distance} in this setting. Recall in Example \ref{ex:BPS_wass} we showed that we can not expect BPS to satisfy the assumptions of Theorem~\ref{thm:strong_error_pdmp} because the reflection operator is in general not Lipschitz. However, we do not need any assumption of this type for Theorem~\ref{thm:tv_distance}. Consider for instance a simple example in which $\lambdabar_1((x,v),s)=\lambda_1(x,v)$. Then Assumption~\ref{ass:lambda_wasserstein} (a) follows provided $\pot\in \C^2$.  If in particular $\pot$ has bounded Hessian, then $\overline{M}_2(x,v)\leq\lVert v\rVert \lVert \nabla^2 U\rVert_\infty$. It remains to verify the moment bounds in Assumption~\ref{ass:lambda_tv} hold. It is clear that these are satisfied if the velocities are bounded, as for instance when refreshments are from $\mathscr{Z}_{n+1}\sim\textnormal{Unif}(\mathbb{S}^{d-1})$ where $\mathbb{S}^{d-1}$ is the unit sphere in $\mathbb{R}^d$. On the other hand, for the BPS with Gaussian refreshments we observe that outside of a compact set one can bound the moments in Assumption~\ref{ass:lambda_tv} by the expectation of the Lyapunov functions derived in \cite{BPSexp_erg} or \cite{BPS_Durmus}. Therefore we can apply Theorem~\ref{thm:tv_distance} to obtain convergence as $\delta\to 0$.

Let us derive a rough estimate on the dimensional dependence of Theorem~\ref{thm:tv_distance} in the $p=1$ case. In particular we focus on the dependence of $D(t_n,z)$ in the dimension. Observe from the proof of the theorem that $D(t,z)$ depends linearly on $L_1,L_2,L_3$, and thus it is sufficient to check the dimensional dependence of such constants. By applying Cauchy-Schwartz inequality the interesting terms are of the form $\mathbb{E}_z[(\lambda_1(X_t,V_t))^2]$. We approximate this expectation with its value in stationarity.
Let us restrict to the case of Gaussian refreshments and a standard Gaussian invariant measure. Then in stationarity we obtain
\begin{align*}
    \mathbb{E}_\pi[(\lambda_1(X,V))^2] = \mathbb{E}_\pi[\langle V,X\rangle_+^2] \leq d^2. 
\end{align*}
Therefore we expect $D(t,z)$ to have a quadratic dependence in the dimension of the PDMP. In order to obtain a fixed error in total variation distance one should then choose $\delta$ such that $D(t,z)\delta$ is constant, and thus $\delta$ of order $d^{-2}$. Observe that taking refreshments on the unit sphere the dependence would be linear in $d$.
\end{example}

\begin{example}[Continuous time approximations of PDMP]\label{ex:continuous_approx_pdmp}
So far we have concentrated on discrete time approximations of PDMP, however it is also possible to apply our results to approximate a PDMP with a second continuous time PDMP. Similarly to the setting of \cite{huggins2017quantifying} suppose we have an approximation $\tilde{\lambda}_i$ of $\lambda_i$, i.e. there exists $\varepsilon>0$ such that for all $i\in\{1,\ldots,m\}$ we have
\begin{equation}\label{eq:cts_lambda_error}
    \lvert \tilde{\lambda}_i(z)-\lambda_i(z) \rvert \leq \Tilde{M}(z) \varepsilon.
\end{equation}
A possible motivation for this approach is the case of ZZS when we either can not evaluate $\partial_i\pot$ exactly or it is too expensive to do so. Then we can use an approximation $\overline{\partial_i\pot}$ to obtain an approximation of $\lambda$, i.e. $\tilde{\lambda}_i=(v_i\overline{\partial_i\pot})_+$.
Now we define a PDMP with approximated rates $\tilde{\lambda}_i$ which moves according to the generator $\tilde{\cL}$ acting on sufficiently smooth functions by
\begin{equation}\label{eq:gen_cts_approx_pdmp}
    \tilde{\cL} f(z) = \langle \phi(z),\nabla_zf(z)\rangle + \sum_{i=1}^m \tilde{\lambda}_i(z)\int (f(z')-f(z)) Q_i(z,\dd z').
\end{equation}
We are interested in comparing this process to the PDMP with generator $\cL$ given by \eqref{eq:genPDMPmanykernels}. In order to use Theorem~\ref{thm:tv_distance} we introduce a discrete time process which we can use to compare to both the PDMPs corresponding to $\cL$ and $\tilde{\cL}$. Set $\delta=\varepsilon$ and $t_n=n\delta$, define $\{\Zbar_{t_n}\}_{n\in \N}$ to be given by Algorithm \ref{alg:advanced_EPDMP} with rates $\lambdabar_i(z,t)=\tilde{\lambda}_i(\varphi_t(z))$, and according to the exact flow $\phibar_t=\varphi_t$ and Markov kernels $Q_i$, i.e. $\Fbar_i=F_i$. We assume that the moment bounds of Assumption~\ref{ass:lambda_tv} are satisfied which is clear for example if the processes move with bounded velocity. We may apply Theorem~\ref{thm:tv_distance} both when the PDMP is given by $\cL$ and by $\tilde{\cL}$. Therefore for any $t>0$ there exist a constant $D(t,z)>0$ such that
\begin{equation*}
   \lVert \cP_{t}(z,\cdot)-\tilde{\cP}_{t}(z,\cdot) \rVert_{TV} \leq 2(1-e^{-D(t,z) \,t\,
   \varepsilon})\leq 2D(t,z) \,t\,
   \varepsilon.  
\end{equation*}
Here $\{\cP_{t}\}_{t\geq 0}$ denotes the semigroup of a PDMP with generator $\cL$ and  $\{\tilde{\cP}_{t}\}_{t\geq 0}$ denotes the semigroup of a PDMP with generator $\tilde{\cL}$. We remark that the analysis above could be adapted to the setting of the Numerical Zig-zag algorithm introduced in \cite{cotter2020nuzz}.
\end{example}

\begin{example}[ZZS with subsampling]\label{ex:zzs_subsampling}
In Bayesian statistics the posterior distribution $\pi(x)\propto\exp(-\pot(x))$ is often of the form $\pot(x) = \sum_{j=1}^N \pot_j(x)$, where $\pot_j(x)$ depends only on a subset of the data.
As described in \cite{ZZ}, the ZZS allows for exact subsampling, which means that the simulation of each event time is calculated using $\pot_J$ for some $J\sim \Unif\{1,\dots,N\}$ instead of the full negative log-density $\pot$. This is achieved by defining switching rates $\lambda_i^j(x,v)= (v_i\partial_i \pot_j(x))_+$ and computational bounds $M_i(t)$ such that $\lambda_i^j(x+vt,v) \leq M_i(t)$. Then starting at state $(x,v)$ at time $t$, a proposal for the next event time is found by taking $\tau_{i^*}=\min \tau_i$, where $\tau_i$ has rate $M_i(t)$ for $i=1,\dots,d$. Then the proposal is accepted with probability $\lambda_{i^*}^J(x+v\tau_{i^*},v) / M_{i^*}(\tau_{i^*})$ for $J\sim \Unif\{1,\dots,N\}$ and in case of acceptance we set $V_{t+\tau_{i^*}}=R_{i^*}V_t$. 

Motivated by the fact that the bounds $M_i(t)$ may be unavailable or hard to compute, we can approximate this process as follows. Here we restrict to the case of frozen switching rates, that is $$\lambdabar_i^j((x,v),s;\delta)=\lambda_i^j(x,v).$$ We apply the same idea behind Algorithm \ref{alg:advanced_EPDMP} to obtain $(\overline{X}_{t_{n+1}},\overline{V}_{t_{n+1}})$ given the previous state by first drawing $J\sim \Unif\{1,\dots,N\}$, and then simulating the next switching time $\bar{\tau}=\bar{\tau}_{i^*}=\min\Bar{\tau}_i$ with rates $\lambda_i^J(\Xbar_{t_n},\Vbar_{t_n})$. Finally
	\begin{align*}
	\overline{X}_{t_{n+1}}&\coloneqq\overline{X}_{t_n}+(\delta_{n+1}-\bar{\tau})\overline{V}_{t_n} +\bar{\tau}\overline{V}_{t_{n+1}}\\
	\overline{V}_{t_{n+1}}&\coloneqq\begin{cases} R_{i^*}\overline{V}_{t_n} & \text{ if } \bar{\tau}\leq \delta_{n+1},\\
	\overline{V}_{t_n} & \text{ if } \bar{\tau}> \delta_{n+1},
	\end{cases}
	\end{align*} 
where the operator $R_i$ was defined in Example \ref{ex:ZZ}. 

The fundamental difference with respect to the setting of Algorithm \ref{alg:advanced_EPDMP} lies in the additional level of randomness introduced by the random variable $J$.
We can adapt the proof of Theorem~\ref{thm:tv_distance} to also allow this additional randomness. Indeed in each step we use a synchronous coupling of the random variable $J$, then conditional on $J$ we can apply Coupling \ref{coup:error_tv_distance} with $\lambda_i$ replaced by $\lambda_i^J$ in \eqref{eq:accrej_cts}-\eqref{eq:accrej_dis},  and setting 
$$
\lambda_{tot}^i((x,v),t;\delta) = \sum_{j=1}^N\lambda_i^j(x+vt,v) +\sum_{j=1}^N\lambda_i^j(x,v)+1.
$$
Thus provided $\pot\in \C^2$ for any $(x,v)\in E$ and $t>0$ there exists $D=D(t,(x,v))>0$ such that
    \begin{equation*}
        \lVert \cP_{t}((x,v),\cdot)-\overline{\cP}_{t}((x,v),\cdot) \rVert_{TV} \leq 1-e^{-D t \delta},
    \end{equation*}
    where $\{\cP_t\}_{t\geq 0}$ is the semigroup of ZZS with subsampling and $\{\overline{\cP}_t\}_{t\geq 0}$ is the transition probability of the approximation. We remark that a similar reasoning can be applied to the BPS with subsampling (see \cite{BPS}).
\end{example}

\subsection{Examples for Section \ref{sec:weakerror}}

\begin{example}[Randomized Hamiltonian Monte Carlo algorithm continued]\label{ex:rhmc_weakerror}
We continue Example \ref{ex:randHMC} by verifying the various assumptions for Theorem~\ref{thm:weakerror}. For this example we assume that $\pot\in \C^2$, is strongly convex and has bounded Hessian, i.e. for some $K,L>0$
\begin{equation}\label{eq:Hessianboundabovebelow}
    K I_d \preceq \nabla^2\pot(q)\preceq L I_d.
\end{equation}
When this holds we have that Assumption~\ref{ass:geoerg} is satisfied by \cite[Theorem 3.9]{bou2017randomized} with $\lf=H$, where $H$ is the Hamiltonian function
$$
H(q,p)=\pot(q)+\frac{1}{2}\lVert p\rVert^2.
$$
Since we consider the approximation based on Algorithm \ref{alg:advanced_EPDMP} we do not need Assumption~\ref{ass:derivativedeccommutatoralongflow}. As $\lambda$ is constant in this case Assumption~\ref{ass:momentboundadvEulerWE} is satisfied provided
\begin{equation}\label{eq:momentbound_rhmc}
    \sup_{n\in \N}\mathbb{E}_{(q,p)}\left[\pot(\Qbar_{t_n})+\frac{1}{2}\lVert \Pbar_{t_n}\rVert^2\right]<\infty.
\end{equation}
Because $\pot$ has bounded second order derivative this reduces to showing that the second moment of the approximation is bounded uniformly in time. This condition depends on the choice of the numerical integrator and should be checked depending on the specific choice.

As mentioned in Note \ref{note:numerical} in order to apply Theorem~\ref{thm:weakerror} with a numerical error we need to verify \eqref{eq:numintest} holds. It is sufficient to show that the derivative of the semigroup decays exponentially. In order to prove this we shall rely on two Lipschitz conditions for the Hamiltonian flow $\varphi_t$: there exist $\nu,K_1, C\geq1, \gamma\in(0,1)$ such that for any $q,\bar{q},p,\bar{p}\in \R^d$
\begin{align}
      \sup_{t>0}\lVert \varphi_t(q,p)-\varphi_t(\bar{q},\bar{p})\rVert &\leq C \lVert (q,p)-(\bar{q},\bar{p})\rVert\label{eq:Lipschitzhamflow}\\
      \lVert \varphi_t(q,p)-\varphi_t(\bar{q},p)\rVert &\leq \gamma\lVert q-\bar{q}\rVert \quad \text{ for } \nu<t\leq K_1. \label{eq:contractionhamflow}
\end{align}
It is shown in \cite{Bou_Rabee_2020} that under \eqref{eq:Hessianboundabovebelow} the contraction \eqref{eq:contractionhamflow} holds for some $\nu,\gamma,K_1$. Indeed the authors prove a stronger result under which $\nu=0$, but $\gamma$ depends on $t$. There are also extensions to non-convex functions $\psi$, however here we only consider the convex setting. On the other hand, \eqref{eq:Lipschitzhamflow} is for instance satisfied for linear flows since the flow preserves the Hamiltonian and the Hamiltonian is equivalent to the norm. To simplify the exposition we will restrict to the case where \eqref{eq:Lipschitzhamflow} and \eqref{eq:contractionhamflow} hold. 
\end{example}
\begin{prop}\label{prop:rHMC_derivativeest}
    Let $\{\cP_t\}_{t\geq 0}$ denote the semigroup of RHMC. Suppose that \eqref{eq:Lipschitzhamflow} and \eqref{eq:contractionhamflow} hold. Moreover assume that
    \begin{equation*}
        0<\kappa:=C (1-(e^{-\lambda\nu}-e^{-\lambda K_1})(1-\gamma C^{-1}))<1.
    \end{equation*}
    Then
    \begin{equation*}
        \lVert \nabla_{q,p} \cP_tf(q,p)\rVert \leq C^2 e^{-\kappa t} \lVert f\rVert_{C_b^1}.
    \end{equation*}
\end{prop}

\begin{proof}
The proof is deferred to Appendix \ref{sec:proofs_rhmc_weakerror}.
\end{proof}

A case where it is easy to see that $\kappa<1$ is the standard Gaussian case, i.e. $\pot(q)=\lVert q\rVert^2/2$, since in this case we have that $C=1$. 

\begin{theorem}
Suppose that $\pot$ satisfies \eqref{eq:Hessianboundabovebelow}, \eqref{eq:Lipschitzhamflow}, \eqref{eq:contractionhamflow}, and the numerical integrator satisfies \eqref{eq:momentbound_rhmc}. Then the conclusions of Theorem~\ref{thm:weakerror} hold.
\end{theorem}

\begin{example}[Zig Zag Sampler continued]\label{ex:zzs_weakerror}
Recall the notation of Example \ref{ex:ZZ} and \ref{ex:zzs_wass}. Let us verify the assumptions of Theorem~\ref{thm:weakerror} for the ZZS. 
\end{example}

We will make the following assumptions on $\pot$. Assume $\pot\in \C^2$ and
		\begin{equation} \label{eq:growthbounds}          
		\lim_{\lVert x\rVert\to\infty}\frac{\max(1,\lVert \nabla_x^2 \pot(x)\rVert)}{\lVert \nabla_x \pot(x)\rVert}=0 \quad \text{ and } \lim_{\lVert x\rVert \to \infty}\frac{\lVert \nabla_x\pot(x)\rVert}{\pot(x)}=0.
		\end{equation}

\textbf{Verifying Assumption~\ref{ass:geoerg}:}

Geometric ergodicity of the ZZS was established in \cite{Bierkensergodicity} under \eqref{eq:growthbounds} with Lyapunov function 
\begin{equation}\label{eq:1dlfZZS}
    \lf_{\alpha,\epsilon}(x,v) = \exp\left(\alpha \pot(x)+\sum_{i=1}^d\phi_\epsilon(v_i\partial_i\pot(x))\right).
\end{equation}
Here $\phi_\epsilon(s)=\mathrm{sign}(s)\log(1+\epsilon \lvert s\rvert)/2$, $\alpha\in(0,1)$, $\epsilon>0$, $\alpha>\epsilon \overline{\gamma}$, where $\overline{\gamma}$ upper bounds the excess switching rate $\gamma:E\to\mathbb{R}_+$.

\textbf{Verifying Assumption~\ref{ass:derivativedeccommutatoralongflow}:}

When dealing with this assumption it is convenient to use a smooth choice of $\lambda_i$, so for this section we will set $\phi(r)=r(1+r)^{-1}$ and
\begin{equation}\label{eq:smooth_lambda}
    \lambda_i(x,v) = -\log\left(\phi(\exp(-v_i\partial_i\pot(x))\right)
\end{equation}
which was shown in \cite{AndrieuLivingstone} to be a smooth choice of $\lambda_i$ for which the ZZS has the correct invariant measure. Note that for this choice of $\lambda_i$ the excess switching rate $\gamma$ takes values between $0$ and $\overline{\gamma}=\log(2)$.

\begin{lemma}\label{lem:ZZS_commutator}
Let $\{\cP_t\}_{t\geq 0}$ denote the semigroup corresponding to the ZZS as described in Example \ref{ex:ZZ}. Assume $\pot\in \C^2$ and has bounded Hessian. For $\lambda_i$ given by \eqref{eq:smooth_lambda} there exist a constant $C$ depending on the Hessian of $\pot$ and on $d$ such that for any $f\in \C^1$ we have
\begin{align*}
    [\vfd,Q](f\circ\varphi_{\delta-s})(x+sv,v) &\leq C \sup_{i\in\{1,\ldots,d\}}\{\lvert f(x+sv+(\delta-s)F_iv,F_iv)\rvert \\
    & \quad +\lvert\partial_{x_i}f(x+sv+(\delta-s)F_iv,F_iv)\rvert\}.
\end{align*}

\end{lemma}

\begin{proof}[Proof of Lemma \ref{lem:ZZS_commutator}]
The proof is deferred to Appendix \ref{sec:proofs_zzs_weakerror}.
\end{proof}

We apply this Lemma with $f=\cP_tg$ with $g\in \mathcal{G}_1$ where
\begin{equation}\label{eq:setoffunctions}
    \mathcal{G}_1=\{g:E\to\R: x\mapsto g(x,v)\in \C^1, \mu(g)=0, \lvert g\rvert \leq \lf_{\overline{\alpha},\epsilon}, \lVert \nabla_x g\rVert \leq \lf_{\overline{\alpha},\epsilon}\}
\end{equation}
where $\overline{\gamma}\epsilon<\overline{\alpha}<\alpha<1$. Such an $\overline{\alpha}$ can always be found by taking $\epsilon$ sufficiently small. It remains to show that $\nabla_x\cP_tg$ converges to zero for $g\in \mathcal{G}_1$.


\begin{theorem}\label{thm:exptimederivativebound}
	Let $\{\cP_t\}_{t\geq 0}$ denote the semigroup of the ZZS with generator given by \eqref{eq:ZZgen} and with $\lambda_i$ such that $x\mapsto\lambda_i(x,v)\in \C^1$ for each $v$ and has bounded derivative $\nabla_x\lambda_i(x,v)$. Fix $\epsilon\overline{\gamma}<\overline{\alpha}<\alpha$, and let $\mathcal{G}_1$ be given by \eqref{eq:setoffunctions}. Then there exists a constant $C$ such that for any $g\in \mathcal{G}_{1}$ 
	\begin{align}\notag
	\lVert\nabla_x\cP_tg(x,v)\rVert &\leq  C(1+t)e^{-\omega t} \lf_{\alpha,\epsilon}(x,v). 
	\end{align}
\end{theorem} 

\begin{proof}[Proof of Theorem~\ref{thm:exptimederivativebound}]
The proof is deferred to Appendix \ref{sec:proofs_zzs_weakerror}.
\end{proof}

Note that by adjusting $C,\kappa$ and $\alpha$ we can show that
\begin{align*}
	\sup_{\delta\in (0,\delta),s\in[0,\delta]}\lVert\nabla_x\cP_tg\circ \varphi_{\delta-s}(x+sv,v)\rVert &\leq  Ce^{-\omega t} \lf_{\alpha,\epsilon}(x,v). 
\end{align*}
Therefore Assumption~\ref{ass:derivativedeccommutatoralongflow} holds by combining Lemma \ref{lem:ZZS_commutator}, Theorem~\ref{thm:exptimederivativebound} and Assumption~\ref{ass:geoerg}.

\textbf{Verifying Assumption~\ref{ass:momentboundadvEulerWE}:}
Note that since $\lambda$ grows at most linearly it is sufficient to show that
there exist a function $H(x,v)$ for any mesh $0=t_0\leq t_1\leq ...$ with $\delta_k=t_k-t_{k-1}<\delta_0$ for any $k$
\begin{equation*}
    \mathbb{E}_z\left[ \lVert \Xbar_{t_k} \rVert^2\lf_{\alpha,\epsilon}(\Xbar_{t_k},\Vbar_{t_k}) \right] \leq H(x,v).
\end{equation*}
Note that $\lf_{\alpha,\epsilon}$ is dominated by the $e^{\alpha\pot}$ term so we can bound $\lVert x\rVert^2\lf_{\alpha,\epsilon}(x,v)$ by $e^{\alpha_1\pot}$ for any $\alpha_1>\alpha$ so it remains to show there exist a function $H(x,v)$ for any mesh $0=t_0\leq t_1\leq ...$ with $\delta_k=t_k-t_{k-1}<\delta_0$ for any $k$
\begin{equation}\label{eq:mom_bound_zzs}
    \mathbb{E}_z\left[ e^{\alpha_1\pot(\Xbar_{t_k})} \right] \leq H(x,v).
\end{equation}
Then in the $1$-dimensional case we prove that the required bound holds.
\begin{lemma}
Suppose \eqref{eq:growthbounds} hold and $d=1$. Then \eqref{eq:mom_bound_zzs} holds for both Algorithms \ref{alg:basic_EPDMP} and \ref{alg:advanced_EPDMP}, hence Assumption~\ref{ass:momentboundadvEulerWE} is satisfied.
\end{lemma}
This follows from Lemmas \ref{lem:1dLyapunovfunction} and \ref{lem:lfordering} which can be found in Appendix \ref{sec:proofs_zzs_weakerror}. The generalisation to the $d$-dimensional setting is challenging and is thus left as a conjecture, supported by the experiments in Figure \ref{fig:zzs_momentbound}.
\begin{figure}[t]
\centering
\begin{subfigure}{0.49\textwidth}
    \includegraphics[width=\textwidth]{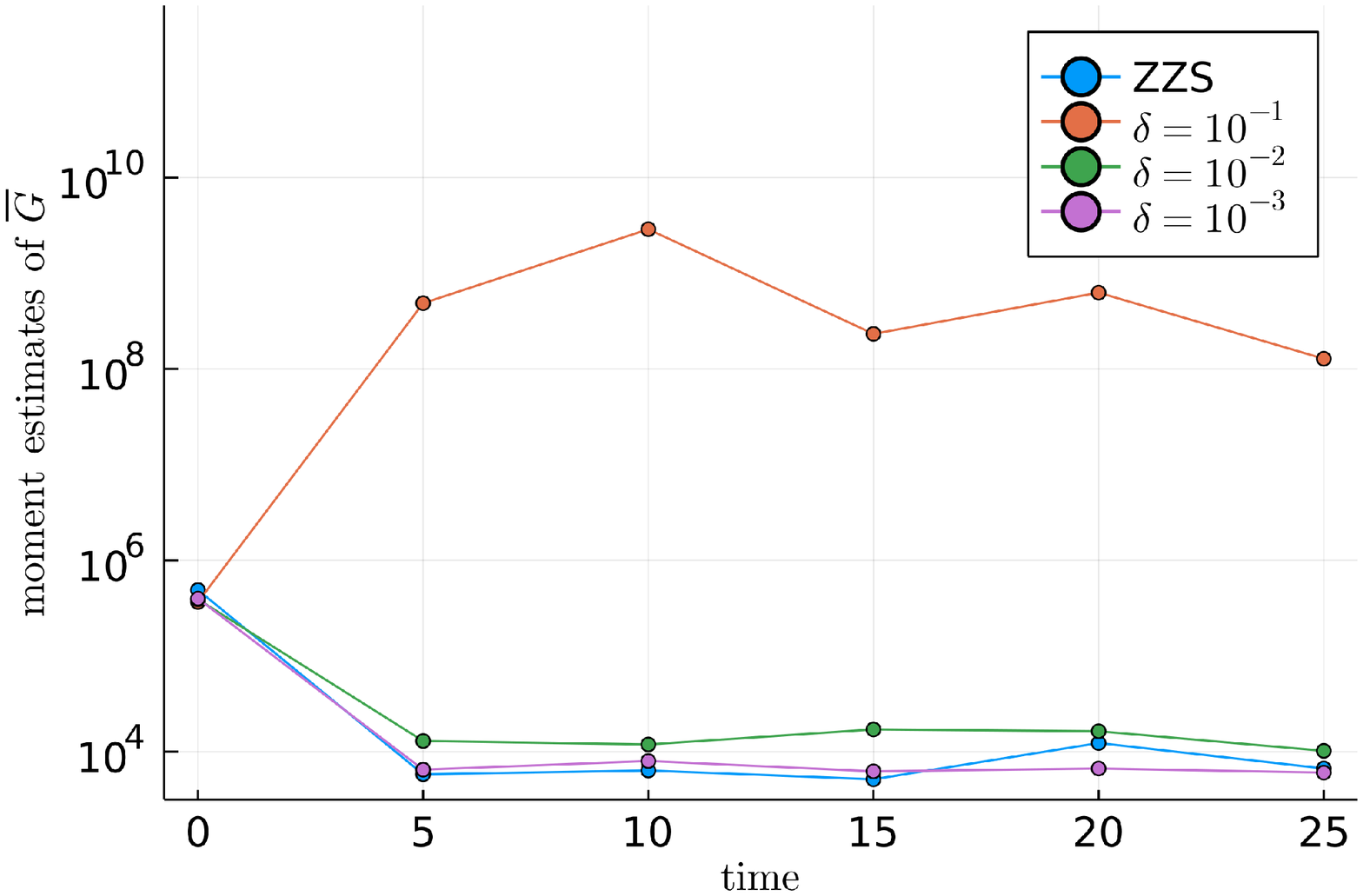}
    \caption{Results for the ZZS and its approximations given by Algorithm \ref{alg:basic_EPDMP}. Here $\gamma(x,v)=0$.}
    \label{fig:zzs_momentbound}
\end{subfigure}
\begin{subfigure}{0.49\textwidth}
    \includegraphics[width=\textwidth]{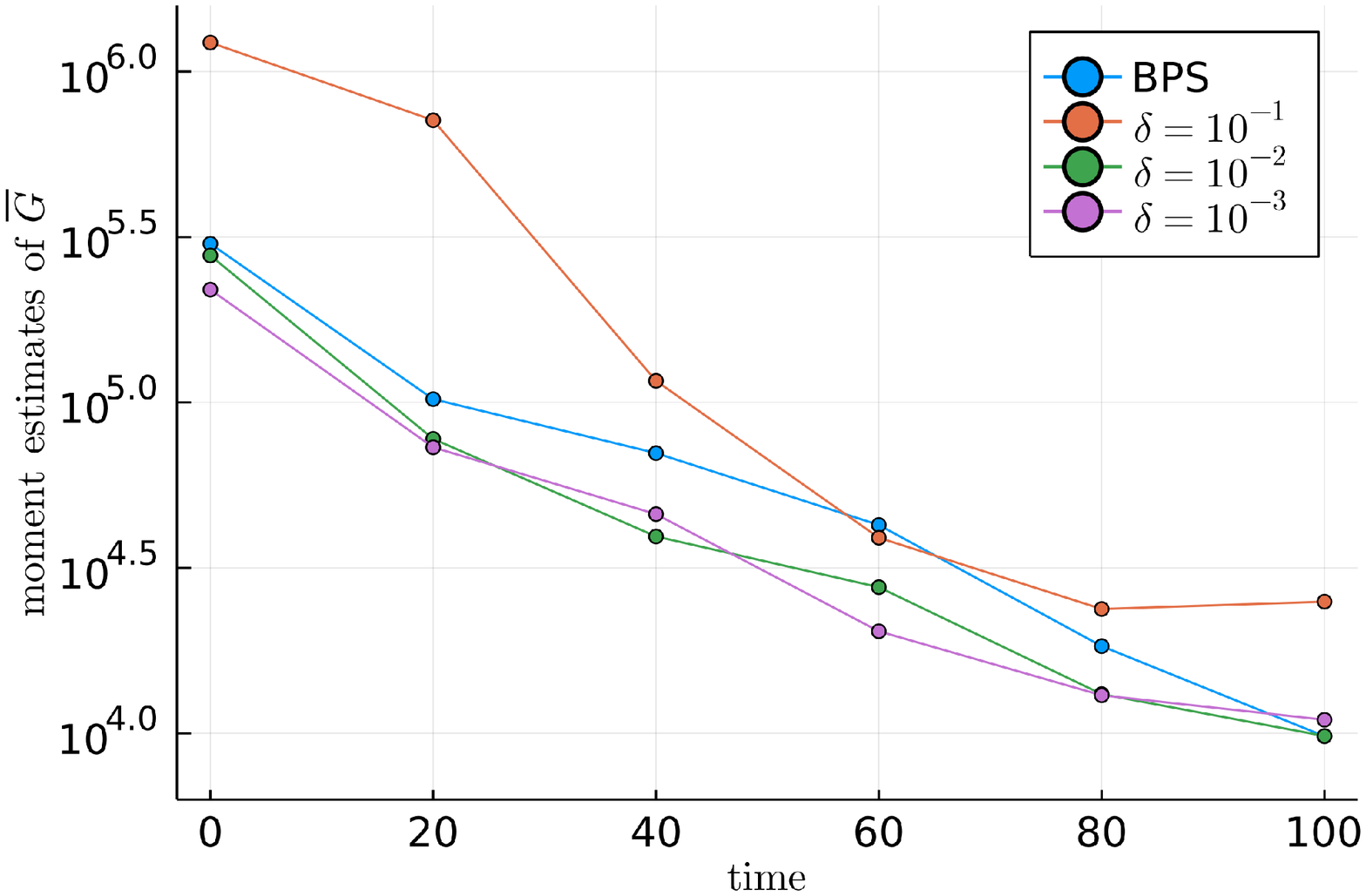}
    \caption{Results for the BPS and its approximations given by Algorithm \ref{alg:advanced_EPDMP}. Here $\lambda_r=1$.}
    \label{fig:bps_momentbound}
\end{subfigure}
\caption{Plots of the estimates of $\mathbb{E}[\overline{G}(X_t,V_t)]$ and $\mathbb{E}[\lf(\Xbar_t,\Vbar_t)]$, which are respectively for the continuous time PDMPs and their approximations for several values of the step size. The plots show the average of $10^5$ experiments. The continuous PDMPs have a $25$-dimensional standard Gaussian as stationary measure. Here we choose $\lambdabar((x,v),s;\delta)=\lambda(x,v)$. For each experiment $X_0,\overline{X}_0$ are given by an independent realisation of the sum of a $25$-dimensional standard Gaussian and a uniform random variable on $[0,1]^{25}$, while $V_0,\overline{V}_0$ are drawn from the stationary distribution of the continuous time PDMP.}
\label{fig:momentbounds}
\end{figure}

\begin{conjecture}\label{con:momentbound}
Suppose $\pot$ satisfies \eqref{eq:growthbounds}. Then inequality \eqref{eq:mom_bound_zzs} holds for Algorithms \ref{alg:basic_EPDMP} and \ref{alg:advanced_EPDMP}.
\end{conjecture}

\begin{theorem}\label{thm:weakerrorZZS}
Let $\{(X_t,V_t)\}_{t\geq 0}$ be the ZZS. Let $\{(\Xbar_t,\Vbar_t)\}_{t\geq 0}$ be the process described in Example \ref{ex:ZZ}. Assume that $\pot$ satisfies \eqref{eq:growthbounds} and that Conjecture \ref{con:momentbound} holds. Let $\mathcal{G}_1$ be given by \eqref{eq:setoffunctions}. Then the conclusions of Theorem~\ref{thm:weakerror} hold.
\end{theorem}
In Figure \ref{fig:errors} we show some numerical results in the case of a Gaussian target. We observe that the error in the estimation of the first component of the mean of the approximations is similar to that of the continuous ZZS, while the error for the radius statistic, i.e. $t(x)=\sum_{i=1}^d x_i^2$, obtained with the approximations decreases to that of the ZZS as $\delta$ becomes smaller.

\begin{example}[BPS continued]\label{ex:bps_weakerror}
Let us discuss the assumptions of Theorem~\ref{thm:weakerror} for the approximation of BPS as given in Example \ref{ex:BPSintro}, \ref{ex:BPS_wass}, and \ref{ex:bps_tv}. Since this approximation is based on Algorithm \ref{alg:advanced_EPDMP} it is sufficient to check Assumptions \ref{ass:geoerg} and \ref{ass:momentboundadvEulerWE}. Conditions under which Assumption~\ref{ass:geoerg} holds are given in \cite{BPSexp_erg} and \cite{BPS_Durmus}. To be concrete we  concentrate on \cite{BPSexp_erg}, in which the Lyapunov function is given by
\begin{equation*}
    \lf(x,v)=\frac{e^{\frac{1}{2}\psi(x)}}{\sqrt{\lambda(x,-v)}}.
\end{equation*}
Here at refreshment times a new velocity vector is drawn from the uniform distribution on the unit sphere. In Figure \ref{fig:bps_momentbound} we estimate the moments of $\lf$ for the continuous time BPS with a $25$-dimensional standard Gaussian target and compare it to the approximations obtained by applying Algorithm \ref{alg:advanced_EPDMP} for several values of $\delta$. We observe that the moments of the approximations resemble the continuous BPS and $\mathbb{E}[\lf(\Xbar_t,\Vbar_t)]$ appears to be bounded uniformly in time. Therefore we conjecture that Theorem~\ref{thm:weakerror} holds under the assumptions of \cite{BPSexp_erg} for approximations of the BPS according to Algorithm \ref{alg:advanced_EPDMP}. 

In Figure~\ref{fig:errors} we compare the errors of the BPS and its approximations given by Algorithm \ref{alg:advanced_EPDMP} in the case of a Gaussian target. We observe that the approximations perform similarly to the BPS. \review{Note that for this target measure the BPS and the approximation are both rotationally invariant so they both have mean zero and hence in Figure~\ref{fig:errors}~(c) we do not see the effect of the bias of the approximation.}
\end{example}

\begin{figure}[t]
\centering
\begin{subfigure}{0.49\textwidth}
    \includegraphics[width=\textwidth]{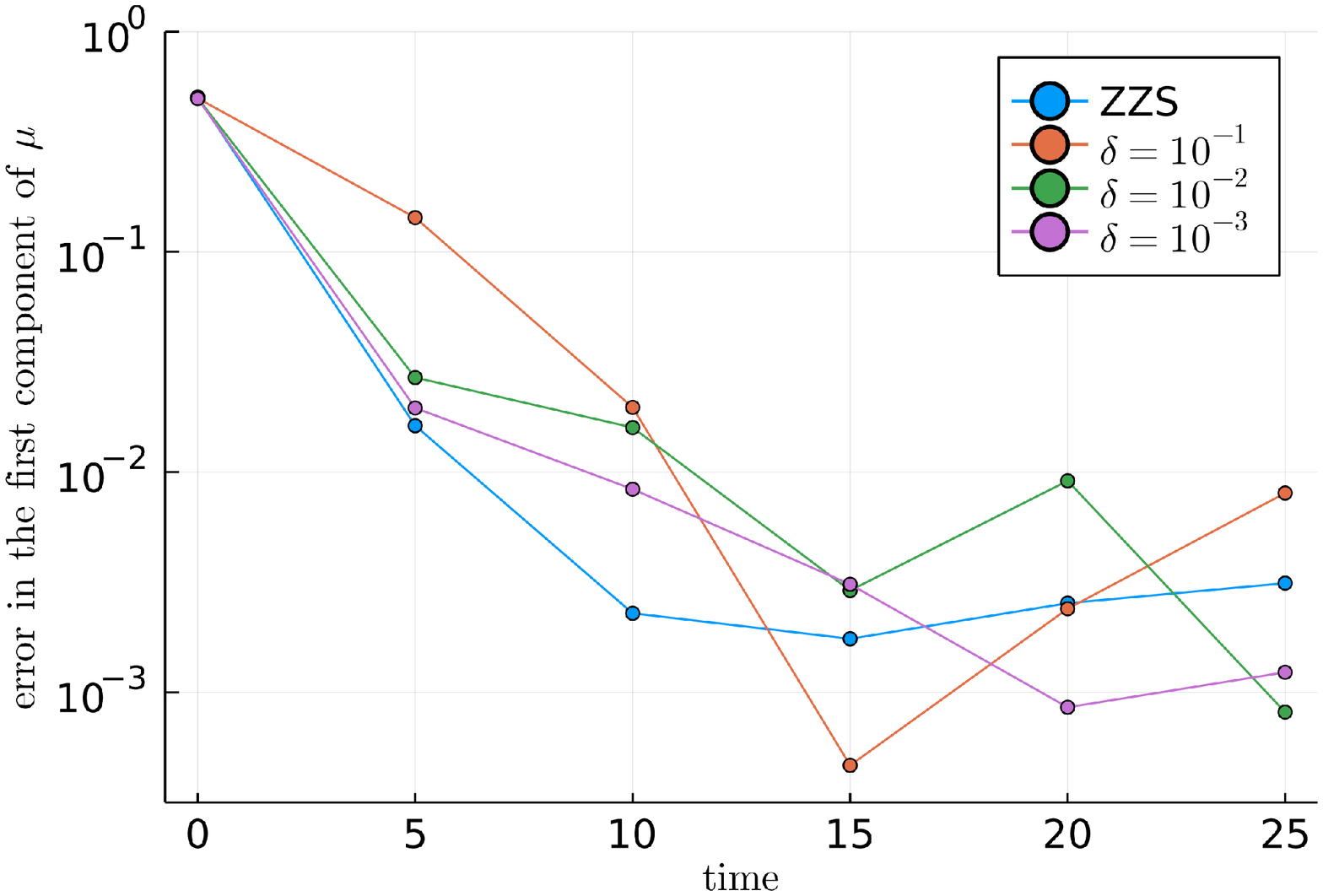}
    \caption{Error for the mean for the ZZS and its approximations given by Algorithm \ref{alg:basic_EPDMP}. }
    \label{fig:zzs_meanerror}
\end{subfigure}
\begin{subfigure}{0.49\textwidth}
    \includegraphics[width=\textwidth]{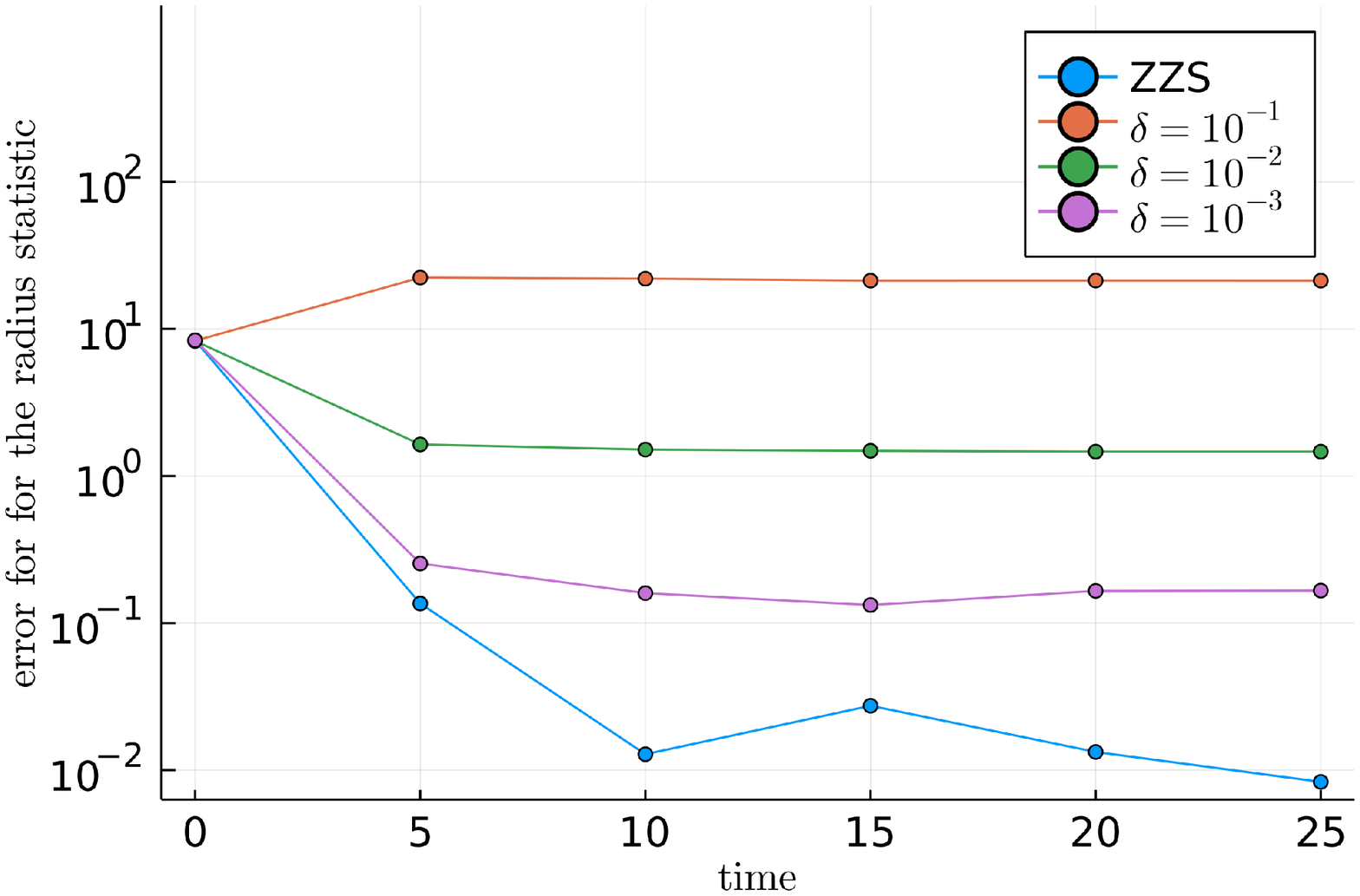}
    \caption{Error for the radius for the ZZS and its approximations given by Algorithm \ref{alg:basic_EPDMP}. }
    \label{fig:zzs_raderror}
\end{subfigure}
\begin{subfigure}{0.49\textwidth}
    \includegraphics[width=\textwidth]{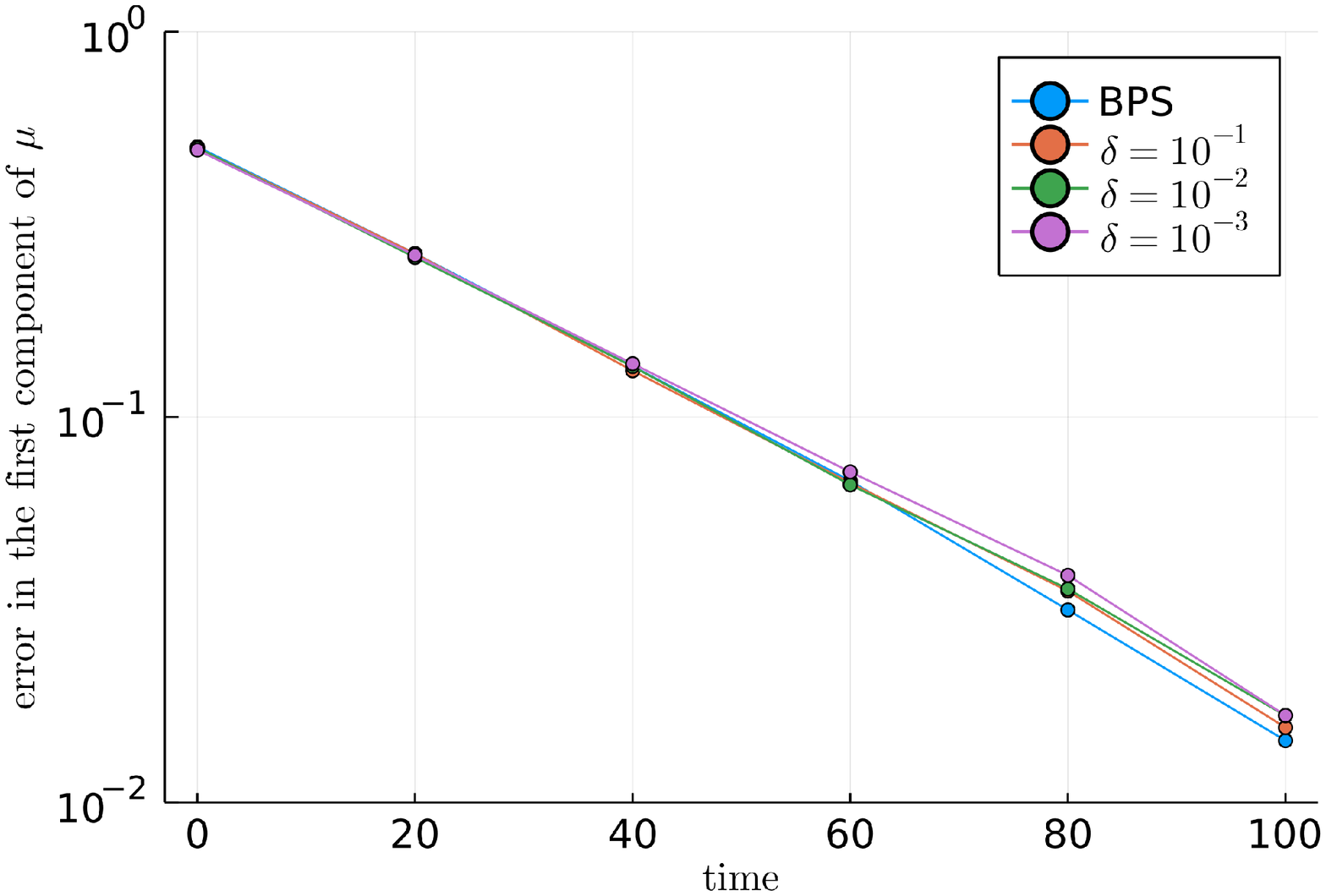}
    \caption{Error for the mean for the BPS and its approximations given by Algorithm \ref{alg:advanced_EPDMP}. }
    \label{fig:bps_meanerror}
\end{subfigure}
\begin{subfigure}{0.49\textwidth}
    \includegraphics[width=\textwidth]{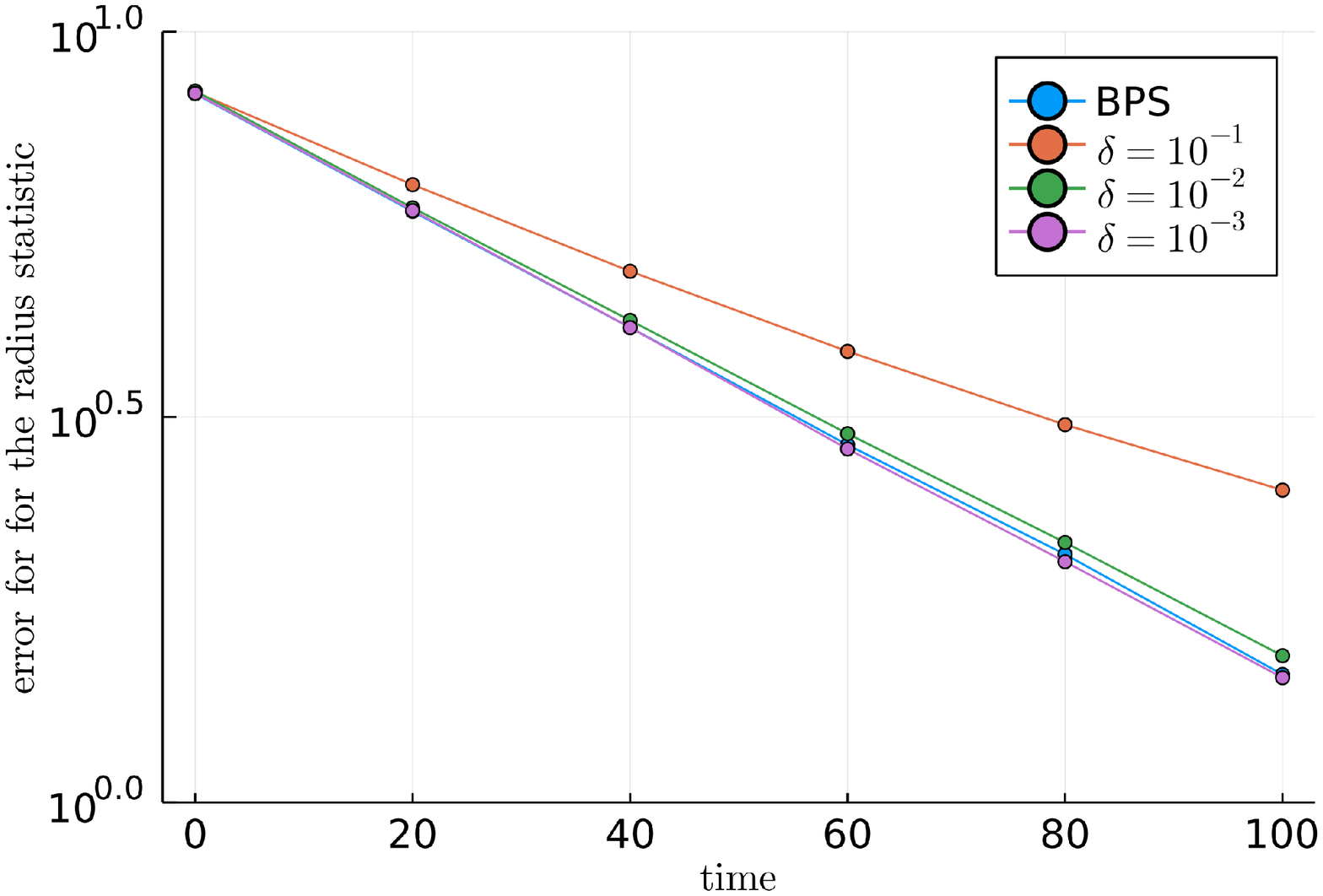}
    \caption{Error for the radius for the BPS and its approximations given by Algorithm \ref{alg:advanced_EPDMP}.}
    \label{fig:bps_raderror}
\end{subfigure}
\caption{Errors in the estimation of the first component of the mean and radius statistic in the context of Figure \ref{fig:momentbounds}. For the ZZS we take $\gamma(x,v)=0$, while for the BPS we have $\lambda_r=1$.}
\label{fig:errors}
\end{figure}

\begin{example}[Continuous time approximation of a PDMP]\label{ex:continuous_approx_weakerror}
We continue our analysis of the setting introduced in Example  \ref{ex:continuous_approx_pdmp}. We wish to extend the conclusions of Theorem~\ref{thm:weakerror} to the continuous PDMP with generator $\Tilde{\cL}$ given by \eqref{eq:gen_cts_approx_pdmp}.
\end{example}
\begin{theorem}
Suppose both the PDMPs with generators $\cL$ and $\Tilde{\cL}$ satisfy Assumption~\ref{ass:geoerg} with Lyapunov functions $\overline{G}$ and $\Tilde{G}$ respectively, and with invariant measures $\mu$ and $\Tilde{\mu}$. Assume \eqref{eq:cts_lambda_error} holds for some $\varepsilon>0$. Moreover, suppose the approximation of the PDMP with generator $\cL$ described in Example \ref{ex:continuous_approx_pdmp} satisfies Assumption~\ref{ass:momentboundadvEulerWE} both for  $\overline{G}$ and $\Tilde{G}$. Set $\mathcal{G}_1=\{g\in\C(E): \lvert g(x,v)\rvert\leq \min\{\overline{G}(x,v),\tilde{G}(x,v)\}\}$.
Then for all $g\in\mathcal{G}_1$
\begin{equation*}
    \left\lvert \mathbb{E}_{z}[g(Z_{t})] - \mathbb{E}_{z}[g(\tilde{Z}_{t})] \right\rvert \leq C \varepsilon H(z).
\end{equation*}
Moreover, letting $t\to\infty$ we have
\begin{equation*}
    \left\lvert \mu(g) - \tilde{\mu}(g) \right\rvert \leq D \varepsilon .
\end{equation*}
\end{theorem}
Hence, in the case of ZZS we recover the result obtained in Theorem 6.2 of \cite{huggins2017quantifying}.

\section{Proof of Theorem~\ref{thm:strong_error_pdmp}}\label{sec:proof_wass_theorem}
We shall first prove the case of $p=1$ in Section \ref{sec:proof_wass_p=1}, and then in Section \ref{sec:proofs_wass_p>1} we will use the $p=1$ setting as a base case in a proof by induction to obtain the result for $p>1$.

\subsection{The case of $p=1$}\label{sec:proof_wass_p=1}
To prove Theorem~\ref{thm:strong_error_pdmp} in this setting we define a coupling of $Z_{t_n}$ and $\Zbar_{t_n}$ that satisfies the bounds in the statement. Then because the Wasserstein distance is defined as an infimum over all couplings we immediately obtain 
$$\mathcal{W}_1(\cP_{T}(z,\cdot),\overline{\cP}_{T}(z,\cdot))\leq  \mathbb{E}_z[\lVert Z_{T}-\Zbar_{T} \rVert],$$
where the expectation in the right hand side is with respect to the specific coupling we consider.

Let us now introduce a general framework that contains both Algorithm \ref{alg:basic_EPDMP} and Algorithm \ref{alg:advanced_EPDMP}. Denote the approximation process as $\Zbar_{t_n}$ with initial state $\Zbar_0 = z$. Then given the previous state $\Zbar_{t_n}$ define 
\begin{equation}
\begin{aligned}\label{eq:tautilde_wass}
    \tilde{\tau}_{n+1}^i &\coloneqq \inf \left\{t\geq 0: 1-\exp\left(-\int_0^t \lambdabar_i(\Zbar_{t_n},s;\delta_{n+1})\dd s  \right) \geq \tilde{U}_{n+1}^i\right\}, \\ \tilde{\tau}_{n+1} &\coloneqq \min_{i=1,\dots,m} \tilde{\tau}_{n+1}^i, \quad I_{n+1}=\argmin_{i=1,\dots,m} \tilde{\tau}_{n+1}^i
\end{aligned}
\end{equation}
where $\tilde{U}_{n+1}^1,\ldots, \tilde{U}_{n+1}^m\stackrel{iid}{\sim} \text{Unif}[0,1]$. Then the  switching time of the process is
\begin{equation}\notag
    \taub = \overline{\tau}_{n+1}(\Tilde{\tau}_{n+1},\delta_{n+1}),
\end{equation}
with the requirement that $\taub\leq \delta_{n+1}$ if and only if $\tilde{\tau}_{n+1} \leq \delta_{n+1}$.
In particular Algorithm \ref{alg:basic_EPDMP} corresponds to the choice
\[\taub(\Tilde{\tau}_{n+1},\delta_{n+1}) = \delta_{n+1}\mathbbm{1}_{\{\tilde{\tau}_{n+1}\leq \delta_{n+1}\}} + \infty \mathbbm{1}_{\{\tilde{\tau}_{n+1}> \delta_{n+1}\}}, \]
while Algorithm \ref{alg:advanced_EPDMP} corresponds to 
\[\taub(\Tilde{\tau}_{n+1},\delta_{n+1}) = \tilde{\tau}_{n+1}.\]
The process can now be defined as follows:
\begin{align*}
    &\Zbar_{t_{n+1}} = \begin{cases}  \phibar_{\delta_{n+1}}(\Zbar_{t_n};\delta_{n+1}) &\text{if } \taub> \delta_{n+1}, \\
    \phibar_{\delta_{n+1}-\taub}(\Fbar_{I_{n+1}}(\phibar_{\taub}(\Zbar_{t_n};\delta_{n+1}),U_{n+1};\delta_{n+1});\delta_{n+1}) &\text{if } \taub\leq \delta_{n+1},
    \end{cases}
\end{align*}
where $U_{n+1}\sim \nuU$ takes values in $\cU$.

Let us now define the coupling of $Z_{t_n}$ and $\Zbar_{t_n}$ that we will use to prove Theorem~\ref{thm:strong_error_pdmp}.

\begin{coupling}\label{coup:wass_gen_pdmp}
Fix both processes up to time $t_n$ and let $(Z_{t_{n+1}},\Zbar_{t_{n+1}})$ evolve as follows. Let $U_{n+1} \sim \nuU$ and $\tilde{U}_{n+1}^1,\ldots,\tilde{U}_{n+1}^m \stackrel{iid}{\sim} \textnormal{Unif}([0,1])$ be independent of each other and of $Z_{t_n}$ and $\Zbar_{t_n}$. The coupling evolves as follows:
\begin{itemize}
    \item Define the next switching time of the continuous process as 
    \begin{equation}\label{eq:tau_wass}
    \begin{aligned}
        \tau_{n+1}^i &\coloneqq \inf \left\{t\geq 0: 1-\exp\left(-\int_0^t \lambda_i(\varphi_s(Z_{t_n}))\dd s  \right) \geq \tilde{U}_{n+1}^i \right\},\\
     \tau_{n+1} &\coloneqq \min_{i=1,\dots,m} \tau_{n+1}^i.
        \end{aligned}
    \end{equation}  
    Then there are two cases:
    \begin{itemize}
        \item if $\tau_{n+1} \leq \delta_{n+1}$, then set $Z_{t_n+s}=\varphi_s(Z_{t_n})$ for $s\in(0,\tau_{n+1})$ and $$Z_{t_n + \tau_{n+1}} = F_{I_{n+1}}(\varphi_{\tau_{n+1}}(Z_{t_n}),U_{n+1})$$ where ${I_{n+1}}=\argmin_i \tau_{n+1}^i$. Then simulate the process independently of the rest for the remaining time $\delta_{n+1} - \tau_{n+1}$.
        \item if $\tau_{n+1} > \delta_{n+1}$, then set $Z_{t_n+s} = \varphi_s(Z_{t_n})$ for $s\in(0,\delta_{n+1}]$.
    \end{itemize}
    \item Define $\tilde{\tau}_{n+1}$ as in Equation~\eqref{eq:tautilde_wass}, where $\Tilde{U}_{n+1}^1,\ldots, \Tilde{U}_{n+1}^m$ is the same random variable used in \eqref{eq:tau_wass}. Compute $\taub = \overline{\tau}_{n+1}(\Tilde{\tau}_{n+1},\delta_{n+1})$. Then the approximation process evolves as:
    \begin{itemize}
        \item if $\taub \leq \delta_{n+1}$, then set $$\Zbar_{t_{n+1}}= \phibar_{\delta_{n+1}-\taub}(\overline{F}_{\overline{I}_{n+1}}(\phibar_{\taub}(\Zbar _{t_n};\delta_{n+1}),U_{n+1};\delta_{n+1});\delta_{n+1})$$ where ${\bar{I}_{n+1}}=\argmin_i \tilde{\tau}_{n+1}^i$. 
        \item if $\taub > \delta_{n+1}$, then set $\Zbar_{t_{n+1}} = \phibar_{\delta_{n+1}}(\Zbar_{t_n};\delta_{n+1})$.
    \end{itemize}
\end{itemize}
Therefore the first switching times of the two processes are coupled, and so is the eventual random jump. Once $Z_{t_{n+1}}$ and $\Zbar_{t_{n+1}}$ have been obtained, repeat the same procedure to obtain $Z_{t_{n+2}}$ and $\Zbar_{t_{n+2}}$.
\end{coupling}
\noindent We remark that the marginal distributions of each process is the correct one, and thus this is indeed a valid coupling of the two processes.

In the proof that follows we simplify the notation denoting the approximations as $\phibar_z(z)$, $\lambdabar_i(z,s)$, and $\Fbar_i(z,U)$, instead of $\phibar_z(z;\delta_{n+1})$, $\lambdabar_i(z,s;\delta_{n+1})$, and $\Fbar_i(z,U;\delta_{n+1})$.
\begin{proof}[Proof (Theorem~\ref{thm:strong_error_pdmp})]
    We begin by partitioning the space as
	\begin{equation*}
	\mathbb{E}_z [\lVert Z_{t_{n+1}}-\Zbar_{t_{n+1}}\rVert] = \mathbb{E}_z [\lVert Z_{t_{n+1}}-\Zbar_{t_{n+1}}\rVert (\mathbbm{1}_{E_{00}}+\mathbbm{1}_{E_{11}}+\mathbbm{1}_{E_{10}}+\mathbbm{1}_{E_{01}})],
	\end{equation*}
	where $E_{ij}$ for $i,j=0,1$ denotes the event in which there are $i$ random events for the approximation process, while $j=0$ denotes that no events take place for the continuous process, and $j=1$ that at least one event for the original process happens in the time interval $s\in[t_n,t_{n+1})$. The four events are considered respectively in Lemmas \ref{Lemma:strong_error_E00}, \ref{lemma:strong_error_E11}, \ref{lemma:strong_error_E10}, and \ref{lemma:strong_error_E01}. \review{Since the upper bounds in these results are non-decreasing functions of the time $t_n$, we combine the results of the Lemmas to obtain that there exist constants $K_1= K_1(t_{n+1})$ and $K_2=K_2(t_{n+1})$ such that }
	\begin{equation*}
	    \mathbb{E}_z[\lVert Z_{t_{n+1}}-\Zbar_{t_{n+1}} \rVert] \leq (1+\delta_{n+1} K_1) \mathbb{E}_z[\lVert Z_{t_{n}}-\Zbar_{t_{n}} \rVert] + \delta_{n+1}^2 K_2.
	\end{equation*}
	Since the two processes start at the same point this implies, by recursion, 
	\begin{equation}\label{eq:wassbound_nonconst}
	    \mathbb{E}_z[\lVert Z_{t_{n+1}}-\Zbar_{t_{n+1}} \rVert] \leq \sum_{k=1}^{n+1} K_2\delta^2_k \left(\prod_{\ell=k}^{n+1} (1+\delta_\ell K_1)\right).
	\end{equation}
	In the setting when $\delta_n=\delta$ the right hand side of \eqref{eq:wassbound_nonconst} becomes a geometric series which leads to the estimate 
	\begin{equation}\label{eq:wassbound_const}
	\begin{aligned}
	\mathbb{E}_z[\lVert Z_{t_{n+1}}-\Zbar_{t_{n+1}} \rVert] &\leq \delta^2 \frac{(1+\delta K_1)^{n+1} -1}{(1+\delta K_1) -1} K_2 \\
	& \leq \delta \left( e^{K_1 (n+1)\delta } -1 \right) \frac{K_2}{K_1}.
	\end{aligned}
	\end{equation}
\end{proof}

\subsection{The case of $p>1$}\label{sec:proofs_wass_p>1}

In order to simply the notation we shall restrict to the case $\delta_n=\delta$. To prove the result we reason by induction on $p$. In particular, we consider the following inductive hypothesis. Fix $p\geq 1$ and $n\geq 1$.
\begin{hypothesis}\label{hyp:induction_higherorder}
Suppose the PDMP satisfies Assumptions \ref{ass:lipschitz_phi}-\ref{ass:lambda_lipschitz}. Moreover suppose the approximation given by Algorithm \ref{alg:secondorder_EPDMP} satisfies Assumptions \ref{ass:integrator}-\ref{ass:boundedPDMP} hold for some $\delta_0>0$. Given $(Z_{t_n},\Zbar_{t_n})$ there exist a coupling $(Z_{t_{n+1}},\Zbar_{t_{n+1}})$ with respective marginals corresponding to $\cP_{\delta}(Z_{t_n},\cdot),\overline{\cP}_{\delta}(\Zbar_{t_n},\cdot;\delta,p)$ there exist $A=A(T)$, $B=B(T)$ independent of $n$ such that for any $0<\delta\leq \delta_0$ 
	\begin{equation}\label{eq:inductivehypothesis}
	    \mathbb{E}_z[\lVert Z_{t_{n+1}}-\Zbar_{t_{n+1}} \rVert] \leq A \delta^{p+1} +  (1+B\delta) \mathbb{E}_z[\lVert Z_{t_n}-\Zbar_{t_n} \rVert].
	\end{equation}
\end{hypothesis}
It is sufficient to show that the Inductive hypothesis holds and then the statement of the Theorem follows by recursion in $n$ as done in \eqref{eq:wassbound_const}.
Observe that the case $p=1$, which corresponds to Algorithm \ref{alg:advanced_EPDMP}, holds by the proof of Section \ref{sec:proof_wass_p=1}. Suppose the Inductive Hypothesis holds for some $p\geq 1$, let us consider the case of $p+1$.
Let us define the following coupling of $(Z_{t_{n+1}},\Zbar_{t_{n+1}})$ given $(Z_{t_n},\Zbar_{t_n})$.
\begin{coupling}\label{coup:general_higherorder}
Define for $0\leq t\leq \delta$ $$\lambda^i_{tot}(z,\overline{z},t;\delta,p+1) = \lambdabar_i(\overline{z},t;\delta,p+1)+\lambda_i(\varphi_t(z))+1.$$ 
Then for $i=1,\dots,m$ draw the proposed event times $T_i$ with distribution given by $$\mathbb{P}(T_i > t) = \exp\left(-\int_{0}^t \lambda_{tot}^i(Z_{t_n},\Zbar_{t_n},r;\delta,p+1)\dd r\right).$$  Let $T_{i^*} = \min_{i=1,\dots,m} T_i $ and let $i^*$ be the argument that minimises $T_i$. If $T_{i^*}\geq \delta$, then let $Z_{t_{n+1}}=\varphi_\delta(Z_{t_n})$, $\Zbar_{t_{n+1}}=\phibar_{\delta}(\Zbar_{t_n};\delta,p+1)$.

Consider now the case in which $T_{i^*}< \delta$. Let $ U\sim \nuU$ and $\overline{U}\sim \textnormal{Unif}([0,1])$ independent of the $T_i$'s and independent of each other. Then set 
$$
\tau_* = T_{i^*} \quad \textnormal{ if } \overline{U} \leq \frac{\lambda_{i^*}(\varphi_{T_{i^*}}(Z_{t_n}))}{\lambda_{tot}^{i^*} (Z_{t_n},\Zbar_{t_n},T_{i^*};\delta,p+1)},
$$ 
i.e. the proposed event time is accepted for the continuous time process. Alternatively set $\tau_*(z) = R$ for some constant $R>\delta$. Similarly let
$$ 
\overline{\tau}_* = T_{i^*} \quad \textnormal{ if } \overline{U}\leq \frac{\lambdabar_{i^*}(\Zbar_{t_n},T_{i^*};\delta,p+1)}{\lambda_{tot}^{i^*} (Z_{t_n},\Zbar_{t_n},T_{i^*};\delta,p+1)},
$$
and thus conditional on acceptance $\overline{\tau}_*$ is the next event time for the approximation process. In case of rejection set $\overline{\tau}_* = R$ for some constant $R>\delta$ as done above.
Set $Z_{t_n+s}=\varphi_s(z)$ and $\Zbar_{t_n+s}=\phibar_s(\Zbar_{t_n};\delta,p+1)$ for $s\in[0,T_{i^*})$. We distinguish three scenarios:
\begin{enumerate}[label=(\arabic*)]
    \item The proposed switching time $T_{i^*}$ is accepted by both processes. Then set 
    \begin{align*}
        & Z_{t_n+T_{i^*}}=F_{i^*}(\varphi_{T_{i^*}}(Z_{t_n}),U), \\
        & \Zbar_{t_n+T_{i^*}}=\Fbar_{i^*}(\phibar_{T_{i^*}}(\Zbar_{t_n};\delta,p+1),U;\delta,p+1).
    \end{align*}
    To get from time $t_n+T_{i^*}$ to $t_{n+1}$ we apply the coupling given by the Inductive Hypothesis \ref{hyp:induction_higherorder}.
    \item The proposed switching time $T_{i^*}$ is accepted for one process, but rejected for the other. To get from time $t_n+T_{i^*}$ to $t_{n+1}$ we let the two processes evolve independently according to their marginal distributions.
    \item The proposed switching time $T_{i^*}$ is rejected by both processes. Then set 
    $Z_{T_{i^*}}= \varphi_{T_{i^*}}(Z_{t_n})$ and $\Zbar_{T_{i^*}} =\phibar_{T_{i^*}}(\Zbar_{t_n};\delta,p+1)$.
    To get from time $t_n+T_{i^*}$ to $t_{n+1}$ we repeat this procedure starting at time $t_n+T_{i^*}$ and with $\delta$ replaced with $\delta-T_{i^*}$.
\end{enumerate}
\end{coupling}

\begin{proof}[Proof of Theorem~\ref{thm:strong_error_pdmp}]
Assume $Z_0=\Zbar_0=z$. Suppose that \eqref{eq:inductivehypothesis} holds for some $p\geq 1$. We will show that \eqref{eq:inductivehypothesis} then follows for $p$ replaced $p+1$ by using Coupling \ref{coup:general_higherorder}.

Suppose first that $T_{i^*}>\delta$. Then the two processes follow the deterministic flow and by Assumption~\ref{ass:integrator} with order $p+1$ and Lemma \ref{lem:flow_to_integrator} we have
\begin{align*}
     \mathbb{E}_z[\lVert Z_{t_{n+1}}-\Zbar_{t_{n+1}} \rVert \mathbbm{1}_{T_{i^*}>\delta}] &= \mathbb{E}_z[\lVert \varphi_\delta(Z_{t_{n}})-\phibar_\delta(\Zbar_{t_{n}};\delta,p+1) \rVert \mathbbm{1}_{T_{i^*}>\delta}] \\
     & \leq (1+CC'\delta) \mathbb{E}_z[\lVert Z_{t_{n+1}}-\Zbar_{t_{n+1}} \rVert ] + \tilde{C}^{p+2}.
\end{align*}

Let us consider the case (1) in Coupling \ref{coup:general_higherorder} and denote the corresponding event as $E_1$. Then using the Inductive Hypothesis \ref{hyp:induction_higherorder}
\begin{align*}
     &\mathbb{E}_z[\lVert Z_{t_{n+1}}-\Zbar_{t_{n+1}} \rVert \mathbbm{1}_{E_1}] = \mathbb{E}_z \mathbb{E}_z[\lVert Z_{t_{n+1}}-\Zbar_{t_{n+1}} \rVert \mathbbm{1}_{E_1}\rvert T_{i^*}] \\
     & \leq \mathbb{E}_z[(A\delta^{p+1} + (1+B\delta)\lVert Z_{t_{n}+T_{i^*}}-\Zbar_{t_{n}+T_{i^*}} \rVert ) \mathbbm{1}_{E_1}]\\
     & = \mathbb{E}_z[(A\delta^{p+1} + (1+B\delta)\lVert F_{i^*}(\varphi_{T_{i^*}}(Z_{t_{n}}),U)\!-\!\Fbar_{i^*}(\phibar_{T_{i^*}}(\Zbar_{t_{n}};\delta,p+1),U;\delta,p+1) \rVert ) \mathbbm{1}_{E_1}]\\
     & \leq \mathbb{E}_z[(A\delta^{p+1} + (1+B\delta)(M_1\delta^{p+1}  + D_2\lVert \varphi_{T_{i^*}}(Z_{t_{n}})-\phibar_{T_{i^*}}(\Zbar_{t_{n}};\delta,p+1) \rVert) ) \mathbbm{1}_{E_1}] \\
     & \leq \mathbb{E}_z[(A\delta^{p+1} + (1+B\delta)(M_1\delta^{p+1}  + D_2(1+CC'\delta)\lVert Z_{t_{n}}-\Zbar_{t_{n}} \rVert + D_2\tilde{C}\delta^{p+2} ) ) \mathbbm{1}_{E_1}].
\end{align*}
Here we used Assumption~\ref{ass:F_new}(b), and Lemma \ref{lem:flow_to_integrator}.
Then we take advantage of $$\mathbb{P}_z(T_{i^*}<\delta)\leq 1-\exp(-\delta(2L(t_{n+1},z,p+1)+m))\leq  \delta(2L(t_{n+1},z,p+1)+m)$$ to get
\begin{align}\notag
    \mathbb{E}_z[\lVert Z_{t_{n+1}}-\Zbar_{t_{n+1}} \rVert \mathbbm{1}_{E_1}] \leq \Tilde{A}_1\delta^{p+2} + (1+\Tilde{B}_1\delta) \mathbb{E}_z[\lVert Z_{t_{n}}-\Zbar_{t_{n}} \rVert]
\end{align}
for suitable constants $\tilde{A}_1,\Tilde{B}_1$ and taking advantage of $\delta<\delta_0$.

Now consider the case (2) in Coupling \ref{coup:general_higherorder} and denote the corresponding event as $E_2$. Note that 
\begin{equation*}
    \mathbb{P}_z(E_2\rvert Z_{t_n},\Zbar_{t_n}) =  \delta(2L(t_{n+1},z,p+1)+m) \frac{\lvert \lambda_{i^*}(\varphi_{T_{i^*}}(Z_{t_n}))-\lambdabar_{i^*}(\Zbar_{t_n},T_{i^*};\delta,p+1)\rvert}{\lambda_{tot}^{i^*} (Z_{t_n},\Zbar_{t_n},T_{i^*};\delta,p+1)}.
\end{equation*}
Using Assumptions \ref{ass:lipschitz_phi}, \ref{ass:lambda_lipschitz}, \ref{ass:lambda_wasserstein}, and the triangle inequality we obtain
\begin{align*}
    \mathbb{P}_z(E_2\rvert Z_{t_n},\Zbar_{t_n}) \leq  \delta(2L(t_{n+1},z,p+1)+m) (D_4C'\lVert Z_{t_n}-\Zbar_{t_n}\rVert+\delta^{p+1}\overline{M}_2(\Zbar_{t_n})).
\end{align*}
Therefore using Assumption~\ref{ass:boundedPDMP}
\begin{equation*}
    \mathbb{E}_z[\lVert Z_{t_{n+1}}-\Zbar_{t_{n+1}} \rVert \mathbbm{1}_{E_2}] \leq \Tilde{A}_2\delta^{p+2} + (1+\Tilde{B}_2\delta) \mathbb{E}_z[\lVert Z_{t_{n}}-\Zbar_{t_{n}} \rVert]
\end{equation*}
    for some constants $\tilde{A}_2,\tilde{B}_2$.
    
Let us consider the case (3) in Coupling \ref{coup:general_higherorder} and denote the corresponding event as $E_3$. Note that since case (3) involves repeating the coupling we may have to repeat this step an arbitrary number of times. Let $q$ denote the number of times we propose a candidate jumping time. If $q<p+2$ then we must have reached case (1) or (2), so it is sufficient to use the respective estimate derived above to get the desired result. On the other hand, the probability that $q\geq p+2$  is bounded by $(2L(t_{n+1},z,p+1)+m)^{p+2}\delta^{p+2}$, which gives us the correct order. 
\end{proof}

\section{Proof of Theorem~\ref{thm:tv_distance}}\label{sec:proofs_tv}

\subsection{The case of $p=1$}\label{sec:proofs_tv_p=1}
To prove the result we define a coupling of the continuous process with the approximation process. The intuitive idea is that, assuming the two processes are equal at the beginning of the current time step, we can use Poisson thinning \cite{devroye:1986,lewis_shedler_thinning} to simulate a proposal for the next event time that is common to both processes. This is achieved by simulating a Poisson process with rate given by the sum of the rates of the two processes. The proposal is then accepted or rejected individually for each process based on the correct switching rates. For this acceptance-rejection step a common uniform random variable is used. If the proposal is accepted for both processes, then a coupled event takes place, thus ensuring that the processes are equal after the event has happened. If the thinning step is successful it follows that the processes are equal for all $s \in (t_n,t_{n+1}]$  unless a second event takes place for the continuous time process in the current time interval, which is an event with $\mathcal{O}(\delta_{n+1}^2)$ probability. Let us now give the formal definition of the coupling.

\begin{coupling}\label{coup:error_tv_distance}
Let $t_n$ be the current time and assume $Z_{t_n}=\Zbar_{t_n}=z_n$. Define $\lambda^i_{tot}(z,t;\delta_{n+1}) = \lambdabar_{i}(z,t;\delta_{n+1})+\lambda_i(\varphi_t(z))+1$. Then for $i=1,\dots,m$ draw the proposed event times $T_i(z_n)$ with distribution 
\begin{equation*}
    \mathbb{P}(T_i(z_n) \leq t) = 1-\exp\left(-\int_{0}^t \lambda_{tot}^i(z_n,r;\delta_{n+1})\dd r\right).
\end{equation*}
Let $T_{i^*}(z) = \min_{i=1,\dots,m} T_i(z) $. Now let $U_{n+1}\sim \nuU$ and $\overline{U}\sim \textnormal{Unif}([0,1])$ independent of the $T_i$'s and of $Z_{t_n}$. Then set 
\begin{equation}\label{eq:accrej_cts}
\tau(z_n) = T_{i^*}(z_n) \quad \textnormal{ if } \overline{U} \leq \frac{\lambda_{i^*}(\varphi_{T_{i^*}(z_n)}(z_n))}{\lambda_{tot}^{i^*} (z_n,T_{i^*}(z_n);\delta_{n+1})},
\end{equation}
hence upon acceptance the proposed event time is the next switching time for the continuous time process. Alternatively set $\tau(z_n) = R>\delta_{n+1}$ for some constant $R\neq T_{i^*}(z_n)$. Similarly let
\begin{equation}\label{eq:accrej_dis}
\overline{\tau}(z_n) = T_{i^*}(z_n) \quad \textnormal{ if } \overline{U}\leq \frac{\overline{\lambda}_{i^*}(z_n,T_{i^*}(z_n);\delta_{n+1})}{\lambda_{tot}^{i^*} (z_n,T_{i^*}(z_n);\delta_{n+1})},
\end{equation}
and thus conditional on acceptance $\overline{\tau}(z_n)$ is the next event time for the approximation process. In case of rejection set $\overline{\tau}(z_n) = R>\delta_{n+1}$ for some constant $R\neq T_{i^*}(z_n)$ as done above.

If $T_{i^*}\geq \delta_{n+1}$, then let $Z_{t_n+s}= \Zbar_{t_n+s}=\varphi_s(z_n)$ for $s\in(0,\delta_{n+1}]$. In this case the two processes are equal at time $t_{n+1}$.

Alternatively, we have $T_{i^*}< \delta_{n+1}$ and thus we set $Z_{t_n+s}=\Zbar_{t_n+s}=\varphi_s(z_n)$ for $s\in(0,T_{i^*}(z_n))$. Then the continuous process evolves as follows:

\begin{itemize}
    \item if $\tau(z_n) = T_{i^*}(z_n)$, then set $$Z_{t_n+\tau(z_n)} = F_{i^*}(\varphi_{\tau(z_n)}(z_n), U_{n+1}).$$ Then let the process evolve independently of the approximation until time $t_{n+1}$.
    \item  if $\tau(z_n) \neq T_{i^*}(z_n)$, the proposed event time is rejected and we let the process evolves independently of the approximation until time $t_{n+1}$.
\end{itemize}
On the other hand, the approximation process evolves as follows:
\begin{itemize}
    \item if $\overline{\tau}(z_n)= T_{i^*}(z)$, set  
    $$
    Z_{t_n+\overline{\tau}(z_n)} = F_{i^*}(\varphi_{\overline{\tau}(z_n)}(z_n), U_{n+1}),
    $$
    and finally $\Zbar_{t_n+s}=\varphi_s(Z_{t_n+\overline{\tau}(z_n)})$ for $s\in(\overline{\tau}(z_n),\delta_{n+1}]$.
    \item if $\overline{\tau}(z_n) \neq T_{i^*}(z_n)$, then repeat this procedure from the beginning starting at time $t_{n}+T_{i^*}(z_n)$ and with step $\delta_{n+1}-T_{i^*}(z_n)$.
\end{itemize}
\end{coupling}
\begin{lemma}\label{lem:tv_dist_aux}
Under Assumption~\ref{ass:lambda_tv}, there exists $D(t_n,z)>0$ such that
\begin{equation*}
    \mathbb{P}_z(Z_{t_{n}} \neq \Zbar_{t_{n}}\rvert Z_{t_{n-1}}=\Zbar_{t_{n-1}}) \leq D(t_n,z)\delta_{n}^2,
\end{equation*}
for $D(t_n,z) = (L_1(t_{n},z)/2 +L_2(t_n,z) + L_3(t_n,z)/2)$.
\end{lemma}
\begin{proof}
The proof is postponed to Appendix \ref{sec:proofs_tv_p=1_appendix}.
\end{proof}
\begin{proof}[Proof of Theorem~\ref{thm:tv_distance}]
    By the coupling inequality we have $$\lVert P_{t_n}(z,\cdot)-\overline{P}_{t_n}(z,\cdot) \rVert_{TV} \leq \mathbb{P}_z(Z_{t_n}\neq \Zbar_{t_n})$$ and thus it is sufficient to bound the right hand side.
    Apply Lemma \ref{lem:tv_dist_aux} to obtain
    \begin{equation}
    \begin{aligned}\label{eq:recursion_tv}
        \mathbb{P}_z(Z_{t_{n}} \neq \Zbar_{t_n}) & = \mathbb{P}_z(Z_{t_{n}} \neq \Zbar_{t_n}\rvert Z_{t_{n-1}} \neq \Zbar_{t_{n-1}})  \mathbb{P}_z(Z_{t_{n-1}} \neq \Zbar_{t_{n-1}}) \\
        & \quad + \mathbb{P}_z(Z_{t_{n}} \neq \Zbar_{t_n}\rvert Z_{t_{n-1}} = \Zbar_{t_{n-1}})  (1-\mathbb{P}_z(Z_{t_{n-1}} \neq \Zbar_{t_{n-1}})) \\
        & \leq \mathbb{P}_z(Z_{t_{n-1}} \neq \Zbar_{t_{n-1}}) + D(t_n,z)\delta_n^2(1-\mathbb{P}_z(Z_{t_{n-1}} \neq \Zbar_{t_{n-1}})) \\
        & = (1-D(t_n,z)\delta_n^2) \mathbb{P}_z(Z_{t_{n-1}} \neq \Zbar_{t_{n-1}}) + D(t_n,z)\delta_n^2.
        \end{aligned}
    \end{equation}
    Thus by recursion and since $\Zbar_0=Z_0=z$ it follows that
    \begin{equation*}
        \mathbb{P}_z(Z_{t_{n}} \neq \Zbar_{t_n}) \leq \sum_{i=1}^n D(t_n,z)\delta_i^2 \prod_{\ell=i+1}^n (1-D(t_n,z)\delta_\ell^2).
    \end{equation*}
    In particular if $\delta_n=\delta$ for all $n\in\N$ we have that 
    \begin{align*}
        \mathbb{P}_z(Z_{t_{n}}  \neq \Zbar_{t_n}) &\leq D(t_n,z)\delta^2 \sum_{\ell=0}^{n-1} (1-D(t_n,z)\delta^2)^\ell \\
        & \leq 1-(1-D(t_n,z)\delta^2)^n \\
        & \leq 1-e^{-D(t_n,z) t_n\delta}.
    \end{align*}
\end{proof}

\subsection{The case of $p>1$}\label{sec:proof_higher_order_tv}
In order to simplify the notation we shall restrict to the case $\delta_n=\delta$. To prove the result we reason by induction on $p$ similarly to Section \ref{sec:proofs_wass_p>1}. In particular, we consider the following inductive hypothesis. Fix $p\geq 1$ and $n\geq 1$.
\begin{hypothesis}\label{hyp:induction_higherorder_tv}
Suppose $\lambdabar$ satisfies Assumption~\ref{ass:lambda_tv} for some $\delta_0>0$. Given $Z_{t_n}=\Zbar_{t_n}$ there exist a coupling $(Z_{t_{n+1}},\Zbar_{t_{n+1}})$ with respective marginals corresponding to $P_{\delta}(Z_{t_n},\cdot)$, $\overline{P}_{\delta}(\Zbar_{t_n},\cdot;\delta,p)$, and constants $A=A(T)$, $B=B(T)$ independent of $n$ such that for any $0<\delta\leq \delta_0$ 
	\begin{equation}\notag
	    \mathbb{P}_z(Z_{t_{n+1}}\neq\Zbar_{t_{n+1}}\rvert Z_{t_{n}}=\Zbar_{t_{n}}) \leq 
	    A \delta^{p+1}.
	\end{equation}
\end{hypothesis}
It is sufficient to show that the Inductive hypothesis holds and the statement of the Theorem follows by recursion in $n$ as done in \eqref{eq:recursion_tv}.
Observe that the case $p=1$ holds by the proof of Section \ref{sec:proofs_tv_p=1}. To obtain the result we use Coupling \ref{coup:general_higherorder} but with $\phibar=\varphi$, $\Fbar_i=F_i$, and replacing Inductive Hypothesis \ref{hyp:induction_higherorder} with Inductive Hypothesis \ref{hyp:induction_higherorder_tv}. Because the strategy is similar to that in Section \ref{sec:proofs_wass_p>1} we postpone the formal proof to Appendix \ref{sec:proof_higher_order_tv_appendix}.

\section{Proof of Theorem~\ref{thm:weakerror}}\label{sec:proof_weakerror_thm}
Recall in Section \ref{sec:proof_wass_p=1} we introduced a general framework which includes both Algorithms \ref{alg:basic_EPDMP} and \ref{alg:advanced_EPDMP}. We  now introduce some further notation. Let $\ptau{\delta,i}$ be a probability measure on $[0,\infty]$ which denotes the law of $\overline{\tau}$ for Algorithm $i$ with initial condition at $z$ for a time step of length $\delta$. Note that for Algorithm \ref{alg:basic_EPDMP} we have $\ptau{\delta,i}$ is a point measure with
\begin{equation}\label{eq:basicjumplaw}
\begin{alignedat}{3}
&\ptau{\delta,\ref{alg:basic_EPDMP}}(\{\delta\}) &&= 1-e^{-\int_0^{\delta}\lambdabar(z,s;\delta)\dd s},\\
&\ptau{\delta,\ref{alg:basic_EPDMP}}(\{+\infty\}) &&= e^{-\int_0^{\delta}\lambdabar(z,s;\delta)\dd s}.
\end{alignedat}
\end{equation}
On the other hand, in the case of Algorithm \ref{alg:advanced_EPDMP}  $\ptau{\delta,\ref{alg:advanced_EPDMP}}$ is admits a density which is given by
\begin{equation}\label{eq:advancedjumplaw}
    \ptau{\delta,\ref{alg:advanced_EPDMP}}(\dd s) = \overline{\lambda}(z,s;\delta) \exp\left(-\int_0^s \overline{\lambda}(z,r;\delta)\dd r\right)\dd s.
\end{equation}

\begin{proof}[Proof of Theorem~\ref{thm:weakerror}]
Fix $g\in \mathcal{G}_1$. Then by a telescoping sum we have
\begin{equation*}
     \mathbb{E}_z[g(\Zbar_{t_n})] -\mathbb{E}_z[g(Z_{t_n})]= \sum_{k=0}^{n-1}(\mathbb{E}_z[\cP_{t_n-t_{k+1}}g(\Zbar_{t_{k+1}})] -\mathbb{E}_z[\cP_{t_n-t_k}g(\Zbar_{t_k})]).  
\end{equation*}
For each $k\in \{0,\ldots, n-1\}$, set $f_k(y,s) = \cP_{t_n-t_{k}-s}g(y)$ then we have
\begin{equation*}
     \mathbb{E}_z[g(\Zbar_{t_n})] -\mathbb{E}_z[g(Z_{t_n})] = \sum_{k=0}^{n-1} \mathbb{E}_z[f_k(\Zbar_{t_{k+1}},\delta_{k+1})-f_k(\Zbar_{t_k},0)].  
\end{equation*}
By conditioning on $\Zbar_{t_k}$ it is sufficient to prove that
\begin{equation}\label{eq:boundkthterm}
    \lvert \mathbb{E}_z[f_k(\Zbar_{\delta_{k+1}},\delta_{k+1})]-f_k(z,0)\rvert \leq  Re^{-\omega(t_n-t_{k+1})} \lf_i(z)\delta_{k+1}^2.
\end{equation}
Here with an abuse of notation we have denoted as $\Zbar_{\delta_{k+1}}$ the approximation process with initial condition at $z$ and step size $\delta_{k+1}$.
Indeed if we have that \eqref{eq:boundkthterm} holds then by Assumption~\ref{ass:momentboundadvEulerWE} we have
\begin{align*}
    \lvert \mathbb{E}_z[g(Z_{t_n})] -\mathbb{E}_z[g(\Zbar_{t_n})]\rvert &\leq R\sum_{k=0}^{n-1}e^{-\omega(t_n-t_{k+1})} \delta_{k+1}^2\mathbb{E}_z[ \lf_i(\Zbar_{t_k})] \\
    &\leq RCS_n H_{i}(z) .  
\end{align*}
Which gives the desired result. It remains to show that \eqref{eq:boundkthterm} holds.

Using that the approximation process jumps according to $Q$ at a time determined by $\ptau{\delta_{k+1},i}$ we can evaluate the expectations
\begin{align*}
    & \mathbb{E}_z[f_k(\Zbar_{\delta_{k+1}},\delta_{k+1})]-f_k(z,0)]  = \\
    & \quad=\mathbb{E}_z[f_k(\Zbar_{\delta_{k+1}},\delta_{k+1})] - f_k(\varphi_{\delta_{k+1}}(z),\delta_{k+1})+f_k(\varphi_{\delta_{k+1}}(z),\delta_{k+1})-f_k(z,0)\\
    &\quad =\int_0^{\delta_{k+1}} Q(f_k(\varphi_{\delta_{k+1}-s}(\cdot),\delta_{k+1}))(\varphi_s(z)) - f_k(\varphi_{\delta_{k+1}}(z),\delta_{k+1}) \ptau{\delta_{k+1},i}(\dd s)\\
    &\qquad+f_k(\varphi_{\delta_{k+1}}(z),\delta_{k+1})-f_k(z,0).
\end{align*}
Recall $\ptau{\delta,\ref{alg:basic_EPDMP}}$ ($\ptau{\delta,\ref{alg:advanced_EPDMP}}$ respectively) is defined in \eqref{eq:basicjumplaw} (resp. \eqref{eq:advancedjumplaw}). 
Using the fundamental Theorem of calculus we can rewrite this as
\begin{align*}
    & \mathbb{E}_z[f_k(\Zbar_{\delta_{k+1}},\delta_{k+1})]-f_k(z,0)  =\\
    & \quad =\int_0^{\delta_{k+1}} Q(f_k(\varphi_{\delta_{k+1}-s}(\cdot),\delta_{k+1}))(\varphi_s(z)) - f_k(\varphi_{\delta_{k+1}}(z),\delta_{k+1}) \ptau{\delta_{k+1},i}(\dd s)\\
    & \qquad +\int_0^{\delta_{k+1}}\frac{\dd }{\dd r}f_k(\varphi_{r}(z),r)\dd r\\
    & \quad=\int_0^{\delta_{k+1}} Q(f_k(\varphi_{\delta_{k+1}-s}(\cdot),\delta_{k+1}))(\phi_s(z)) - f_k(\varphi_{\delta_{k+1}}(z),\delta_{k+1}) \ptau{\delta_{k+1},i}(ds)\\
    & \qquad +\int_0^{\delta_{k+1}}\langle \vf(\varphi_r(z)),\nabla f_k(\varphi_r(z),r)\rangle +(\partial_sf_k)(\varphi_{r}(z),r)\dd r.
\end{align*}
Note that $\partial_sf_k(y,s)=-\cL f_k(y,s)$
\begin{align*}
    & \mathbb{E}_z[f_k(\Zbar_{\delta_{k+1}},\delta_{k+1})]-f_k(z,0)  =\\
    & \quad = \int_0^{\delta_{k+1}} Q(f_k(\varphi_{\delta_{k+1}-s}(\cdot),\delta_{k+1}))(\varphi_s(z)) - f_k(\varphi_{\delta_{k+1}}(z),\delta_{k+1}) \ptau{\delta_{k+1},i}(\dd s)\\
    &\qquad+\int_0^{\delta_{k+1}}\langle \vf(\varphi_r(z)),\nabla f_k(\varphi_r(z),r)\rangle -\cL f_k(\varphi_{r}(z),r)\dd r.
\end{align*}
 Recall $\cL$ is given by \eqref{eq:genPDMPonekernel} so we can write the above as
 \begin{align*}
    & \mathbb{E}_z[f_k(\Zbar_{\delta_{k+1}},\delta_{k+1})]-f_k(z,0)  =\\
    & \quad = \int_0^{\delta_{k+1}} Q(f_k(\varphi_{\delta_{k+1}-s}(\cdot),\delta_{k+1}))(\varphi_s(z))- f_k(\varphi_{\delta_{k+1}}(z),\delta_{k+1}) \ptau{\delta_{k+1},i}(\dd s)\\
    &\qquad+\int_0^{\delta_{k+1}} -\lambda(\varphi_r(z))[Q(f_k(\cdot,r))(\varphi_{r}(z))-f_k(\varphi_r(z),r)]\dd r.
\end{align*}
We rewrite this as
\begin{align}
    &\mathbb{E}_z[f_k(\Zbar_{\delta_{k+1}},\delta_{k+1})]-f_k(z,0)=\nonumber\\  
    &=\!\!\int_0^{\delta_{k+1}}\!\!\!\!\! \!\!\!\!\!\left( Q(f_k(\varphi_{\delta_{k+1}-s}(\cdot),\delta_{k+1}))(\varphi_s(z)) \!-\! f_k(\varphi_{\delta_{k+1}}(z),\delta_{k+1})\right) (\ptau{\delta_{k+1},i}(\dd s)\!-\!\lambda(\varphi_s(z))\dd s) \notag\\
    &\quad-\int_0^{\delta_{k+1}} \lambda(\varphi_r(z))[Q(f_k(\cdot,r))(\varphi_{r}(z))-Q(f_k( \varphi_{\delta_{k+1}-r}(\cdot),\delta_{k+1}))(\varphi_r(z))]\dd r \label{eq:weexpansion}\\
    &\quad -\int_0^{\delta_{k+1}} \lambda(\varphi_r(z))[f_k(\varphi_{\delta_{k+1}}(z),\delta_{k+1})-f_k(\varphi_r(z),r)]\dd r.\notag
    \end{align}

We will divide the remainder of the proof into 3 steps:
\begin{enumerate}[\textbf{Step (\roman*):}]
    \item For this step we distinguish between Algorithm \ref{alg:basic_EPDMP} and \ref{alg:advanced_EPDMP}. Let
    \begin{equation*}
        h_s=  Q(f_k(\varphi_{\delta_{k+1}-s}(\cdot),\delta_{k+1}))(\varphi_s(z)) - f_k(\varphi_{\delta_{k+1}}(z),\delta_{k+1}).
    \end{equation*}
    Then we will show that there exists a constant $R>0$ such that for any $h\in C^1_b([0,\delta])$ (for Algorithm \ref{alg:advanced_EPDMP} we only need $h\in C_b([0,\delta])$) we have
    \begin{equation}\label{eq:WEstep1}
        \left\lvert\int_0^{\delta_{k+1}} h_s (\ptau{\delta_{k+1},i}(\dd s)-\lambda(\varphi_s(z))\dd s)\right\rvert \leq Re^{-\omega(t_n-t_{k+1})}\delta_{k+1}^2 \sup_{s,r\in [0,\delta_0]} K_i(z,s,r)
    \end{equation}
    where $K_i$ is as in Assumption~\ref{ass:momentboundadvEulerWE}.
    \item For any $z\in E, r\in [0,\delta_{k+1}]$ we have 
    \begin{equation}
    \begin{aligned}\label{eq:boundondifferenceoff}
    & \lvert f_k(\varphi_{\delta_{k+1}}(z),\delta_{k+1})-f_k(\varphi_r(z),r)\rvert 
    \leq\\
    & \quad \leq R_1\int_r^{\delta_{k+1}} e^{-\omega (t_n-t_k-s)}\lambda(\varphi_s(z))(Q\lf(\varphi_s(z)) +\lf(\varphi_s(z)) )\dd s.
\end{aligned}
\end{equation}
    \item For any $z\in E, r\in [0,\delta_{k+1}]$ we have 
    \begin{equation}
    \begin{aligned}\label{eq:boundondifferenceofQf}
    & \lvert Q(f_k(\varphi_{\delta_{k+1}-r}(\cdot),\delta_{k+1}))(z)-Q(f_k(\cdot,r))(z)\rvert 
    \leq\\
    & \quad \leq R_1\int_r^{\delta_{k+1}} e^{-\omega (t_n-t_k-s)}Q((Q\lf +\lf )\lambda)(\varphi_{s-r}(z))\dd s.
\end{aligned}
\end{equation}
\end{enumerate}

Equation~\eqref{eq:boundkthterm} follows from Step (i), (ii), (iii) and \eqref{eq:weexpansion}, as this gives
\begin{align*}
    & \mathbb{E}_z[f_k(\Zbar_{\delta_{k+1}},\delta_{k+1})]-f_k(z,0)  \leq Re^{-\omega(t_n-t_{k+1})}\delta_{k+1}^2 \sup_{s,r\in [0,\delta_0]}K_i(z,r,s)\\
    &\quad+\int_0^{\delta_{k+1}} \lambda(\varphi_r(z))R_1\int_r^{\delta_{k+1}} e^{-\omega (t_n-t_k-s)}Q(\lambda[Q\lf +\lf ])(\varphi_{s-r}(z))\dd s\dd r\\
    &\quad +\int_0^{\delta_{k+1}} \lambda(\varphi_r(z)) R_1\int_r^{\delta_{k+1}} e^{-\omega (t_n-t_k-s)}\lambda(\varphi_s(z))[Q\lf(\varphi_s(z)) +\lf(\varphi_s(z)) ]\dd s \dd r.
\end{align*}
Recall that $\lf_i(z,r,s)$ is given by \eqref{eq:largerlf}, then we have
\begin{align*}
    \lvert\mathbb{E}_z[f_k(\Zbar_{\delta_{k+1}},\delta_{k+1})]-f_k(z,0) \rvert &\leq Re^{-\omega(t_n-t_{k+1})}\delta_{k+1}^2 \sup_{s,r\in [0,\delta_0]}\lf_i(z,r,s) .
\end{align*}

\textbf{Proof of Step (i):} This step follows from Lemma \ref{lem:estlawforjump}. It remains to find a bound for $\lvert h_s\rvert$ and $\lvert \partial_s h_s\rvert$ for the case of Algorithm \ref{alg:basic_EPDMP}. 
By Assumption~\ref{ass:geoerg}
\begin{align}
   \lvert h_s\rvert &=\left\lvert  Q(f_k(\varphi_{\delta_{k+1}-s}(\cdot),\delta_{k+1}))(\varphi_s(z))  - f_k(\varphi_{\delta_{k+1}}(z),\delta_{k+1})\right\rvert  \notag \\
   &\leq \left\lvert  Q(f_k(\varphi_{\delta_{k+1}-s}(\cdot),\delta_{k+1}))(\varphi_s(z))  -\mu(f_k(\varphi_{\delta_{k+1}-s}(\cdot),\delta_{k+1}))\right\rvert \notag\\
   &\quad +\left\lvert \mu(f_k(\varphi_{\delta_{k+1}-s}(\cdot),\delta_{k+1}))- f_k\circ\varphi_{\delta_{k+1}-s}(\varphi_{s}(z),\delta_{k+1})\right\rvert \notag \\
   &\leq R_1e^{-\omega(t_n-t_{k+1})} [Q\lf(\varphi_s(z))+\lf(\varphi_s(z))].\label{eq:boundhs}
\end{align}

For Algorithm \ref{alg:basic_EPDMP} we also require to control $\lvert \partial_sh_s\rvert$ for which we require a bound on the derivative of $f$, for this case we use Assumption~\ref{ass:derivativedeccommutatoralongflow}. Note that
\begin{align*}
    \partial_sh_s &= \Big\lvert \langle \vf(\varphi_s(z)), \nabla_z( Q(f_k(\varphi_{\delta_{k+1}-s}(\cdot),\delta_{k+1})))(\varphi_s(z))\rangle  \\
    & \quad -Q(\langle \vf, \nabla_z( f_k( \varphi_{\delta_{k+1}-s}(\cdot),\delta_{k+1}))\rangle)(\varphi_s(z))) \Big\rvert\\
    &=[\vfd,Q](f_k(\varphi_{\delta_{k+1}-s}(\cdot),\delta_{k+1}))(\varphi_s(z)).
\end{align*}
Recall here we have defined the commutator in Section \ref{sec:notation} and we are denoting by $\vfd$ the differential operator corresponding to $\vf$. This term is bounded by Assumption~\ref{ass:derivativedeccommutatoralongflow} and we have
\begin{align}\label{eq:boundderivativehs}
    \lvert \partial_sh_s\rvert &\leq R_2 e^{-\omega (t_n-t_{k+1})}\lf(z).
\end{align}
Combining Lemma \ref{lem:estlawforjump} with \eqref{eq:boundhs} and \eqref{eq:boundderivativehs} we have that \eqref{eq:WEstep1} holds.

\textbf{Proof of Step (ii):} Observe that since $\partial_sf_k(y,s) +\langle \vf(y),\nabla f_k(y,s)\rangle = -\lambda(y)[Qf_k(y,s)-f_k(y,s)]$ we have
\begin{align*}
    &\lvert f_k(\varphi_{\delta_{k+1}}(z),\delta_{k+1})-f_k(\varphi_r(z),r)\rvert =\left\lvert \int_r^{\delta_{k+1}} \frac{\dd }{\dd s} f_k(\varphi_s(z),s)\dd s \right\rvert \\
    &= \left\lvert \int_r^{\delta_{k+1}} \lambda(\varphi_s(z))[Qf_k(\varphi_s(z),s)-f_k(\varphi_s(z),s)]\dd s\right\rvert\\
    &\leq  \int_r^{\delta_{k+1}} \lambda(\varphi_s(z))[\lvert Qf_k(\varphi_s(z),s)-\mu(Q(f_k(\cdot,s)))\rvert +\lvert \mu(f_k(\cdot,s))-f_k(\varphi_s(z),s)\rvert ]\dd s.
\end{align*}
We can bound this using Assumption~\ref{ass:geoerg} we obtain \eqref{eq:boundondifferenceoff}.

\textbf{Proof of Step (iii):} Applying \eqref{eq:boundondifferenceoff} with $z$ replaced by $\varphi_{-r}(y)$ and applying $Q$ we have \eqref{eq:boundondifferenceofQf}.

\end{proof}



\begin{appendix}
\section{Proofs of Section \ref{sec:Wass_bounds}}\label{sec:proofs_wasserstein_appendix}
\subsection{Proof of Theorem~\ref{thm:strong_error_pdmp}}\label{sec:proof_wass_theorem_appendix}
In this section we prove the lemmas that are used to prove Theorem~\ref{thm:strong_error_pdmp} in the case $p=1$.
In the proofs that follow we simplify the notation denoting the approximations as $\phibar_z(z)$, $\lambdabar_i(z,s)$, and $\Fbar_i(z,U)$, instead of $\phibar_z(z;\delta_{n+1})$,  $\lambdabar_i(z,s;\delta_{n+1})$, and $\Fbar_i(z,U;\delta_{n+1})$.
Before proving bounds on the events $E_{ij}$, let us state three simple lemmas which will be used multiple times in the proof. The proofs are omitted as they are a straightforward consequence of the assumptions.
\begin{lemma}\label{lemma:ineq_flow_strong_error}
Assumption~\ref{ass:lipschitz_phi} implies that for any $\delta_0>0$ there exists a constant $C'=C'(\delta_0)>0$ such that for any $z,z'\in E$ and $t\in(0,\delta_0)$ it holds that
\begin{equation}\label{eq:varphi_lipschitz}
\lVert \ic_t(z) - \ic_t(z')\rVert \leq C' \lVert z-z'\rVert.
\end{equation}
Moreover, for any $t\in(0,\delta_0)$ and any $z,z'\in E$ we have the alternative bound
\begin{equation*}
    \lVert \varphi_t(z)-\varphi_t(z')\rVert \leq (1+C C' t) \lVert z-z'\rVert.
\end{equation*}
\end{lemma}
\begin{lemma}\label{lem:flow_to_integrator}
Suppose Assumptions \ref{ass:lipschitz_phi} and \ref{ass:integrator} hold. Then for any $p\geq 1$, $s\geq 0$ and $z,z'\in E$ it holds that 
\begin{align*}
    \lVert \varphi_s(z)-\phibar_s(z')\rVert\leq C'\lVert z-z'\rVert + \Tilde{C} s^{p+1}.
\end{align*}
Moreover, using Lemma \ref{lemma:ineq_flow_strong_error} we can replace $C'$ with $1+CC'\delta$.
\end{lemma}
\begin{lemma}\label{lem:bounded_lambdas}
Under Assumptions \ref{ass:lambda_lipschitz}, \ref{ass:lambda_wasserstein}, and \ref{ass:boundedPDMP}, for any $t\geq 0$, $p\geq 1$ there exists a positive constant $L(t,z,p)$ such that 
\begin{equation}\notag
    \sup_{r\in[0,t], s\in[0,\delta_0]} \max\{ \lambda(Z_r),\overline{\lambda}(\Zbar_r,s;\delta_{n+1},p) \}\leq L(t,z,p) \quad \textnormal{a.s.}
\end{equation}
where in particular $z=Z_0=\Zbar_0$. Note if $p=1$ we write $L(t,z,1)=L(t,z)$.
\end{lemma}
We can now start showing a bound on event $E_{00}$, followed by the other events.
\begin{lemma}\label{Lemma:strong_error_E00}
Under Assumptions \ref{ass:lipschitz_phi} and \ref{ass:integrator}, it holds that 
\begin{equation}\notag
    \mathbb{E}_z [\lVert Z_{t_{n+1}}-\Zbar_{t_{n+1}}\rVert \mathbbm{1}_{E_{00}}] \leq (1+ \delta_{n+1} C C')\mathbb{E}_z [\lVert Z_{t_n}-\Zbar_{t_n}\rVert ] + \Tilde{C}\delta_{n+1}^2.
\end{equation}
\end{lemma}
\begin{proof}
    On $E_{00}$ we are interested only in the error introduced by the integrator $\phibar$. We have
    \begin{align*}
	    \mathbb{E}_z [\lVert Z_{t_{n+1}}-\Zbar_{t_{n+1}}\rVert \mathbbm{1}_{E_{00}}] & = \mathbb{E}_z \left[ \left\lVert \varphi_{\delta_{n+1}}(Z_{t_n})  - \phibar_{\delta_{n+1}}(\Zbar_{t_n}) \right\rVert \mathbbm{1}_{E_{00}} \right]\\
	    & \leq  \mathbb{E}_z \left[ \left\lVert \varphi_{\delta_{n+1}}(Z_{t_n})  - \varphi_{\delta_{n+1}}(\Zbar_{t_n}) \right\rVert \right]\\
	    & \quad + \mathbb{E}_z \left[ \left\lVert \varphi_{\delta_{n+1}}(\Zbar_{t_n})  - \phibar_{\delta_{n+1}}(\Zbar_{t_n}) \right\rVert \right] ,
	\end{align*}
    Then one can directly apply to the first term Assumption~\ref{ass:lipschitz_phi} and thus Lemma \ref{lemma:ineq_flow_strong_error}, together with the assumption that $\delta_n\leq \delta_0$, and to the second term Assumption~\ref{ass:integrator} to obtain the wanted result for $C'=C'(\delta_0)$.
\end{proof}

\begin{lemma}\label{lemma:strong_error_E11}
Under Assumption~\ref{ass:lipschitz_phi}, parts (b) and (c) of Assumption~\ref{ass:F_new}, as well as Assumptions \ref{ass:lambda_lipschitz}-\ref{ass:boundedPDMP}, it holds that 
\begin{equation}\notag
    \begin{aligned}
    \mathbb{E}_z [\lVert Z_{t_{n+1}}-\Zbar_{t_{n+1}}\rVert \mathbbm{1}_{E_{11}}] & \leq \delta^2_{n+1} \big(m K_2 + m(m-1) \Tilde{K}_2 + 2B(t_{n+1},z) (L(t_{n+1},z))^2)  \big) \\
    & \quad + \delta_{n+1} (m K_1 + m(m-1) \Tilde{K}_1) \mathbb{E}_z\left[ \lVert Z_{t_{n}}-\Zbar_{t_{n}} \rVert \right],
    \end{aligned}
\end{equation}
where $L(t_{n+1},z)$ was defined in Lemma \ref{lem:bounded_lambdas}, while
\begin{align*}
    & K_1 = D_2 (C')^2 L(t_{n+1},z), \\
    & K_2 = (D_2 \Tilde{C} C'+L(t_{n+1},z)(2D_3 + M_1 C') ) ,\\
    & \tilde{K}_1 = D_2 (C')^2 (L(t_{n+1},z))^2, \\
    & \tilde{K}_2 = (D_2 \Tilde{C} C'+ (L(t_{n+1},z))^2)(2D_3 + M_1 C' ).
\end{align*}
\end{lemma}
\begin{proof}
    	Let us first restrict to the event that $Z_s$ has only one event for $s\in(t_{n},t_{n+1}]$ and denote such event as $\overline{E}$. For any $i,j\in \{1,\ldots,m\}$, let $A_{ij}$ be the event that $Z_{t_n}$ jumps according to $F_i$ and $\Zbar_{t_n}$ jumps according to $\Fbar_j$ and no other jumps occur. Note that $\{A_{ij}\}_{i,j=1}^m$ is a partition of $E_{11}\cap \overline{E}$ so we may write
        \begin{equation*}
            \mathbb{E}_z [\lVert Z_{t_{n+1}}-\Zbar_{t_{n+1}}\rVert \mathbbm{1}_{E_{11}\cap \overline{E}}]  = \sum_{i,j=1}^m \mathbb{E}_z [\lVert Z_{t_{n+1}}-\Zbar_{t_{n+1}}\rVert \mathbbm{1}_{A_{ij}}].
        \end{equation*}
    
     Let us first consider event $A_{ii}$, i.e. the processes have a switch according to kernels $F_i$ and $\Fbar_i$. Considering Coupling \ref{coup:wass_gen_pdmp}, we first observe that $A_{ii}$ is an order $\delta_{n+1}$ event. This follows from the fact that in this case we require
	$$
	\tilde{U}^i_{n+1} \leq \min \left\{ 1-\exp\left(-\int_0^{\delta_{n+1}}
	\lambdabar_i(\Zbar_{t_{n}},s)\dd s\right),1-\exp\left( -\int_{0}^{\delta_{n+1}} \lambda_i(\varphi_s(Z_{t_{n}}))\dd s\right) \right\}.
	$$ 
	Therefore, using that $1-\exp(-z) \leq z$ we obtain
	\begin{equation}\label{eq:aux_E_11}
	    \mathbb{E}_z\left[ \mathbbm{1}_{A_{ii}}\rvert Z_{t_n},\Zbar_{t_n} \right] \leq \min \left\{\int_0^{\delta_{n+1}}
    	\lambdabar_i (\Zbar_{t_{n}},s)\dd s , \int_{0}^{\delta_{n+1}} \lambda_i(\varphi_s(Z_{t_{n}}))\dd s \right\}  \leq \delta_{n+1} L(t_{n+1},z),
	\end{equation}
	where $L(t_{n+1},z)<\infty$ was defined in Lemma \ref{lem:bounded_lambdas}.
	We can then separate the effects of the different approximations by the triangle inequality:
\reqnomode
	\begin{align*}
	    & \mathbb{E}_z  \left[ \lVert Z_{t_{n+1}} -\Zbar_{t_{n+1}} \rVert \mathbbm{1}_{A_{ii}} \right] = \\
	    & = \mathbb{E}_z[\lVert \varphi_{\delta_{n+1}-\tau_{n+1}}(F_i(\varphi_{\tau_{n+1}}(Z_{t_{n}}),U_{n+1})) - \phibar_{\delta_{n+1}-\taub}(\Fbar_{i}(\phibar_{\taub}(\Zbar_{t_n}),U_{n+1})) \rVert \mathbbm{1}_{A_{ii}} ] \\
	    & \leq \mathbb{E}_z\left[\lVert \varphi_{\delta_{n+1}-\tau_{n+1}}(F_i(\varphi_{\tau_{n+1}}(Z_{t_{n}}),U_{n+1})) - \varphi_{\delta_{n+1}-\taub} F_i(\varphi_{\taub}(Z_{t_{n}}),U_{n+1}) \rVert \mathbbm{1}_{A_{ii}} \right] \label{eq:E_11_term1}\tag{*} \\
	    &  + \mathbb{E}_z\left[\lVert \varphi_{\delta_{n+1}-\taub} (F_i(\varphi_{\taub}(Z_{t_{n}}),U_{n+1}))  - \varphi_{\delta_{n+1}-\taub}(F_i(\phibar_{\taub}(\Zbar_{t_{n}}),U_{n+1})) \rVert \mathbbm{1}_{A_{ii}} \right] \label{eq:E_11_term2}\tag{**} \\
	    &  +\mathbb{E}_z\left[\lVert \varphi_{\delta_{n+1}-\taub}(F_i(\phibar_{\taub}(\Zbar_{t_{n}}),U_{n+1})) - \varphi_{\delta_{n+1}-\taub}(\Fbar_i(\phibar_{\taub}(\Zbar_{t_{n}}),U_{n+1})) \rVert \mathbbm{1}_{A_{ii}} \right] \label{eq:E_11_term3}\tag{***} \\
	    &  +\mathbb{E}_z\left[\lVert  \varphi_{\delta_{n+1}-\taub}(\Fbar_i(\phibar_{\taub}(\Zbar_{t_{n}}),U_{n+1})) -\phibar_{\delta_{n+1}-\taub}(\Fbar_i(\phibar_{\taub}(\Zbar_{t_n}),U_{n+1})) \rVert \mathbbm{1}_{A_{ii}} \right]. \label{eq:E_11_term4}\tag{****}
	\end{align*}
	For term \eqref{eq:E_11_term1} we first compare both terms to $F_i(\varphi_{\delta_{n+1}}(Z_{t_n}),U_{n+1})$, and then we condition on all random variables apart from $U_{n+1}$ in order to apply Assumption~\ref{ass:F_new}(c):
	\begin{align*}
	    \eqref{eq:E_11_term1} & \leq \mathbb{E}_z\big[\lVert \varphi_{\delta_{n+1}-\tau_{n+1}}(F_i(\varphi_{\tau_{n+1}}(Z_{t_{n}}),U_{n+1})) - F_i(\varphi_{\delta_{n+1}}(Z_{t_n}),U_{n+1}) \rVert \mathbbm{1}_{A_{ii}} \big]  \\
	    & \quad + \mathbb{E}_z\left[\lVert F_i(\varphi_{\delta_{n+1}}(Z_{t_n}),U_{n+1}) - \varphi_{\delta_{n+1}-\taub} F_i(\varphi_{\taub}(Z_{t_{n}}),U_{n+1}) \rVert \mathbbm{1}_{A_{ii}} \right] & \\
	    & \leq 2 D_3 \delta_{n+1} \mathbb{P}_z(A_{ii})  \\
	    &\leq \delta_{n+1}^2 2 D_3  L(t_{n+1},z).
	\end{align*}
	In the last inequality we used the inequality derived in \eqref{eq:aux_E_11}.
    Term \eqref{eq:E_11_term2} can be bounded applying inequality \eqref{eq:varphi_lipschitz}, then conditioning on $Z_{t_n},\Zbar_{t_n},\taub$ and using Assumption~\ref{ass:F_new}(b), and finally applying Lemma \ref{lem:flow_to_integrator} and \eqref{eq:aux_E_11}:
    \begin{equation}\label{eq:aux_dblstar}
    \begin{aligned}
        \eqref{eq:E_11_term2} &\leq C' \mathbb{E}_z\left[\lVert  F_i(\varphi_{\taub}(Z_{t_{n}}),U_{n+1})  - F_i(\phibar_{\taub}(\Zbar_{t_{n}}),U_{n+1}) \rVert \mathbbm{1}_{A_{ii}} \right] \\
        & \leq C' D_2  \mathbb{E}_z\left[\lVert  \varphi_{\taub}(Z_{t_{n}}) - \phibar_{\taub}(\Zbar_{t_{n}}) \rVert \mathbbm{1}_{A_{ii}} \right] \\
        & \leq (C')^2 D_2 \mathbb{E}_z[\lVert Z_{t_{n}}-\Zbar_{t_{n}}\rVert \mathbbm{1}_{A_{ii}} ] + \delta_{n+1}^2 C' D_2 \Tilde{C}\\
        & \leq (C')^2 D_2 L(t_{n+1},z) \delta_{n+1} \mathbb{E}_z[\lVert Z_{t_{n}}-\Zbar_{t_{n}}\rVert] + \delta_{n+1}^2 C' D_2 \Tilde{C}.
    \end{aligned}
    \end{equation}
    Term \eqref{eq:E_11_term3} is estimated by inequality \eqref{eq:varphi_lipschitz}, then again conditioning on $Z_{t_n},\Zbar_{t_n},\taub$ and applying Assumption~\ref{ass:approx_jump_kernel}, and finally using \eqref{eq:aux_E_11}:
    \begin{equation}\label{eq:aux_trplstar}
    \begin{aligned}
        \eqref{eq:E_11_term3} & \leq C'\mathbb{E}_z\left[\lVert F_i(\phibar_{\taub}(\Zbar_{t_{n}}),U_{n+1}) - \Fbar_i(\phibar_{\taub}(\Zbar_{t_{n}}),U_{n+1}) \rVert \mathbbm{1}_{A_{ii}} \right]\\
        & \leq C'M_1 \delta_{n+1} \mathbb{P}_z(A_{ii})\\
        & \leq \delta_{n+1}^2 C' M_1 L(t_{n+1},z).
    \end{aligned}
    \end{equation}
	Term \eqref{eq:E_11_term4} is bounded using Assumption~\ref{ass:integrator} and bounding by $1$ the probability of $A_{ii}$:
	\begin{align*}
	    \eqref{eq:E_11_term4} &\leq \tilde{C}\delta_{n+1}^2\mathbb{P}_z(A_{ii}) \leq \tilde{C}\delta_{n+1}^2. 
	\end{align*}
	Putting together terms \eqref{eq:E_11_term1}, \eqref{eq:E_11_term2}, \eqref{eq:E_11_term3}, \eqref{eq:E_11_term4} we obtain the following bound on event $A_{ii}$:
	\leqnomode
	\begin{align}
	    \mathbb{E}_z  \left[ \lVert Z_{t_{n+1}} -\Zbar_{t_{n+1}} \rVert \mathbbm{1}_{A_{ii}} \right] &\leq \delta^2_{n+1} K_2 + \delta_{n+1} K_1 \mathbb{E}_z\left[ \lVert Z_{t_{n}}-\Zbar_{t_{n}} \rVert \right]
	\end{align}
    where $K_1, K_2$ are as in the statement of the lemma.
    
    Now consider event $A_{ij}$ for $i\neq j$. In this case we take advantage of independence of $\tilde{U}_{n+1}^i$ and $\tilde{U}_{n+1}^j$ to conclude that 
    \begin{align}
        &\mathbb{E}_z[\mathbbm{1}_{A_{ij}}\rvert Z_{t_n},\Zbar_{t_n}]  \leq \notag \\
        &\quad \leq \left(1-\exp\left(-\int_0^{\delta_{n+1}} \lambda_i(\varphi_s(Z_{t_n}))\dd s \right)\right)  \left(1-\exp\left(-\int_0^{\delta_{n+1}} \lambdabar_j(\Zbar_{t_n},s)\dd s \right)\right) \notag\\
        & \quad\leq \int_0^{\delta_{n+1}} \lambda_i(\varphi_s(Z_{t_n}))\dd s \int_0^{\delta_{n+1}} \lambdabar_j(\Zbar_{t_n},s)\dd s \notag\\
        & \quad\leq \delta_{n+1}^2 (L(t_{n+1},z))^2 .\label{eq:prob_A_ij}
    \end{align}
    Then we can use the decomposition
    \reqnomode
    \begin{align}
        &\mathbb{E}_z [\lVert Z_{t_{n+1}}-\Zbar_{t_{n+1}}\rVert \mathbbm{1}_{A_{ij}}] = \notag\\
        & = \mathbb{E}_z [\lVert \varphi_{\delta_{n+1}-\tau_{n+1}}(F_i(\varphi_{\tau_{n+1}}(Z_{t_n}),U_{n+1}))- \phibar_{\delta_{n+1}-\overline{\tau}_{n+1}}(\Fbar_j(\phibar_{\overline{\tau}_{n+1}}(\Zbar_{t_n}),U_{n+1}))) \rVert \mathbbm{1}_{A_{ij}}] \notag\\
        & \leq \mathbb{E}_z [\lVert \varphi_{\delta_{n+1}-\tau_{n+1}}(F_i(\varphi_{\tau_{n+1}}(Z_{t_n}),U_{n+1}))) - F_i(\varphi_{\delta_{n+1}}(Z_{t_n}),U_{n+1})))\rVert \mathbbm{1}_{A_{ij}}] \tag{$\dagger$}\label{eq:A_ij_1}\\
        & +\mathbb{E}_z [\lVert  F_i(\varphi_{\delta_{n+1}}(Z_{t_n}),U_{n+1}))) - F_j(\varphi_{\delta_{n+1}}(Z_{t_n}),U_{n+1})))\rVert \mathbbm{1}_{A_{ij}}] \tag{$\ddag$}\label{eq:A_ij_2}\\
        &+\mathbb{E}_z [\lVert  F_j(\varphi_{\delta_{n+1}}(Z_{t_n}),U_{n+1}))) - \varphi_{\delta_{n+1}-\taub}(F_j(\varphi_{\taub}(Z_{t_n}),U_{n+1})))\rVert \mathbbm{1}_{A_{ij}}] \tag{$\ddag\dagger$} \label{eq:A_ij_3}\\
        & +\mathbb{E}_z [\lVert  \varphi_{\delta_{n+1}-\taub}(F_j(\varphi_{\taub}(Z_{t_n}),U_{n+1})))-\phibar_{\delta_{n+1}-\taub}(\overline{F}_j(\phibar_{\taub}(\Zbar_{t_n}),U_{n+1})))\rVert \mathbbm{1}_{A_{ij}}] \tag{$\ddag\ddag$} \label{eq:A_ij_4}
    \end{align}
    To bound \eqref{eq:A_ij_1} and \eqref{eq:A_ij_3} we use Assumption~\ref{ass:F_new}(c), while for \eqref{eq:A_ij_2} we add and subtract $\varphi_{\delta_{n+1}}(Z_{t_n})$ and use Assumption~\ref{ass:F_new}(a), and for \eqref{eq:A_ij_4} we use a similar argument to the $A_{ii}$ case. Combining this with the bound in \eqref{eq:prob_A_ij} we obtain
    \begin{align*}
        \mathbb{E}_z [\lVert Z_{t_{n+1}}-\Zbar_{t_{n+1}}\rVert \mathbbm{1}_{A_{ij}}] &\leq ((L(t_{n+1},z))^2(2D_3+D_1)  + \Tilde{K_2}) \delta_{n+1}^2 \\
        & \quad + \delta_{n+1} \tilde{K}_1 \mathbb{E}_z\left[ \lVert Z_{t_{n}}-\Zbar_{t_{n}} \rVert \right]
    \end{align*}
    where $\tilde{K}_1, \tilde{K}_2$ are as in the statement of the lemma. 

	
	
	Let us finally consider $\overline{E}^c$, i.e. the case in which $Z_t$ has two or more jumps. The probability this event is given by
	\begin{equation}\label{eq:bound_2events}
	    \begin{aligned}
	        \mathbb{P}_z(\overline{E}^c) & = \mathbb{E}_z\Bigg[\int_0^{\delta_{n+1}} \left( 1 - \exp\left( - \int_0^{\delta_{n+1}-t} \lambda(\varphi_s(F_{I_{n+1}}(\varphi_t(Z_{t_n}),U_{n+1})))\dd s \right)\right)  \\
	        & \qquad \lambda(\varphi_t(Z_{t_n})) \exp\left(-\int_0^t \lambda(\varphi_r(Z_{t_n}))\dd r \right) \dd t\Bigg] \\
	        & \leq \mathbb{E}_z\Bigg[\int_0^{\delta_{n+1}} \left( \int_0^{\delta_{n+1}-t} \lambda(\varphi_s(F_{I_{n+1}}(\varphi_t(z),U_{n+1})))\dd s \right) \lambda(\varphi_t(z)) \dd t \Bigg]\\
	        & \leq \delta_{n+1}^2 (L(t_{n+1},z))^2. 
	    \end{aligned}
	\end{equation}
	Then we can bound the norms $\lVert Z _{t_{n+1}}\rVert$ and $\lVert \Zbar_{t_{n+1}}\rVert$ by Assumption~\ref{ass:boundedPDMP} to obtain
	\begin{equation}\label{eq:bound_twoormoreevents}
	\begin{aligned}
	    \mathbb{E}_z\left[\lVert Z_{t_{n+1}}-\Zbar_{t_{n+1}} \rVert \mathbbm{1}_{E_{11}\cap\overline{E}^c} \right] & \leq 2 B(t_{n+1},z) \mathbb{E}_z\left[\mathbbm{1}_{\overline{E}^c}\right] \\
	    & \leq 2\delta_{n+1}^2 B(t_{n+1},z) (L(t_{n+1},z))^2 .
	\end{aligned}
	\end{equation}
	Combining the bounds on $\overline{E}$ and $\overline{E}^c$ we obtain the statement of the lemma.
\end{proof}

\begin{lemma}\label{lemma:strong_error_E10}
Under Assumptions \ref{ass:lipschitz_phi}, \ref{ass:F_new}(a), \ref{ass:lambda_lipschitz}-\ref{ass:boundedPDMP}, it holds that
\begin{equation}\notag
    \begin{aligned}
        \mathbb{E}_z\big[\lVert  Z_{t_{n+1}}-\Zbar_{t_{n+1}}\rVert &  \mathbbm{1}_{E_{10}} \big] \leq \delta_{n+1}\, m (C')^2 ( D_1 D_4 + 2D_2L(t_{n+1},z))\, \mathbb{E}_z \left[  \lVert Z_{t_{n}}-\Zbar_{t_{n}}\rVert \right] \\
        & \quad+ \delta_{n+1}^2 (\tilde{C}+  C'(D_2\tilde{C}+mD_1M_2(t_n,z) + 2mM_1 L(t_{n+1},z) )).
    \end{aligned}
\end{equation}
\end{lemma}
\begin{proof}
    Recall that $E_{10}$ is the event in which there are no switches for $Z_s$ for $s\in(t_{n},t_{n+1}]$ and there is one event for the approximation.
    Taking advantage of the coupling of the two processes as described in Coupling \ref{coup:wass_gen_pdmp} we find that $E_{10}$ takes place as long as for some $i$ 
	\begin{equation}\notag
	    \tilde{U}^i_{n+1} \in \Bigg(1-\exp\left(-\int_{0}^{\delta_{n+1}} \lambda_i(\varphi_s(Z_{t_{n}}))\dd s\right),1-\exp\left(-\int_0^{\delta_{n+1}} \lambdabar_i(\Zbar_{t_{n}},s)\dd s\right) \Bigg].
	\end{equation}
	Then we can estimate the probability of this event as follows:
	\begin{align*}
	    \mathbb{E}_z[\mathbbm{1}_{E_{10}}\rvert Z_{t_n},\Zbar_{t_n} ] 
	    & \leq \sum_{i=1}^m \left\lvert \exp\left(-\int_{0}^{\delta_{n+1}} \lambda_i(\varphi_s(Z_{t_{n}}))\dd s\right) - \exp\left(-\int_0^{\delta_{n+1}} \lambdabar_i(\Zbar_{t_{n}},s)\dd s\right)   \right\rvert  \\
	    & \leq \sum_{i=1}^m  \int_{0}^{\delta_{n+1}}\left\lvert  \lambda_i(\varphi_s(Z_{t_{n}})) - \lambdabar_i(\Zbar_{t_{n}},s)  \right\rvert\dd s  
	\end{align*}
	where we used that $\exp(-z)$ is $1$-Lipschitz for $z\geq 0$. Then we find bounds for $\mathbb{E}_z[\mathbbm{1}_{E_{10}} ] $ and $\mathbb{E}_z[\mathbbm{1}_{E_{10}}\rvert Z_{t_n},\Zbar_{t_n} ] $ respectively. For the first case we use the triangle inequality, followed by observing that $\lambda$ and $\varphi_s$ are Lipschitz by the inequality shown in \eqref{eq:varphi_lipschitz}  and then Assumption~\ref{ass:lambda_wasserstein}:
	\begin{equation}\label{eq:E_10_prob}
	  \begin{aligned}
        \mathbb{E}_z[\mathbbm{1}_{E_{10}}] &\leq \sum_{i=1}^m \mathbb{E}_z \Bigg[   \int_{0}^{\delta_{n+1}}\left\lvert  \lambda_i(\varphi_s(Z_{t_{n}})) - \lambda_i(\varphi_s(\Zbar_{t_{n}}))  \right\rvert\dd s \\
        & \quad + \int_{0}^{\delta_{n+1}}\left\lvert  \lambda_i(\varphi_s(\Zbar_{t_{n}})) - \lambdabar_i(\Zbar_{t_{n}},s)  \right\rvert\dd s  \Bigg]\\ 
        & \leq \sum_{i=1}^m \mathbb{E}_z \left[   \int_{0}^{\delta_{n+1}} D_4 C' \lVert  Z_{t_{n}} - \Zbar_{t_{n}} \rVert\dd s + \int_{0}^{\delta_{n+1}} \delta_{n+1} \overline{M}_2(\Zbar_{t_n})\dd s  \right] \\
	    & \leq \delta_{n+1} m  D_4 C' \mathbb{E}_z[\lVert  Z_{t_{n}} - \Zbar_{t_{n}} \rVert] + m \delta^2_{n+1}  \mathbb{E}_z [\overline{M}_2(\Zbar_{t_n})] \\
	    & \leq \delta_{n+1}  m D_4 C' \mathbb{E}_z[\lVert  Z_{t_{n}} - \Zbar_{t_{n}} \rVert] + \delta^2_{n+1} m M_2(t_n,z),
	  \end{aligned}
	\end{equation}
	where in the last inequality we used again Assumption~\ref{ass:lambda_wasserstein}. Alternatively, we can bound the switching rates by Lemma \ref{lem:bounded_lambdas}:
	\begin{equation}\label{eq:E_10_prob_bis}
	  \begin{aligned}
        \mathbb{E}_z[\mathbbm{1}_{E_{10}}\rvert Z_{t_n},\Zbar_{t_n}] &\leq  \sum_{i=1}^m   \int_{0}^{\delta_{n+1}} \big( \lvert  \lambda_i(\varphi_s(Z_{t_{n}})) \rvert + \lvert \lambdabar_i(\Zbar_{t_{n}},s)  \rvert )\dd s \\
	    & \leq 2 m\delta_{n+1}  L(t_{n+1},z) ,
	  \end{aligned}
	\end{equation}
    
    Let us now focus on bounding the distance between the two processes. On event $E_{10}$ we have $Z_{t_{n+1}}=\varphi_{\delta_{n+1}}(Z_{t_n})$, while $\Zbar_{t_{n+1}} = \phibar_{\delta_{n+1}-\taub} (\Fbar(\phibar_{\taub}(\Zbar_{t_{n}}),U_{n+1}))$ where $\taub$ is the time of the event for the approximation. By triangle inequality we can decompose the distance in the following terms
    \reqnomode
	\begin{align*}
	    & \mathbb{E}_z [\lVert Z_{t_{n+1}}-\Zbar_{t_{n+1}}\rVert \mathbbm{1}_{E_{10}}] 
	     =\\
	     & =\mathbb{E}_z \left[\lVert \varphi_{\delta_{n+1}}(Z_{t_{n}})-\phibar_{\delta_{n+1}-\taub} (\Fbar_{\bar{I}_{n+1}}(\phibar_{\taub}(\Zbar_{t_{n}}),U_{n+1})) \rVert \mathbbm{1}_{E_{10}} \right]\\
	    & \leq \mathbb{E}_z \left[  \lVert \varphi_{\delta_{n+1}}(Z_{t_{n}})-  \varphi_{\delta_{n+1}-\taub} (F_{\bar{I}_{n+1}}(\varphi_\taub(Z_{t_n}),U_{n+1}))
	    \rVert \mathbbm{1}_{E_{10}}\right] \label{eq:E_10_term1}\tag{*} \\
	    & + \mathbb{E}_z \Big[ \lVert \varphi_{\delta_{n+1}-\taub} (F_{\bar{I}_{n+1}}(\varphi_\taub(Z_{t_n}),U_{n+1})) \label{eq:E_10_term2}\tag{**} \\
	    & \quad - \varphi_{\delta_{n+1}-\taub}(F_{\bar{I}_{n+1}}(\phibar_\taub(\Zbar_{t_n}),U_{n+1}))
	    \rVert  \mathbbm{1}_{E_{10}} \Big]  \\
	    & + \mathbb{E}_z \Big[ \lVert \varphi_{\delta_{n+1}-\taub}(F_{\bar{I}_{n+1}}(\phibar_\taub(\Zbar_{t_n}),U_{n+1})) \label{eq:E_10_term3}\tag{***} \\
	    & \quad - \phibar_{\delta_{n+1}-\taub} (\Fbar_{\bar{I}_{n+1}}(\phibar_{\taub}(\Zbar_{t_{n}}),U_{n+1}))
	    \rVert  \mathbbm{1}_{E_{10}} \Big]\\
	    & = \eqref{eq:E_10_term1}+\eqref{eq:E_10_term2}+\eqref{eq:E_10_term3}. 
	\end{align*}
	In order to estimate term \eqref{eq:E_10_term1} we apply inequality \eqref{eq:varphi_lipschitz}, then Assumption~\ref{ass:F_new}(a) by conditioning on $Z_{t_n},\taub$, and then we apply \eqref{eq:E_10_prob}:
	\begin{equation}\label{eq:aux_star}
	\begin{aligned}
	    \eqref{eq:E_10_term1} &\leq C' \mathbb{E}_z [\lVert \varphi_\taub(Z_{t_n})-F_{\bar{I}_{n+1}}(\varphi_\taub(Z_{t_n}),U_{n+1}) \rVert\mathbbm{1}_{E_{10}}] \\
	    & \leq  C' D_1 \mathbb{E}_z[\mathbbm{1}_{E_{10}}] \\
	    & \leq \delta_{n+1} m  D_4 (C')^2 D_1 \mathbb{E}_z[\lVert  Z_{t_{n}} - \Zbar_{t_{n}} \rVert] + \delta^2_{n+1} m  C' D_1 M_2(t_n,z).
	\end{aligned}
	\end{equation}
	For term \eqref{eq:E_10_term2} we use the same reasoning of \eqref{eq:aux_dblstar} together with the estimate \eqref{eq:E_10_prob_bis}:
	\begin{align*}
	    \eqref{eq:E_10_term2} &
	    \leq   \delta_{n+1} m 2 L(t_{n+1},z)(C')^2 D_2 \mathbb{E}_z[\lVert Z_{t_n}-\Zbar_{t_n}\rVert ] +   C' D_2 \tilde{C} \delta_{n+1}^2.
	\end{align*}
	Then for term \eqref{eq:E_10_term3} we follow the reasoning in \eqref{eq:aux_trplstar} and apply estimate \eqref{eq:E_10_prob_bis} to obtain
	\begin{align}
	    \eqref{eq:E_10_term3} &
	    \leq \delta_{n+1}^2 \big(2 m C'M_1 L(t_{n+1},z) + \Tilde{C} \big).\nonumber
	\end{align}
	
	The statement of the lemma follows then by combining estimates \eqref{eq:E_10_term1}, \eqref{eq:E_10_term2}, \eqref{eq:E_10_term3}.
\end{proof}

\begin{lemma}\label{lemma:strong_error_E01}
Under Assumptions \ref{ass:lipschitz_phi}, \ref{ass:F_new}(a), \ref{ass:lambda_lipschitz}, \ref{ass:integrator}, \ref{ass:lambda_wasserstein}, \ref{ass:boundedPDMP}, it holds that 
\begin{equation}\notag
    \begin{aligned}
        \mathbb{E}_z\left[\lVert Z_{t_{n+1}}-\Zbar_{t_{n+1}}\rVert \mathbbm{1}_{E_{01}} \right] &\leq  \delta_{n+1}\,m C'\big(L(t_{n+1},z)+ C' D_1 D_4\big) \,\mathbb{E}_z \left[ \lVert Z_{t_{n}}-\Zbar_{t_{n}}\rVert\right] \\
        & \quad +\delta_{n+1}^2 \big(2B(t_{n+1},z) (L(t_{n+1},z))^2+ m C' D_1  M_2(t_n,z)+ \Tilde{C}\big).
    \end{aligned}
\end{equation}
\end{lemma}
\begin{proof}
    Recall that $E_{01}$ is the event in which for $s\in [t_{n},t_{n+1})$ there are no switches for $\Zbar_s$, while there is at least one event for $Z_s$. Similarly to the proof of Lemma \ref{lemma:strong_error_E11}, let us denote as $\overline{E}$ the event in which there is exactly one event for $Z_s$ in the current time interval. On $\overline{E}^c$ we can use the bound \eqref{eq:bound_twoormoreevents}. On $\overline{E}$
    \reqnomode
	\begin{align*}
	    & \mathbb{E}_z [\lVert Z_{t_{n+1}}-\Zbar_{t_{n+1}}\rVert \mathbbm{1}_{E_{01}\cap\overline{E}}] 
	     =\\
	     &=\mathbb{E}_z \left[ \left\lVert \varphi_{\delta_{n+1}-\tau_{n+1}}(F_{I_{n+1}}(\varphi_{\tau_{n+1}}(Z_{t_{n}}),U_{n+1}))-\phibar_{\delta_{n+1}} (\Zbar_{t_{n}})  \right\rVert \mathbbm{1}_{E_{01}\cap\overline{E}} \right] \\
	    & \leq \mathbb{E}_z \left[ \left\lVert \varphi_{\delta_{n+1}-\tau_{n+1}}(F_{I_{n+1}}(\varphi_{\tau_{n+1}}(Z_{t_{n}}),U_{n+1}))-\varphi_{\delta_{n+1}}(Z_{t_{n}})\right\rVert \mathbbm{1}_{E_{01}\cap\overline{E}} \right] \label{eq:E_01_term1}\tag{*}\\
	    & \quad + \mathbb{E}_z \left[ \rVert \varphi_{\delta_{n+1}}(Z_{t_{n}})-
	    \phibar_{\delta_{n+1}} (\Zbar_{t_{n}})  \rVert \mathbbm{1}_{E_{01}\cap\overline{E}} \right] \label{eq:E_01_term2}\tag{**}.
	\end{align*}
	In order to find an estimate for term \eqref{eq:E_01_term1} observe that the probability of  $E_{01}$ can be estimated similarly to what done for the probability of $E_{10}$. Then following the reasoning in \eqref{eq:aux_star} we obtain
	\begin{align*}
	    \eqref{eq:E_01_term1} & \leq \delta_{n+1} m (C')^2 D_1 D_4 \mathbb{E}_z[\lVert  Z_{t_{n}} - \Zbar_{t_{n}} \rVert] + \delta^2_{n+1} m C' D_1 M_2(t_n,z).
	\end{align*}
	Similarly, for term \eqref{eq:E_01_term2} it is sufficient to apply Lemma \ref{lem:flow_to_integrator} and then to bound $\mathbb{E}[\mathbbm{1}_{E_{01}}]$ by the probability that the continuous process has a random event:
	\begin{equation*}
	    \eqref{eq:E_01_term2}  \leq \delta_{n+1} m C' L(t_{n+1},z) \mathbb{E}_z \left[ \lVert  Z_{t_{n}}- \Zbar_{t_{n}} \rVert \right] + \Tilde{C}\delta_{n+1}^2.
	\end{equation*}
\end{proof}

\subsection{Proof of Corollary~\ref{cor:strong_error_pdmp}}\label{sec:proof_corollary_wass}

\begin{proof}
    We only need to show that Assumptions \ref{ass:F_new} and \ref{ass:boundedPDMP} are verified in this setting under Assumption~\ref{ass:discr_velocities_wass}. Clearly the process moves with bounded velocity, and thus \ref{ass:boundedPDMP} holds. Then let us focus on verifying Assumption~\ref{ass:F_new} and consider the $\ell^1$-norm. For part (a) it is clear that for $z=(x,v)$
    \begin{equation*}
        \mathbb{E} [\lVert z - F_i(z,\tilde{U})\rVert] =\mathbb{E} [\lVert v - F_i^v((x,v),\tilde{U})\rVert] \leq V_{max}.
    \end{equation*}
    Then consider part (b). For $z'=(y,w)$
    \begin{align*}
        &\mathbb{E} [\lVert F_i(z,\tilde{U})-F_i(z',\tilde{U})\rVert]  \leq \lVert x-y \rVert + \mathbb{E} [\lVert F_i^v((x,v),\tilde{U})-F_i^v((y,w),\tilde{U})\rVert]  \\
        & \quad\leq \lVert x-y \rVert + \frac{\lVert v-w\rVert}{V_{min}} V_{max} + \mathbb{E} [\lVert F_i^v((x,w),\tilde{U})-F_i^v((y,w),\tilde{U})\rVert]   \\
        & \quad\leq \max\left\{\frac{V_{max}}{V_{min}},1+D\right\} \lVert z-z'\rVert.
    \end{align*}
    In the second inequality we used the triangle inequality and that $\lVert v-w\rVert\geq V_{min}$, while in the last inequality we bounded the rightmost term by Assumption~\ref{ass:discr_velocities_wass}. 
    Let us focus on part (c). For the position part we have for $s\in[0,\delta]$ and $z=(x,v)$
    \begin{align*}
        \mathbb{E} &[\lVert \varphi_{\delta-s}(\varphi_s(z),F_i^v((\varphi_s(z),v),\tilde{U})) - \varphi_\delta(z) \rVert ]  = \\
        & = \mathbb{E}\Big[\Big\lVert x+\int_0^s \vf(v)\dd s + \int_0^{\delta-s}\vf(F_i^v((\varphi_s(z),v),U))\dd r - x - \int_0^\delta \vf(v)\dd s \Big\rVert \Big] \\
        & \leq \mathbb{E}[ \lVert  s\vf(v) + (\delta-s) \vf(F_i^v((\varphi_s(z),v),U)) -\delta \vf(v)\rVert ] \\
        & \leq \mathbb{E}[ \lVert (\delta-s) (\vf(F_i^v((\varphi_s(z),v),U)) - \vf(v) ) \rVert ] \\
        & \leq \delta C \mathbb{E} [\lVert F_i^v((\varphi_s(z),v),U)-v\rVert ] \\
        & \leq \delta C V_{max}.
    \end{align*}
    On the other hand, for the velocity part we obtain using Assumptions \ref{ass:discr_velocities_wass} and \ref{ass:lipschitz_phi}
    \begin{align*}
        \mathbb{E}[\lVert F_i^v((\varphi_s(z),v),\tilde{U}) - F_i^v((\varphi_\delta(z),v),\tilde{U})\rVert] & \leq D\lVert \varphi_s(z)-\varphi_\delta(z)\rVert \\
        & \leq D C' \lVert z-\varphi_{\delta-s}(z)\rVert \\
        & \leq \delta D (C')^2.
    \end{align*}
    Therefore part (c) of Assumption~\ref{ass:F_new} holds with $D_3 = D (C')^2 + C V_{max}$.
\end{proof}

\subsection{Proof of Proposition~\ref{prop:wass_rhmc}}\label{sec:proof_wass_rhmc}
\begin{proof}
In the proof of Theorem~\ref{thm:strong_error_pdmp} boundedness of the PDMP is used only in the case $p=1$ to deal with the event in which the PDMP has two or more jumps in the same time step (see Lemmas \ref{lemma:strong_error_E11} and \ref{lemma:strong_error_E10}, in particular Equation~\eqref{eq:bound_twoormoreevents}). Then it is sufficient to show that a similar bound holds also under Assumption~\ref{ass:bdd_lambda} instead of Assumption~\ref{ass:boundedPDMP}. Let $p=1$ and consider the case of Lemma \ref{lemma:strong_error_E10}, i.e. restricting to event $E_{10} \cap \overline{E}$, which is the event in which the continuous time process has two or more jumps, while the approximation has zero jumps. Then we want to bound
\begin{align*}
    &\mathbb{E}_z[\lVert Z_{t_{n+1}}-\Zbar_{t_{n+1}}\rVert \mathbbm{1}_{E_{10} \cap \overline{E}}]= \\
    &  =\! \sum_{\ell\geq 2} \!\int_{A_\ell} \!\!\!\!\mathbb{E}_z \big[\lVert \varphi_{s_\ell}\!\!\circ\! F_{I_{n+1}^{\ell}}\!(\cdot,U_{\ell-1})\!\circ\! \varphi_{s_{\ell-1}}\!\!\circ\! \dots\!\circ\! F_{I_{n+1}^{1}}(\cdot,U_{0}) \!\circ\!\varphi_{s_0}\!(Z_{t_n})\!-\!  \phibar_{\delta_{n+1}}\!(\Zbar_{t_n})\rVert p_{Z_{t_n}}(\dd s)\big],
\end{align*}
where $s_0,\dots,s_{\ell-1}$ are the interarrival times of the random jumps of $Z_t$, $I_{n+1}^\ell$ denote the index of the $\ell$-th jump to occur between time $t_n$ and $t_{n+1}$, $s_\ell=\delta_{n+1}-\sum_{i=1}^{\ell-1} s_i$, $U_0,\dots,U_{\ell-1}\stackrel{iid}{\sim}\nuU$,
$$A_\ell=\left\{s=(s_0,\dots,s_\ell): \sum_{i=1}^\ell s_i =\delta_{n+1}, s_i>0\right\},$$
and $p_{Z_{t_n}}(\dd s)$ is the law of the interarrival times,.
Then we have
\begin{align*}
    \mathbb{E}_z[\lVert Z_{t_{n+1}}-\Zbar_{t_{n+1}}\rVert \mathbbm{1}_{E_{10} \cap \overline{E}}] & \leq \mathbb{E} \Bigg[\sum_{\ell\geq 2} \int_{A_\ell} \mathbb{E}_z \Big[\Big(\lVert \varphi_{s_\ell}\circ F_{I_{n+1}^\ell}(\cdot,U_{\ell-1})\circ\dots\circ \varphi_{s_0}(Z_{t_n}) \rVert\\
    & \qquad +\lVert \phibar_{\delta_{n+1}}(\Zbar_{t_n})\rVert) \Big\rvert Z_{t_n}\Big] p_{Z_{t_n}}(\dd s) \Bigg]
\end{align*}
Now we use \eqref{eq:varphi_lipschitz} to conclude that $\varphi_s$ has linear growth for some constant $L$, and so $\lVert \varphi_s(z)\rVert \leq L(\lVert z\rVert +1)$. It follows that 
\begin{align*}
    & \mathbb{E}_z \left[\lVert \varphi_{s_\ell}\circ F_{I_{n+1}^\ell}(\cdot,U_{\ell-1})\circ \varphi_{s_{\ell-1}}\circ F_{I_{n+1}^{\ell-1}}(\cdot,U_{\ell-2}) \circ\dots\circ \varphi_{s_0}(Z_{t_n}) \rVert \rvert Z_{t_n} \right]\leq \\
    & \qquad \leq L \mathbb{E}_z \left[\left(\lVert F_{I_{n+1}^\ell}(\cdot,U_{\ell-1})\circ \varphi_{s_{m-1}}\circ F_{I_{n+1}^{\ell-1}}(\cdot,U_{\ell-2}) \circ\dots\circ \varphi_{s_0}(Z_{t_n}) \rVert +1 \right)\rvert Z_{t_n} \right] \\
    & \qquad \leq L (1+D_1) + L\mathbb{E}_z \left[\left(\lVert \varphi_{s_{\ell-1}}\circ F_{I_{n+1}^{\ell-1}}(\cdot,U_{\ell-2}) \circ\dots\circ \varphi_{s_0}(Z_{t_n}) \rVert  \right)\rvert Z_{t_n} \right].
\end{align*}
In the last inequality we used that $\mathbb{E}[\lVert F_i(z,U)\rVert]\leq \lVert z\rVert + D_1$ for any $i$, which is implied by Assumption~\ref{ass:F_new}(a). Therefore by recursion we have 
\begin{align*}
    & \mathbb{E}_z \left[\lVert \varphi_{s_\ell}\circ F_{I_{n+1}^\ell}(\cdot,U_{\ell-1})\circ \dots\circ \varphi_{s_0}(Z_{t_n}) \rVert \rvert Z_{t_n} \right]\leq \sum_{i=1}^\ell (1+D_1)L^i + L^{\ell+1}(1+ \lVert Z_{t_n}\rVert).
\end{align*}
Moreover we also have that $\lVert \phibar_s(z)\rVert \leq \delta_{n+1}^2 + L(\lVert z\rVert +1)$ by Assumption~\ref{ass:integrator}. It follows that
\begin{align*}
    &\mathbb{E}_z[\lVert Z_{t_{n+1}}-\Zbar_{t_{n+1}}\rVert \mathbbm{1}_{E_{10} \cap \overline{E}}]= \\
    & \leq \sum_{\ell\geq 2}\mathbb{E}_z \left[ \int_{A_\ell} \left(\sum_{i=1}^\ell (1+D_1)L^i + L^{\ell+1}(1+ \lVert Z_{t_n}\rVert) +\delta_{n+1}^2 + L(\lVert \Zbar_{t_n}\rVert+1)\right) p_{Z_{t_n}}(\dd s)\right]\\
    & \leq \mathbb{E}_z \left[ \sum_{\ell\geq 2} \left(\sum_{i=1}^\ell (1+D_1)L^i + L^{\ell+1}(1+ \lVert Z_{t_n}\rVert) +\delta_{n+1}^2 + L(\lVert \Zbar_{t_n}\rVert+1)\right) p_{Z_{t_n}}(A_\ell)\right]\\
      & \leq \tilde{L}\mathbb{E}_z \left[ (1+ \lVert Z_{t_n}\rVert+\lVert \Zbar_{t_n}\rVert)\sum_{\ell\geq 2} \ell L^\ell  p_{Z_{t_n}}(A_\ell)\right]
\end{align*}
for some constant $\tilde{L}$ which depends only on $D_1,L,\delta_0$. The function $f(\ell) = \ell L^\ell $ is increasing in the number of jumps and therefore because the switching rates have a global upper bound $\lambda_{max}$ we obtain 
\begin{align*}
    & \mathbb{E}_z[\lVert Z_{t_{n+1}}-\Zbar_{t_{n+1}}\rVert \mathbbm{1}_{E_{10} \cap \overline{E}}]  \leq \\
    & \quad \leq\tilde{L}\mathbb{E}_z \left[ (1+ \lVert Z_{t_n}\rVert+\lVert \Zbar_{t_n}\rVert) \sum_{\ell\geq 2} \ell L^\ell e^{-\delta_{n+1} \lambda_{max}} \frac{(\delta_{n+1}\lambda_{max})^\ell}{\ell!} \right] \\
    & \quad \leq \Tilde{L}(1+2B(t_{n},z)) \sum_{\ell\geq 2} \ell L^\ell e^{-\delta_{n+1} \lambda_{max}} \frac{(\delta_{n+1}\lambda_{max})^\ell}{\ell !},
\end{align*}
where in the last inequality we used Assumption~\ref{ass:bdd_lambda}. It remains to show that the sum is of order $\delta_{n+1}^2$. This can be proved as follows
\begin{align*}
    & \sum_{\ell\geq 2} \ell L^\ell e^{-\delta_{n+1} \lambda_{max}} \frac{(\delta_{n+1}\lambda_{max})^\ell}{\ell!} = e^{(L-1)\delta_{n+1} \lambda_{max}} \sum_{\ell\geq 2}  e^{-L\delta_{n+1} \lambda_{max}} \frac{(L \delta_{n+1}\lambda_{max})^\ell}{(\ell-1)!}\\
    & \quad= e^{(L-1)\delta_{n+1} \lambda_{max}} L \delta_{n+1}\lambda_{max} \sum_{\ell\geq 1} e^{-L\delta_{n+1} \lambda_{max}} \frac{(L \delta_{n+1}\lambda_{max})^\ell}{\ell!} \\
    & \quad= e^{(L-1)\delta_{n+1} \lambda_{max}} L \delta_{n+1}\lambda_{max} (1-e^{-L\delta_{n+1} \lambda_{max}}) \\
    & \quad \leq \delta_{n+1}^2  e^{(L-1)\delta_{0} \lambda_{max}} L \lambda_{max}.
\end{align*}
In particular we used that $\delta_n\leq\delta_0$ for all $n\in\N$.

The same proof holds on the event $E_{10} \cap \overline{E}$, and thus we have proved the wanted  result.
\end{proof}

\section{Proofs of Section \ref{sec:tv_distance}}\label{sec:proofs_tv_appendix}

\subsection{Proof of Theorem~\ref{thm:tv_distance}: the case of $p=1$}\label{sec:proofs_tv_p=1_appendix}

\begin{proof}[Proof of Lemma \ref{lem:tv_dist_aux}]
    Let us take advantage of the construction in Coupling \ref{coup:error_tv_distance}. First consider the case in which $T_{i^*}(Z_{t_{n-1}}) > \delta_{n}$. In this case there are no random events for either process in the time interval $(t_{n-1},t_{n}]$ and therefore  $Z_{t_{n}}=\Zbar_{t_{n}}=\varphi_{\delta_{n}}(Z_{t_{n-1}})$. Now, consider the case where $T_{i^*} \leq \delta_{n}$. In this scenario, there are three disjoint events:
    \begin{itemize}
        \item The proposed switching time is accepted by both processes. Denote this event as $E_1$.
        \item The proposed switching time is accepted by one process, and rejected by the other. Denote this event as $E_2$.
        \item The proposed switching time is rejected for both processes. Denote this event as $E_3$.
    \end{itemize}
    Therefore we have
    \begin{equation*}
        \mathbb{P}_z(Z_{t_{n}} \neq \Zbar_{t_{n}}\rvert Z_{t_{n-1}}=\Zbar_{t_{n-1}}) = \sum_{i=1}^3 \mathbb{P}_z(Z_{t_{n}} \neq \Zbar_{t_{n}}, E_i\rvert Z_{t_{n-1}}=\Zbar_{t_{n-1}}).
    \end{equation*}
    
    We start with event $E_1$. In this case we have that $\Zbar_{t_n}\neq Z_{t_n}$ if the continuous time process has at least one more jump in time interval $(t_{n-1} + T_{i^*},t_{n}]$. Now let $\lambda(z)= \sum_{i=1}^m \lambda_i(z) $ and $\lambda_{tot}(z,t;\delta_n) = \sum_{i=1}^m \lambda_{tot}^i(z,t;\delta_n)$. Observe that, conditional on $Z_{t_{n-1}}$, the minimum of the $m$ proposed random times is distributed as $\mathbb{P}(T_{i^*} \leq t) = 1-\exp(-\int_0^t\lambda_{tot}(Z_{t_{n-1}},s;\delta_n)\dd s)$. Then bounding by $1$ the probability that both proposals are accepted, and conditioning on $Z_{t_{n-1}}$ we obtain
    \begin{align*}
        \mathbb{P}_z(Z_{t_{n}} \neq \Zbar_{t_{n}}, E_1\rvert Z_{t_{n-1}}=\Zbar_{t_{n-1}}) & \leq \mathbb{E}_z\Bigg[ \int_{0}^{\delta_{n}} \lambda_{tot}(Z_{t_{n-1}},t;\delta_n) e^{-\int_0^t \lambda_{tot}(Z_{t_{n-1}},s;\delta_n)\dd s} \\
        & \quad \left(\!1\!-\!\exp\left(\!-\!\int_{t}^{\delta_{n}} \lambda(\varphi_s(F_{i^*}(\varphi_t(Z_{t_{n-1}}),U_n)))\dd s \right)\right) \!\dd t \Bigg] .
    \end{align*}
    Then using that $1-\exp(-z)\leq z$, that $\exp(-z)\leq 1$ for $z\geq 0$ and by Fubini's theorem we obtain the following bound:
    \begin{align*}
        &\mathbb{P}_z(Z_{t_{n}} \neq \Zbar_{t_{n}}, E_1\rvert Z_{t_{n-1}}=\Zbar_{t_{n-1}})  \leq \\
        &\quad\leq\mathbb{E}_z\Bigg[ \int_{0}^{\delta_{n}} \int_{t}^{\delta_{n}} \lambda_{tot}(Z_{t_{n-1}},t;\delta_n)  \lambda(\varphi_s(F_{i^*}(\varphi_t(Z_{t_{n-1}}),U_n)))\dd s \dd t \Bigg] \\
        & \quad\leq \int_{0}^{\delta_{n}} \int_{t}^{\delta_{n}} \mathbb{E}_z\left[  \lambda_{tot}(Z_{t_{n-1}},t;\delta_n)  \lambda(\varphi_s(F_{i^*}(\varphi_t(Z_{t_{n-1}}),U_n))  \right]\dd s \dd t \\
        & \quad\leq \delta_{n}^2 L_1(t_{n},z)/2.
    \end{align*}
    Note that in the last inequality the bound $L_1(t_n,z)$ follows from part (a) of Assumption~\ref{ass:lambda_tv}.
    
    Let us now consider event $E_2$. As the proposal $T_{i^*}(Z_{t_{n-1}})$ is accepted for one process only, it must be that 
    \begin{align*}
    \overline{U}\in \Bigg(&\min\left\{\frac{\lambda_{i^*}(\varphi_{T_{i^*}}(Z_{t_{n-1}}))}{\lambda_{tot}^{i^*} (Z_{t_{n-1}},T_{i^*}(Z_{t_{n-1}});\delta_n)},\frac{\overline{\lambda}_{i^*}(Z_{t_{n-1}},T_{i^*}(Z_{t_{n-1}}))}{\lambda_{tot}^{i^*} (Z_{t_{n-1}},T_{i^*}(Z_{t_{n-1}});\delta_n)} \right\}, \\
    & \max\left\{\frac{\lambda_{i^*}(\varphi_{T_{i^*}}(Z_{t_{n-1}}))}{\lambda_{tot}^{i^*} (Z_{t_{n-1}},T_{i^*}(Z_{t_{n-1}});\delta_n)},\frac{\overline{\lambda}_{i^*}(Z_{t_{n-1}},T_{i^*}(Z_{t_{n-1}});\delta_n)}{\lambda_{tot}^{i^*} (Z_{t_{n-1}},T_{i^*}(Z_{t_{n-1}});\delta_n)} \right\}  \Bigg].
    \end{align*}
    Therefore using that $\overline{U}$ and $T_{i^*}$ are independent we obtain
    \begin{align*}
        & \mathbb{P}(Z_{t_{n}} \neq \Zbar_{t_{n}}, E_2\rvert Z_{t_{n-1}}=\Zbar_{t_{n-1}})  =\\
        & \quad = \mathbb{E}\left[\mathbbm{1}_{\{T_{i^*}(Z_{t_{n-1}}) \leq \delta_{n}\}} \left\lvert \frac{\overline{\lambda}_{i^*}(Z_{t_{n-1}},T_{i^*}(Z_{t_{n-1}});\delta_n)-\lambda_{i^*}(\varphi_{T_{i^*}}(Z_{t_{n-1}}))}{\lambda^{i^*}_{tot}(Z_{t_{n-1}},T_{i^*}(Z_{t_{n-1}});\delta_n)} \right\rvert \right].
    \end{align*}
    By the definition given in Coupling \ref{coup:error_tv_distance} we have $\lambda_{tot}^{i^*}(z,t;\delta_n)\geq 1$. Using part (b) of Assumption~\ref{ass:lambda_tv} and Fubini's theorem:
    \begin{align*}
        &\mathbb{P}(Z_{t_{n}} \neq \Zbar_{t_{n}}, E_2\rvert Z_{t_{n-1}}=\Zbar_{t_{n-1}})  \leq \delta_n \mathbb{E}_z\left[\overline{M}_2(Z_{t_{n-1}})\mathbbm{1}_{\{T_{i^*}(Z_{t_{n-1}}) \leq \delta_{n}\}} \right]\\
        & \quad\leq \delta_{n} \mathbb{E}_z \left[\overline{M}_2(Z_{t_{n-1}}) \int_0^{\delta_{n}} \lambda_{tot}(Z_{t_{n-1}},t;\delta_n) e^{-\int_0^t \lambda_{tot}(Z_{t_{n-1}},s;\delta_n)\dd s} \dd t \right]\\
        & \quad \leq \delta_{n}\int_0^{\delta_{n}} \mathbb{E}_z\left[\lambda_{tot}(Z_{t_{n-1}},t;\delta_n) \overline{M}_2(Z_{t_{n-1}}) \right]\dd t \\
        & \quad \leq \delta_{n}^2 L_2(t_{n},z).
    \end{align*}
    
    Finally, we focus on $E_3$. On this event, the processes remain equal unless there is (at least) a switch for either process for $t\in (t_{n-1}+T_{i^*}(Z_{t_n}),t_{n})$. Recall $\overline{\lambda}(z,s;\delta_n)= \sum_{i=1}^m \overline{\lambda}_i(z,s;\delta_n) $. Using this observation together with Assumption~\ref{ass:lambda_tv} and the facts that on this event $T_{i^*}(Z_{t_n})\leq \delta_{n}$ and that $1-\exp(-z)\leq z$ we obtain
    \begin{align*}
        \mathbb{P}(Z_{t_{n}} & \neq \Zbar_{t_{n}}, E_3\rvert Z_{t_{n-1}}=\Zbar_{t_{n-1}})  \leq \mathbb{E}_z \Bigg[ \int_0^{\delta_{n}} \Bigg( \left(1-\exp \left(-\int_{t}^{\delta_{n}} \lambda(\varphi_r(Z_{t_{n-1}}))\dd r \right)\right) \\
        & \qquad + \left(1-\exp\left(-\int_t^{\delta_n} \overline{\lambda}(Z_{t_{n-1}},r;\delta_n)\dd r \right) \right)  \Bigg) \\
        & \qquad \lambda_{tot}(Z_{t_{n-1}},t;\delta_n) \exp{\left(-\int_0^t \lambda_{tot}(Z_{t_{n-1}},s;\delta_n)\dd s\right)
        } \dd t \Bigg] \\
        & \leq \mathbb{E}_z \Bigg[ \int_0^{\delta_{n}}  \int_{t}^{\delta_{n}}\lambda_{tot}(Z_{t_{n-1}},t;\delta_n) \left(\lambda(\varphi_r(Z_{t_{n-1}}))+ \overline{\lambda}(Z_{t_{n-1}},r;\delta_n)\right) \dd r \dd t \Bigg] \\
        & =\int_0^{\delta_{n}}  \int_{t}^{\delta_{n}} \mathbb{E}_z \Big[ \lambda_{tot}(Z_{t_{n-1}},t;\delta_n) \left(\lambda(\varphi_r(Z_{t_{n-1}}))+ \overline{\lambda}(Z_{t_{n-1}},r;\delta_n)\right)  \Big] \dd r \dd t\\
        & \leq \delta_{n}^2 L_3(t_n,z)/2.
    \end{align*}
    Combining the three bounds on events $E_1,E_2,E_3$ we obtain the statement.
\end{proof}

\subsection{Proof of Theorem~\ref{thm:tv_distance}: the case of $p>1$}\label{sec:proof_higher_order_tv_appendix}

\begin{proof}[Proof of Theorem~\ref{thm:tv_distance}]
Observe that if $T_{i^*}>\delta$ the two processes are equal at time $\delta$ and thus the probability that $Z_{t_n} \neq \Zbar_{t_n}$ is $0$.
We analyse in turn the three events $E_1,E_2,E_3$ which were defined in Section \ref{sec:proofs_tv_p=1} in the proof of Lemma \ref{lem:tv_dist_aux}. Define the event $\Eeq=\{Z_{t_{n-1}}=\Zbar_{t_{n-1}}\}$.

On event $E_1$, the proposal $T_{i^*}$ is accepted by both processes. 
Then we reformulate $\mathbb{P}_z(Z_{t_n} \neq \Zbar_{t_n}, E_1 \rvert \Eeq)$ in terms of the conditional probability
\begin{align*}
    \mathbb{P}_z(Z_{t_n} \neq \Zbar_{t_n}, E_1\rvert \Eeq) & = \mathbb{P}_z(Z_{t_n} \neq \Zbar_{t_n}\rvert Z_{t_{n-1}+T_{i^*}} = \Zbar_{t_{n-1}+T_{i^*}}, T_{i^*}<\delta) \mathbb{P}_z(E_1\rvert \Eeq). 
\end{align*}
The first term on the right hand side can be bounded by applying Inductive Hypothesis \ref{hyp:induction_higherorder_tv}. Moreover we can use the bound $\mathbb{P}_z(E_1\rvert \Eeq)\leq \mathbb{P}_z(T_{i^*}<\delta\rvert \Eeq)$ for the rightmost term to obtain
\begin{align*}
    \mathbb{P}_z(Z_{t_n} \neq \Zbar_{t_n}, E_1\rvert \Eeq) & \leq A \delta^{p+1}\, \mathbb{E}_z\left[\left(1-\exp\left(-\int_0^\delta \lambda_{tot}(Z_{t_{n-1}},t;\delta,p+1)\dd t\right)\right) \right]\\
    &  \leq A \delta^{p+2} \sup_{s\in[0,\delta]} \mathbb{E}_z\left[ \lambda_{tot}(Z_{t_n},s;\delta,p+1)\right] \leq A \delta^{p+2} L_4(t_n,z).
\end{align*}
In the last inequality we took advantage of the bound $1-\exp(-z)\leq x$ which is true for $z>0$.

On event $E_2$ the proposal $T_{i^*}$ is accepted for one process, and rejected for the other. This happens when
\begin{align*}
\overline{U}&\in \Bigg(\min\left\{\frac{\lambda_{i^*}(\varphi_{T_{i^*}}(Z_{t_{n-1}}))}{\lambda_{tot}^{i^*} (Z_{t_{n-1}},T_{i^*};\delta,p+1)},\frac{\overline{\lambda}_{i^*}(Z_{t_{n-1}},T_{i^*};\delta,p+1)}{\lambda_{tot}^{i^*} (Z_{t_{n-1}},T_{i^*};\delta,p+1)} \right\},\\
&\quad \max\left\{\frac{\lambda_{i^*}(\varphi_{T_{i^*}}(Z_{t_{n-1}})))}{\lambda_{tot}^{i^*} (Z_{t_{n-1}},T_{i^*};\delta,p+1)},\frac{\overline{\lambda}_{i^*}(Z_{t_{n-1}},T_{i^*};\delta,p+1)}{\lambda_{tot}^{i^*} (Z_{t_{n-1}},T_{i^*};\delta,p+1)} \right\}  \Bigg],
\end{align*}
and therefore with probability
\begin{align*}
    \left\lvert \frac{\lambdabar_{i^*}(Z_{t_{n-1}},T_{i^*};\delta,p+1)\!-\!\lambda_{i^*}(\varphi_{T_{i^*}}(Z_{t_{n-1}}))}{\lambda_{tot}^{i^*} (z,T_{i^*}(z);\delta,p+1)} \right\rvert \!&\leq\! \left\lvert \lambdabar_{i^*}(Z_{t_{n-1}},T_{i^*};\delta,p+1) \!-\! \lambda_{i^*}(\varphi_{T_{i^*}}(Z_{t_{n-1}})) \right\rvert\\
    & \leq \delta^{p+1} \overline{M}_2(Z_{t_{n-1}})
\end{align*}
where we used that by definition $\lambda_{tot}^{i^*}\geq 1$ and then that $\overline{\lambda}_{i^*}(\cdot,\cdot;\delta,p+1)$ is an approximation of $p+1$ order. Thus we have
\begin{align*}
     \mathbb{P}_z(Z_{t_n} \neq \Zbar_{t_n}, E_2\rvert \Eeq) & \leq \delta^{p+1} \mathbb{E}_z\left[\overline{M}_2(Z_{t_{n-1}})\mathbb{P}_z(T_{i^*}<\delta\rvert Z_{t_{n-1}},\Eeq)\right]\\
     & \leq \delta^{p+2} \sup_{s\in[0,\delta]} \mathbb{E}_z\left[\overline{M}_2(Z_{t_{n-1}}) \lambda_{tot}(Z_{t_n},s;\delta,p+1)\right]\\
     & \leq \delta^{p+2} L_4(t_n,z)
\end{align*}

Finally consider event $E_3$. Similarly to the proof of Theorem~\ref{thm:strong_error_pdmp} it is sufficient to bound the event that $p+2$ proposal times occur before the end of the time interval, which is bounded by Assumption~\ref{ass:lambda_tv}.

\end{proof}

\section{Proofs of Section \ref{sec:weakerror}}\label{sec:proofs_weakerror_appendix}

\subsection{Proofs of Theorem~\ref{thm:weakerror} and its corollaries}\label{sec:proofs_weak_error_theorem_appendix}
\begin{lemma}\label{lem:estlawforjump}
Suppose $\lambda$ and $\overline{\lambda}$ satisfy Assumption~\ref{ass:lambda_wasserstein} (a). We will consider the two algorithms separately.
For Algorithm \ref{alg:basic_EPDMP}, let $\ptau{\delta,\ref{alg:basic_EPDMP}}$ be given by \eqref{eq:basicjumplaw} then for any $h\in \C_b^1([0,\delta])$ we have
\begin{align*}
    \left\lvert\int_0^{\delta}  h_s\ptau{\delta,\ref{alg:basic_EPDMP}}(\dd s) -  h_s\lambda(\varphi_s(z))\dd s\right\rvert &\leq \delta^2\sup_{s,r\in [0,\delta]} \big(\lvert \partial_rh_r\rvert  \lambdabar(z,s;\delta)  \\
    & \quad +   \lvert h_s\rvert \big(   \lambdabar(z,s;\delta)\lambdabar(z,r;\delta)  + \overline{M}_2(z)\big)\big).
\end{align*}

For Algorithm \ref{alg:advanced_EPDMP}, let $\ptau{\delta,\ref{alg:advanced_EPDMP}}$ be given by \eqref{eq:advancedjumplaw} then for any $h\in \C_b([0,\delta])$
\begin{equation*}
    \left\lvert \int_0^\delta h_s \ptau{\delta,\ref{alg:advanced_EPDMP}}(\dd s) - \int_0^\delta \lambda(\varphi_s(z))h_s\dd s \right\rvert \leq \delta^2 \sup_{s,r\in[0,\delta]}\left(\lvert h_s\rvert ( \lambdabar(z,s;\delta)\lambdabar(z,r;\delta)  +\overline{M}_2(z))\right)
\end{equation*}
\end{lemma}

\begin{proof}[Proof of Lemma \ref{lem:estlawforjump}]
First consider the case where $\ptau{\delta,\ref{alg:basic_EPDMP}}$ is given by \eqref{eq:basicjumplaw}, and fix $h\in C_b^1([0,\delta])$.
Then
\begin{align*}
    \left\lvert\int_0^{\delta}  h_s\ptau{\delta,\ref{alg:basic_EPDMP}}(\dd s) -  h_s\lambda(\varphi_s(z))\dd s\right\rvert= &\left\lvert  h_\delta \left(1-e^{-\int_0^{\delta}\overline{\lambda}(z,s;\delta)\dd s}\right) -  \int_0^\delta h_s\lambda(\varphi_s(z))\dd s\right\rvert.
\end{align*}
We can rewrite
\begin{equation*}
    1-e^{-\int_0^{\delta}\overline{\lambda}(z,s;\delta)\dd s} = \int_0^\delta \overline{\lambda}(z,s;\delta)e^{-\int_0^{\delta}\overline{\lambda}(z,r;\delta)\dd r}\dd s.
\end{equation*}
Therefore we have
\begin{align}
    \left\lvert\int_0^{\delta}  h_s\ptau{\delta,\ref{alg:basic_EPDMP}}(\dd s) -  h_s\lambda(\varphi_s(z))\dd s\right\rvert &= \left\lvert  h_\delta \int_0^\delta \overline{\lambda}(z,s;\delta)e^{-\int_0^{\delta}\overline{\lambda}(z,r;\delta)\dd r}\dd s -  \int_0^\delta h_s\lambda(\varphi_s(z))\dd s\right\rvert \notag\\
    \begin{split}
        &\leq \left\lvert  \int_0^\delta (h_\delta-h_s)  \overline{\lambda}(z,s;\delta)e^{-\int_0^{\delta}\overline{\lambda}(z,r;\delta)\dd r}\dd s \right\rvert \\
        & \quad + \left\lvert  \int_0^\delta h_s\left(  \overline{\lambda}(z,s;\delta)e^{-\int_0^{\delta}\overline{\lambda}(z,r;\delta)\dd r} - \lambda(\varphi_s(z))\right)\dd s\right\rvert.    \label{eq:basiclawerrorterms}
    \end{split}
\end{align}
We can use Assumption~\ref{ass:lambda_wasserstein} (a)  and that that $1-e^{-y}\leq y$ for $y>0$ to bound the integrand of the second term on the right of \eqref{eq:basiclawerrorterms}, 
\begin{align}\label{eq:lambdaestforlaw}
    \left\lvert \overline{\lambda}(z,s;\delta)e^{-\int_0^{\delta}\overline{\lambda}(z,r;\delta)\dd r} - \lambda(\varphi_s(z))\right\rvert  &\leq \left\lvert \overline{\lambda}(z,s;\delta)(1-e^{-\int_0^{\delta}\overline{\lambda}(z,r;\delta)\dd r}) \right\rvert \notag \\
    & \quad +\left\lvert \overline{\lambda}(z,s;\delta) - \lambda(\varphi_s(z))\right\rvert \nonumber\\
    & \leq  \overline{\lambda}(z,s;\delta)\int_0^{\delta}\overline{\lambda}(z,r;\delta)\dd r  +\delta \overline{M}_2(z).
\end{align}
For the first term on the right hand side of \eqref{eq:basiclawerrorterms} we use that
\begin{equation}\label{eq:MVT}
    \lvert h_\delta-h_s\rvert \leq (\delta-s)\sup_{r\in [0,\delta]} \lvert \partial_rh_r\rvert.
\end{equation}
Applying \eqref{eq:lambdaestforlaw} and \eqref{eq:MVT} to \eqref{eq:basiclawerrorterms} we have
\begin{align*}
    &\left\lvert\int_0^{\delta}  h_s\ptau{\delta,\ref{alg:basic_EPDMP}}(\dd s) -  h_s\lambda(\varphi_s(z))\dd s\right\rvert \leq \\
    &\quad \leq \left\lvert \sup_{r\in [0,\delta]} \lvert \partial_rh_r\rvert \int_0^\delta (\delta-s)  \overline{\lambda}(z,s;\delta)e^{-\int_0^{\delta}\overline{\lambda}(z,r;\delta)\dd r}\dd s \right\rvert \\
    & \qquad + \left\lvert  \int_0^\delta \lvert h_s\rvert \left(   \overline{\lambda}(z,s;\delta)\int_0^{\delta}\overline{\lambda}(z,r;\delta)\dd r  +\delta \overline{M}_2(z)\right)\dd s\right\rvert\\
    & \quad \leq \delta^2\sup_{s,r\in [0,\delta]} \left(\lvert \partial_rh_r\rvert  \overline{\lambda}(z,s;\delta)  +   \lvert h_s\rvert \left(   \overline{\lambda}(z,s;\delta)\overline{\lambda}(z,r)  + \overline{M}_2(z)\right)\right).
\end{align*}

Let us consider the case where $\ptau{\delta,\ref{alg:advanced_EPDMP}}$ is given by \eqref{eq:advancedjumplaw}. We use \eqref{eq:lambdaestforlaw} to bound
\begin{align*}
     & \left\lvert\int_0^{\delta}  h_s\ptau{\delta,\ref{alg:advanced_EPDMP}}(\dd s) -  h_s\lambda(\varphi_s(z))\dd s\right\rvert \leq\\
     & \quad \leq \int_0^{\delta} \lvert h_s\rvert \left\lvert \overline{\lambda}(z,s;\delta) \exp\left(-\int_0^s \overline{\lambda}(z,r;\delta)\dd r\right) -  \lambda(\varphi_s(z))\right\rvert\dd s\\
     &\quad \leq \int_0^{\delta} \lvert h_s\rvert \left( \overline{\lambda}(z,s;\delta)\int_0^{\delta}\overline{\lambda}(z,r;\delta)\dd r  +\delta \overline{M}_2(z)\right)\dd s\\
    &\quad \leq \delta^2 \sup_{s,r\in[0,\delta]}\left(\lvert h_s\rvert ( \overline{\lambda}(z,s;\delta)\overline{\lambda}(z,r;\delta)  +\overline{M}_2(z))\right).
\end{align*}
\end{proof}

\begin{proof}[Proof of Corollary~\ref{cor:convergence-to-stationary-measure}]
    First observe that \eqref{eq:WEfixeddelta} follows from \eqref{eq:WE}.
    Then  \eqref{eq:weakerrortotarget} is obtained by adding \eqref{eq:geoerg} and \eqref{eq:WE}. To obtain \eqref{eq:timeav} we use that
    \begin{align*}
        \left\lvert \frac{1}{N}\sum_{n=1}^N\mathbb{E}_{z}[g(\Zbar_{t_n})] -\mu(g)\right\rvert  &\leq \left\lvert \frac{1}{N}\sum_{n=1}^N(\mathbb{E}_{z}[g(\Zbar_{t_n})]-\mathbb{E}_{z}[g(Z_{t_n})])\right\rvert\\
        &+\left\lvert\frac{1}{N}\sum_{n=1}^N\mathbb{E}_{z}[g(Z_{t_n})] -\mu(g)\right\rvert .
    \end{align*}
    We bound this using \eqref{eq:geoerg} and \eqref{eq:WE}
    \begin{align*}
        \left\lvert \frac{1}{N}\sum_{n=1}^N\mathbb{E}_{z}[g(\Zbar_{t_n})] -\mu(g)\right\rvert  &\leq C\delta \lf_{2}(z)+C \lf_{2}(z)\frac{1}{N}\sum_{n=1}^Ne^{-\omega t_n} \\
        &\leq C \lf_2(z)\left(\delta+\frac{1}{t_N}\right) .
    \end{align*}
\end{proof}

\begin{proof}[Proof of Corollary~\ref{cor:convergence_variable_step_size}]
    It is sufficient to show for $S_n$ given by \eqref{eq:Sn} that $S_n\to 0$ as $n\to \infty$. Fix $\eta>0$. Then we have
    \begin{align*}
        S_n=\sum_{k=0}^{\eta-1}   \delta_{k+1}^2 e^{-\omega(t_n-t_{k+1})} + \sum_{k=\eta}^{n-1}   \delta_{k+1}^2 e^{-\omega(t_n-t_{k+1})}.
    \end{align*}
    Consider the first term:
    \begin{align*}
        \sum_{k=0}^{\eta-1}   \delta_{k+1}^2 e^{-\omega(t_n-t_{k+1})}  &\leq \sup_k\delta_k \int_0^{t_n-\eta}e^{-\omega (t_n-s)}\dd s=\sup_k\delta_k \frac{e^{-\omega \eta} - e^{-\omega t_n}}{\omega}.
    \end{align*}
    Consider the second term:
    \begin{align*}
        \sum_{k=\eta}^{n-1}   \delta_{k+1}^2 e^{-\omega(t_n-t_{k+1})}\leq \left(\sup_{k\in \{\eta,...,n\}} \delta_k\right)\int_{t_\eta}^{t_n}e^{-\omega (t_n-s)}\dd s  = \frac{1-e^{-\omega(t_n-t_\eta)}}{\omega}\sup_{k\in \{\eta,...,n\}} \delta_k.
    \end{align*}
Therefore
    \begin{equation*}
        \limsup_{n\to\infty} S_n \leq \left(\sup_{k\geq 0}\delta_k\right) \frac{e^{-\omega \eta}}{\omega}+\frac{1}{\omega}\sup_{k\geq \eta} \delta_k.
    \end{equation*}
    Since $\eta$ is arbitrary we let $\eta$ tend to $\infty$ which gives that $S_n\to 0$ as $n\to\infty$.
\end{proof}

\subsection{Proofs of Example \ref{ex:rhmc_weakerror}}\label{sec:proofs_rhmc_weakerror}

\begin{proof}[Proof of Proposition~\ref{prop:rHMC_derivativeest}]
    Fix $f\in C_b^1(\R^d\times\R^d)$. Then by the chain rule
    \begin{equation*}
        \lVert\nabla_{q,p}\cP_tf(q,p)\rVert = \lVert\mathbb{E}[\nabla_{q,p}(Q_t,P_t)(\nabla_{q,p}f)(Q_t,P_t)]\rVert \leq \lVert f\rVert_{C_b^1} \mathbb{E}[\lVert\nabla_{q,p}(Q_t,P_t)\rVert].
    \end{equation*}
    Notice that there is a version of $(Q_t,P_t)$ which is differentiable with respect to the initial conditions since we can write $(Q_t,P_t)$ as the composition of smooth operators. 
    Let $T_i$ denote the $i$-th refreshment time and $\xi_i\sim N(0_d,I_d)$ the corresponding refreshed velocity. Set $T_0=0$. We shall track for which refreshment times we have that $\nu\leq T_i-T_{i-1}\leq K$. Let $M_t$ denote the number of refreshment times before time $t$ which have this property and let $N_t$ denote the total number of refreshment times before time $t$. Note that conditional on $N_t$, $M_t$ is distributed according to a Binomial distribution with $N_t$ trials and success rate $e^{-\lambda\nu}-e^{-\lambda K_1}$.
    
    To stress the dependence on the initial condition for the remainder of the proof we shall write $(Q_t^{q,p},P_t^{q,p})$ to denote the process at time $t$ with initial condition $(q,p)$.
    Then by \eqref{eq:Lipschitzhamflow} we have
    \begin{align*}
        \lVert (Q_t^{q,p},P_t^{q,p})-(Q_t^{\bar{q},\bar{p}},P_t^{\bar{q},\bar{p}}) \rVert&= \lVert \varphi_{t-T_{N_t}}(Q_{T_{N_t}}^{q,p},\xi_{N_t})-\varphi_{t-T_{N_t}}(Q_{T_{N_t}}^{\bar{q},\bar{p}},\xi_{N_t})\rVert\\
        &\leq C\lVert Q_{T_{N_t}}^{q,p}-Q_{T_{N_t}}^{\bar{q},\bar{p}}\rVert.
    \end{align*}
    There are now three possible events either $N_t=0$, $\nu \leq T_{N_t}-T_{N_t-1}\leq K$ or $T_{N_t}-T_{N_t-1}\geq K$. If $\nu \leq T_{N_t}-T_{N_t-1}\leq K$ then we use \eqref{eq:contractionhamflow}, however if $T_{N_t}-T_{N_t-1}\geq K$ then we use \eqref{eq:Lipschitzhamflow}. By doing this for each refreshment we have
    \begin{align*}
        \lVert (Q_t^{q,p},P_t^{q,p})-(Q_t^{\bar{q},\bar{p}},P_t^{\bar{q},\bar{p}}) \rVert&\leq C^{1+N_t-M_t} \gamma^{M_t}\lVert Q_{T_{1}}^{q,p}-Q_{T_{1}}^{\bar{q},\bar{p}}\rVert.
    \end{align*}
    Then by applying \eqref{eq:Lipschitzhamflow} once more we have
     \begin{align*}
        \lVert (Q_t^{q,p},P_t^{q,p})-(Q_t^{\bar{q},\bar{p}},P_t^{\bar{q},\bar{p}}) \rVert&\leq C^{2+N_t-M_t} \gamma^{M_t}\lVert (q,p)-(\bar{q},\bar{p})\rVert.
    \end{align*}
    Dividing by $\lVert (q,p)-(\bar{q},\bar{p})\rVert$ and taking the limit as $\lVert (q,p)-(\bar{q},\bar{p})\rVert\to 0$ we have that
    \begin{align*}
        \lVert \nabla_{q,p} (Q_t^{q,p},P_t^{q,p}) \rVert&\leq C^{2+N_t-M_t} \gamma^{M_t}.
    \end{align*}
    It remains to bound $\mathbb{E}[C^{N_t-M_t} \gamma^{M_t}]$.
    By conditioning on $N_t$ we can use the moment generating function of a Binomial distribution to find
    \begin{equation*}
        \mathbb{E}[C^{N_t-M_t} \gamma^{M_t}\rvert N_t] = C^{N_t} (1-(e^{-\lambda\nu}-e^{-\lambda K})(1-\gamma C^{-1}))^{N_t}.
    \end{equation*}
    Now $N_t$ is a Poisson process with rate $\lambda$ so we have
    \begin{equation*}
        \mathbb{E}[C^{N_t-M_t} \gamma^{M_t}] =\exp\left(\lambda\left( C (1-(e^{-\lambda\nu}-e^{-\lambda K})(1-\gamma C^{-1}))-1\right)\right).
    \end{equation*}
    This is decays exponentially provided
    \begin{equation*}
      C (1-(e^{-\lambda\nu}-e^{-\lambda K_1})(1-\gamma C^{-1}))<1.  
    \end{equation*}
      
\end{proof}


\subsection{Proofs of Example \ref{ex:zzs_weakerror}}\label{sec:proofs_zzs_weakerror}

\begin{proof}[Proof of Lemma \ref{lem:ZZS_commutator}]

Note that for the ZZS
\begin{align*}
    [\vfd,Q]f(x,v)&=\sum_{i=1}^d\left\langle v,\nabla_x \left(\frac{\lambda_i(x,v)}{\lambda(x,v)}f(x,F_iv)\right)\right\rangle - \sum_{i=1}^d\frac{\lambda_i(x,v)}{\lambda(x,v)}\langle F_iv,\nabla_x (f(x,F_iv))\rangle\\
    &=\sum_{i=1}^d\left\langle v,\nabla_x \left(\frac{\lambda_i(x,v)}{\lambda(x,v)}\right)\right\rangle f(x,F_iv)+2\sum_{i=1}^d \left(\frac{\lambda_i(x,v)}{\lambda(x,v)}\right) v_i\partial_{x_i}f(x,F_iv).
\end{align*}
When we apply this with $f=\cP_tg\circ \varphi_{\delta-s}$ and $(x,v)$ replaced by $(x+vs,v)$ we have
\begin{align*}
    [\vfd,Q]f(x,v)&=\sum_{i=1}^d\left\langle v,\nabla_x \left(\frac{\lambda_i(x,v)}{\lambda(x,v)}\right)\right\rangle \cP_tg(x+sv+(\delta-s)F_iv,F_iv)\\
    &+2\sum_{i=1}^d \left(\frac{\lambda_i(x,v)}{\lambda(x,v)}\right) v_i\partial_{x_i}(\cP_tg)(x+sv+(\delta-s)F_iv,F_iv)
    \\
    &\leq \sum_{i=1}^d\left\langle v,\nabla_x \left(\frac{\lambda_i(x,v)}{\lambda(x,v)}\right)\right\rangle \lvert\cP_tg(x+sv+(\delta-s)F_iv,F_iv)\rvert\\
    &+2 \lVert \nabla_x(\cP_tg)(x+sv+(\delta-s)F_iv,F_iv)\rVert.
\end{align*}
Observe that
\begin{equation*}
    \partial_r\left(-\log\left(\phi(\exp(-r)\right)\right) = \frac{1}{1+e^{-r}}
\end{equation*}
then we have
\begin{align*}
    \nabla_x \left(\frac{\lambda_i(x,v)}{\lambda(x,v)}\right) &=  \left(\frac{\nabla_x\lambda_i(x,v)}{\lambda(x,v)}\right)-(\nabla_x\lambda(x,v)) \left(\frac{\lambda_i(x,v)}{\lambda(x,v)^2}\right)\\
    &=\left(\frac{v_i\nabla_x\partial_{x_i}\pot(x,v)}{(1+e^{-v_i\partial_{x_i}\pot(x)})\lambda(x,v)}\right)-\sum_{j=1}^d\left(\frac{v_j\nabla_x\partial_{x_j}\pot(x,v)\lambda_i(x,v)}{(1+e^{-v_i\partial_{x_i}\pot(x)})\lambda(x,v)^2}\right)
\end{align*}
Under our assumptions this is bounded.
\end{proof}

\begin{proof}[Proof of Theorem~\ref{thm:exptimederivativebound}]
    Fix $f\in \mathcal{G}_1$. We observe that $\cP_t f$ satisfies
    \begin{equation*}
        \partial_t\cP_tf(x,v) = \cL \cP_t f(x,v).
    \end{equation*}
    We can differentiate this with respect to the $i$-th component of $x$, denoted as $x^i$, to obtain an equation for $\partial_{x^i}\cP_t$.
    \begin{equation*}
        \partial_t\partial_{x^i}\cP_tf(x,v) = \cL \partial_{x^i}\cP_tf+\sum_{j=1}^d \partial_{x^i}\lambda_j(x,v)[\cP_tf(x,R_jv)-\cP_tf(x,v)].
    \end{equation*}
    We can solve this equation using the variation of constants formula
    \begin{align*}
        \partial_{x^i}\cP_tf(x,v) = \cP_t\partial_{x^i} f(x,v) +\sum_{j=1}^d\int_0^t \cP_{t-s}\left( g_s^{i,j}\right)(x,v)\dd s
    \end{align*}
    where
    \begin{equation*}
        g_s^{i,j}(x,v) = \partial_{x^i}\lambda_j(x,v)[\cP_sf(x,R_jv)-\cP_sf(x,v)].
    \end{equation*}
    We can integrate this with respect to $\mu$ to obtain an expression for $\mu(\partial_{x^i}\cP_tf)$
    \begin{align*}
        \mu(\partial_{x^i}\cP_tf) = \mu(\partial_{x^i} f) +\sum_{j=1}^d\int_0^t \mu\left( g_s^{i,j}\right)\dd s.
    \end{align*}
    Here we have used that $\mu$ is an invariant measure for $\cP_t$ to remove $\cP_{t-s}$ terms.
    We shall bound this by comparing to $\mu(\partial_{x^i}\cP_tf)$, 
    \begin{equation}\label{eq:derivativeofsemigpminusmu}
     \begin{aligned}
        \lvert \partial_{x^i}\cP_tf(x,v)-\mu(\partial_{x^i}\cP_tf)\rvert  &\leq \lvert  \cP_t\partial_{x^i} f(x,v) - \mu(\partial_{x^i}f)\rvert \\
        & \quad +\lvert\sum_{j=1}^d\int_0^t \cP_{t-s}\left( g_s^{i,j}\right)(x,v)- \sum_{j=1}^d\int_0^t\mu\left(  g_s^{i,j}\right)\dd s\rvert. 
    \end{aligned}
    \end{equation}
    
    Observe that since $\lambda_i$ is globally Lipschitz there exists a constant $C_\lambda$ such that $\lvert\partial_{x^i}\lambda\rvert \leq C_\lambda$ for any $i\in \{1,\ldots,d\}, x\in \R^d, v\in \{\pm 1\}^d$. Therefore we can bound $g_s^{i,j}$ by
    \begin{equation*}
        \lvert g_s^{i,j}(x,v)\rvert \leq C_\lambda (\lvert \cP_sf(x,R_jv)-\mu(f)\rvert +\lvert\cP_sf(x,v)-\mu(f)\rvert).
    \end{equation*}
    Now we can bound $g_s^{i,j}$ using geometric ergodicity, \eqref{eq:geoerg}, and that $f\in \mathcal{G}_{1}$
    \begin{equation*}
        \lvert g_s^{i,j}(x,v)\rvert \leq CC_\lambda e^{-\kappa t}( \lf_{\overline{\alpha},\epsilon}(x,R_jv) + \lf_{\overline{\alpha},\epsilon}(x,v)).
    \end{equation*}
    By \eqref{eq:growthbounds} there exists a constant $K$ such that
    \begin{equation*}
        \lvert g_s^{i,j}(x,v)\rvert \leq K e^{-\kappa s}\lf_{\alpha,\epsilon}(x,v).
    \end{equation*}
    
   By applying \eqref{eq:geoerg} we have
    \begin{align}
       \lvert\sum_{j=1}^d\int_0^t \cP_{t-s}\left(g_s^{i,j}\right)(x,v)- \sum_{j=1}^d\int_0^t\mu\left( g_s^{i,j}\right)\dd s\rvert &\leq \sum_{j=1}^d\int_0^t \lvert\cP_{t-s}\left(g_s^{i,j}\right)(x,v)-\mu\left( g_s^{i,j}\right) \rvert\dd s \nonumber\\
       &\leq CK\lf_{\alpha,\epsilon}(x,v)\int_0^t   e^{-\kappa t}\dd s\nonumber\\
       &\leq CK\lf_{\alpha,\epsilon}(x,v) t e^{-\kappa t}.\label{eq:boundings}
    \end{align}

    Since $f\in \mathcal{G}_{1}$ we can also apply \eqref{eq:geoerg} for the function $\partial_i f$ 
    \begin{equation}
        \lvert \cP_t \partial_{x^i}f(x,v) -\mu(\partial_{x^i} f)\rvert \leq C \lf_{\alpha,\epsilon}(x,v) e^{-\kappa t}.\label{eq:boundofderivativeoff}
    \end{equation}
    Using \eqref{eq:boundings}, \eqref{eq:boundofderivativeoff} to bound the right hand side of \eqref{eq:derivativeofsemigpminusmu} we obtain
    \begin{equation}\label{eq:errorboundderivativesemigp}
        \lvert \partial_{x^i}\cP_tf(x,v)-\mu(\partial_{x^i}\cP_tf)\rvert  \leq  C (1+Kt)e^{-\kappa t}\lf_{\alpha,\epsilon}(x,v) . 
    \end{equation}
    It remains to consider $\mu(\partial_{x^i}\cP_t f)$. 
    Note that we have following integration by parts formula for $\mu$
    \begin{equation}\label{eq:IBP}
        \int \partial_{x_k}g h \dd \mu = - \int g \partial_{x_k}h \dd \mu + \int \partial_{x_k}\pot g h \dd \mu.
    \end{equation}
    Setting $k=i$, $g=\cP_t f$ and $h=1$ we have
    \begin{equation*}
        \mu(\partial_{x^i}\cP_t f) = \mu(\cP_t f\partial_{x^i}\pot)
    \end{equation*}
    Observe that $\mu(\partial_{x^i}\pot)=0$ (this follows from \eqref{eq:IBP} with $g=h=1$ and $k=i$) then subtracting $\mu(\partial_{x^i}\pot)\mu(f)$ we have
    \begin{equation}
    \begin{aligned}
        \lvert \mu(\partial_{x^i}\cP_t f)\rvert & =\lvert \mu(\partial_{x^i}(\cP_t f-\mu(f))\rvert \\ 
        & \leq \mu(\lvert\partial_{x^i}\pot(x)\rvert \lvert\cP_tf-\mu(f)\rvert)\\
        & \leq C e^{-\kappa t} \mu(\lvert\partial_{x^i}\pot(x)\rvert \lf_{\overline{\alpha},\epsilon}(x,v)).\label{eq:limitofmupartialf}
    \end{aligned}
    \end{equation}
    By \eqref{eq:growthbounds} we have $\lvert\partial_{x^i}\pot(x)\rvert \lf_{\overline{\alpha},\epsilon}(x,v) \leq C\lf_{\alpha,\epsilon}$ and combining \eqref{eq:errorboundderivativesemigp} and \eqref{eq:limitofmupartialf} we obtain a constant $C'$ such that
    \begin{equation*}
        \lvert \partial_{x^i} \cP_tf(x,v)\rvert \leq C'(1+Kt)e^{-\kappa t}\lf_{\alpha,\epsilon}(x,v).
    \end{equation*}
\end{proof}

\begin{lemma}\label{lem:1dLyapunovfunction}
Let $\{(\Xbar_{t_n},\Xbar_{t_n})\}_{n\in \N}$ denote the Euler Zig Zag algorithm in 1-d using Algorithm \ref{alg:basic_EPDMP} or \ref{alg:advanced_EPDMP}. Let $\lambdabar(x,v;\delta) =\lambda(x,v)$ and $\lambda(x,v)=(\pot'(x)v)_++\gamma(x)$ for $\gamma:\mathbb{R}\to[0,\overline{\gamma}]$ with $\overline{\gamma}<\infty$. 
Assume that $\pot\in \C^2$ is such that \eqref{eq:growthbounds} is satisfied.
Let $\alpha\in(0,1)$, $\beta>0$ be such that $\alpha<2\beta$ and define
\begin{equation}\label{eq:approximatelf}
    \lf_{\alpha,\beta}(x,v;\delta) =\begin{cases}
        \exp\left(\alpha \pot(x)+\beta \delta \pot'(x)v\right), \quad \textnormal{ if } v\pot'(x)\geq 0,\\
        \exp\left(\alpha \pot(x)-\beta \delta \pot'(x)v\right), \quad \textnormal{ if } v\pot'(x)<0.\\
    \end{cases} 
\end{equation}
Then there exists a compact set $C$ and $\kappa\in(0,1)$ such that
\begin{equation*}
    \mathbb{E}_{x,v} \lf_{\alpha,\beta}(\Xbar_{t_n},\Vbar_{t_n};\delta) \leq \kappa^n \lf_{\alpha,\beta}(x,v;\delta) \qquad \textnormal{for all } x\notin C.
\end{equation*}
\end{lemma}

\begin{proof}
Let $\{(\Xbar^{\ref{alg:basic_EPDMP}}_\delta,\Vbar_\delta^{\ref{alg:basic_EPDMP}})\}$ ($\{(\Xbar^{\ref{alg:advanced_EPDMP}}_\delta,\Vbar_\delta^{\ref{alg:advanced_EPDMP}})\}$ respectively) be given by Algorithm \ref{alg:basic_EPDMP} (Algorithm \ref{alg:advanced_EPDMP} resp.). To simplify the notation in this proof suppress the $\delta$ dependence of $\lf_{\alpha,\beta}$. 
Set $\beta_\pm=\beta$ if $v\pot'(x)\geq 0$ and $\beta_\pm=-\beta$ otherwise.
Observe that
\begin{align*}
    &\mathbb{E}_{x,v}[\lf_{\alpha,\beta}(\Xbar^{\ref{alg:advanced_EPDMP}}_\delta,\Vbar_\delta^{\ref{alg:advanced_EPDMP}})] - \mathbb{E}_{x,v}[\lf_{\alpha,\beta}(\Xbar^{\ref{alg:basic_EPDMP}}_\delta,\Vbar_\delta^{\ref{alg:basic_EPDMP}})] =\\ &=\int_0^\delta\!\! \lambda(x,v)e^{-\lambda(x,v)s}\!\left(e^{\alpha \pot(x+vs-(\delta-s)v)-\delta\beta_\pm v \pot'(x+sv-(\delta-s)v)} \!-\! e^{\alpha \pot(x+\delta v)-\delta\beta_\pm v \pot'(x+\delta v)}\right) \!ds\\
    &=\int_0^\delta \lambda(x,v)\exp\left(-\lambda(x,v)s+\alpha \pot(x+\delta v)-\delta\beta_\pm v \pot'(x+\delta v)\right)\left(e^{I(x,v,s;\delta)} - 1\right)\dd s,
\end{align*}
where
\begin{align*}
    I(x,v,s;\delta)& \coloneqq\alpha \pot(x+vs-(\delta-s)v)- \alpha \pot(x+\delta v)\\
    & \quad -\delta\beta_\pm v \pot'(x+sv-(\delta-s)v)+\delta\beta_\pm v \pot'(x+\delta v).
\end{align*}
By Taylor's theorem we can find $\xi_1,\xi_2$
\begin{equation*}
    I(x,v,s;\delta)=\alpha 2(s-\delta)v\pot'(x+\delta v)+2\alpha (\delta-s)^2 \pot''(\xi_1)+2\beta_{\pm} \delta (s-\delta) v \pot''(\xi_2).
\end{equation*}
By taking $x$ sufficiently large we can ensure that the sign of $I(x,v,s;\delta)$ is equal to the sign of $-v\pot'(x)$. Therefore,
\begin{align*}
    &\mathbb{E}_{x,v}[\lf_{\alpha,\beta}(\Xbar^{\ref{alg:advanced_EPDMP}}_\delta,\Vbar_\delta^{\ref{alg:advanced_EPDMP}})]\leq  \mathbb{E}_{x,v}[\lf_{\alpha,\beta}(\Xbar^{\ref{alg:basic_EPDMP}}_\delta,\Vbar_\delta^{\ref{alg:basic_EPDMP}})] \quad  \textnormal{if } v\pot'(x)>0,\\
    &\mathbb{E}_{x,v}[\lf_{\alpha,\beta}(\Xbar^{\ref{alg:advanced_EPDMP}}_\delta,\Vbar_\delta^{\ref{alg:advanced_EPDMP}})]\geq  \mathbb{E}_{x,v}[\lf_{\alpha,\beta}(\Xbar^{\ref{alg:basic_EPDMP}}_\delta,\Vbar_\delta^{\ref{alg:basic_EPDMP}})] \quad  \textnormal{if } v\pot'(x)<0.
\end{align*}
In the first case it is sufficient to consider Algorithm \ref{alg:basic_EPDMP}, while in the latter it is sufficient to consider Algorithm \ref{alg:advanced_EPDMP}. We shall consider these two cases separately.

\textbf{Case $v\pot'(x)>0$:} Note that it is sufficient to show that outside of a sufficiently large compact set
\begin{equation*}
    \frac{\mathbb{E}_{x,v} \lf_{\alpha,\beta}(\Xbar^{\ref{alg:basic_EPDMP}}_\delta,\Vbar^{\ref{alg:basic_EPDMP}}_\delta)}{ \lf_{\alpha,\beta}(x,v)} < 1.
\end{equation*}
We can expand $\mathbb{E}_{x,v} \lf_{\alpha,\beta}(\Xbar^{\ref{alg:basic_EPDMP}}_\delta,\Vbar^{\ref{alg:basic_EPDMP}}_\delta)$ as
\begin{equation*}
    \mathbb{E}_{x,v} \lf_{\alpha,\beta}(\Xbar_\delta^{\ref{alg:basic_EPDMP}},\Vbar^{\ref{alg:basic_EPDMP}}_\delta)=   e^{-\delta\lambda(x,v)}\lf_{\alpha,\beta}(x+v\delta,v)+(1-e^{-\delta\lambda(x,v)})\lf_{\alpha,\beta}(x+v\delta,-v).
\end{equation*}
Using the definition of $\lf_{\alpha,\beta}$ we can write
\begin{align*}
    \frac{\lf_{\alpha,\beta}(x+v\delta,v)}{\lf_{\alpha,\beta}(x,v)} &= \exp\left(\alpha (\pot(x+v\delta)-\pot(x)) +\beta \delta  v( \pot'(x+v\delta)-\pot'(x)) \right),\\
    \frac{\lf_{\alpha,\beta}(x+v\delta,-v)}{\lf_{\alpha,\beta}(x,v)} &= \exp\left(\alpha (\pot(x+v\delta)-\pot(x)) -\beta \delta  v( \pot'(x+v\delta)+\pot'(x)) \right).
\end{align*}
We can Taylor expand $U$ to find some $z_1,z_2,z_3$ such that
\begin{align*}
    \frac{\lf_{\alpha,\beta}(x+v\delta,v)}{\lf_{\alpha,\beta}(x,v)} &= \exp\left(\alpha ( \pot'(x) v \delta+\frac{1}{2} \pot''(z_1)\delta^2) +\beta \delta^2  \pot''(z_2) \right),\\
    \frac{\lf_{\alpha,\beta}(x+v\delta,-v)}{\lf_{\alpha,\beta}(x,v)} &= \exp\left(\alpha (\pot'(x) v \delta+\frac{1}{2}\pot''(z_1)\delta^2) -\beta \delta  (2\pot'(x) v+\pot''(z_3)\delta) \right).
\end{align*}
Thus we have
\begin{align*}
    &\frac{\mathbb{E}_{x,v} \lf_{\alpha,\beta}(\Xbar^{\ref{alg:basic_EPDMP}}_\delta,\Vbar^{\ref{alg:basic_EPDMP}}_\delta)}{ \lf_{\alpha,\beta}(x,v)} = e^{-\delta\lambda(x,v)}\exp{\left(\alpha (\pot'(x) v \delta+\frac{1}{2} \pot''(z_1)\delta^2) +\beta \delta^2  \pot''(z_2) \right)}\\
    &\qquad+(1-e^{-\delta\lambda(x,v)})\exp\left(\alpha (\pot'(x) v \delta+\frac{1}{2}\pot''(z_1)\delta^2) -\beta \delta  (2\pot'(x) v+\pot''(z_3)\delta) \right).
\end{align*}
Rearranging we can rewrite this as
\begin{align}
    & \frac{\mathbb{E}_{x,v} \lf_{\alpha,\beta}(\Xbar^{\ref{alg:basic_EPDMP}}_\delta,\Vbar^{\ref{alg:basic_EPDMP}}_\delta)}{ \lf_{\alpha,\beta}(x,v)} = \nonumber\\
    & = \exp{\left(\!\!-\delta\lambda(x,v)\!+\!\alpha \pot'(x)v\delta+ \frac{\alpha}{2}\pot''(z_1)\delta^2\right)}\!\!\left(  e^{\left( \beta \delta^2  \pot''(z_2) \right)} \!-\!e^{\left( -\beta \delta  (2\pot'(x) v+\pot''(z_3)\delta)  \right)}\right)\label{eq:lfjump}\\
    & \quad + \exp{\left((\alpha-2\beta) \pot'(x) v \delta+\frac{1}{2}\alpha\pot''(z_1)\delta^2 -\beta \delta^2  \pot''(z_3) \right)}.\label{eq:lfnojump}
\end{align}
Recall that in this case $\lambda(x,v) \geq v\pot'(x)>0$. Thus for the first term \eqref{eq:lfjump}
\begin{align*}
    &\exp{\left(-\delta\lambda(x,v)+\alpha \pot'(x)v\delta+ \frac{\alpha}{2}\pot''(z_1)\delta^2\right)}\left(  e^{\left( \beta \delta^2  \pot''(z_2) \right)} -e^{\left( -\beta \delta  (2\pot'(x) v+\pot''(z_3)\delta)  \right)}\right)\leq\\
    &\quad\leq\exp{\left(-\delta\lambda(x,v)+\alpha \pot'(x)v\delta+ \frac{\alpha}{2}\pot''(z_1)\delta^2\right)}e^{\left( \beta \delta^2  \pot''(z_2) \right)} \\
    & \quad\leq\exp{\left(-(1-\alpha)\delta v\pot'(x)+ \frac{\alpha}{2}\pot''(z_1)\delta^2\right)}e^{\left( \beta \delta^2  \pot''(z_2) \right)}.
\end{align*}
Now choose $0<\alpha <\min\{1,2\beta\}$ and recall that by assumption $\pot'$ diverges to infinity faster than $\pot''$. It follows that, outside of a large enough compact set, both \eqref{eq:lfjump} and \eqref{eq:lfnojump} can be made arbitrarily small. 



\textbf{Case $v\pot'(x)<0$:} In this case $\lambda(x,v)=\gamma(x)$. We expand $\mathbb{E}_{x,v} \lf_{\alpha,\beta}(\Xbar^{\ref{alg:advanced_EPDMP}}_\delta,\Vbar^{\ref{alg:advanced_EPDMP}}_\delta)$ as
\begin{align*}
    \frac{\mathbb{E}_{x,v} \lf_{\alpha,\beta}(\Xbar_\delta^{\ref{alg:advanced_EPDMP}},\Vbar^{\ref{alg:advanced_EPDMP}}_\delta)}{\lf_{\alpha,\beta}(x,v)}& =   e^{-\delta\lambda(x,v)}\frac{\lf_{\alpha,\beta}(x+v\delta,v)}{\lf_{\alpha,\beta}(x,v)}\\
    & \quad+\int_0^\delta \lambda(x,v)e^{-s\lambda(x,v)}\frac{\lf_{\alpha,\beta}(x+v(2s-\delta),-v)}{\lf_{\alpha,\beta}(x,v)}\dd s.
\end{align*}
Similarly to above, we can use Taylor's theorem to find $z_1,z_2,z_3,z_4$ with
\begin{equation}\label{eq:lyap_negcase}
\begin{aligned}
    &\frac{\mathbb{E}_{x,v} \lf_{\alpha,\beta}(\Xbar_\delta^{\ref{alg:advanced_EPDMP}},\Vbar^{\ref{alg:advanced_EPDMP}}_\delta)}{\lf_{\alpha,\beta}(x,v)}=   e^{-\delta\lambda(x,v)}\exp\left(\alpha ( \pot'(x) v \delta+\frac{1}{2} \pot''(z_1)\delta^2) -\beta \delta^2  \pot''(z_2) \right)\\
    &\quad+\int_0^\delta \lambda(x,v)e^{-s\lambda(x,v)}\exp\Bigg(\alpha (\pot'(x) v (2s-\delta)+\frac{1}{2}\pot''(z_3)(2s-\delta)^2) \\
    & \qquad+\beta \delta  (2\pot'(x) v+\pot''(z_4)(2s-\delta)) \Bigg)\dd s.
\end{aligned}
\end{equation}
Taking advantage of $-\delta\leq 2s-\delta\leq \delta$ we obtain the bound 
\begin{align*}
    \exp\Bigg(\alpha (\pot'(x) v (2s-\delta) +\beta \delta  (2\pot'(x) v) \Bigg) \leq \exp\Bigg((-\pot'(x) v)(\alpha -2\beta) \delta  \Bigg) .
\end{align*}
Using this bound together with the assumption that $-v\pot'(x)$ diverges to $+\infty$ faster than $\pot''$, we obtain that for $0<\alpha<2\beta$ the right hand side of \eqref{eq:lyap_negcase} can be made arbitrarily small for sufficiently large values of $x$.

Combining the two cases above we obtain the statement of the lemma.
\end{proof}

\begin{lemma}\label{lem:lfordering}
Assume that $\psi\in \C^2$ satisfies \eqref{eq:growthbounds}. Let $\lf_{\alpha,\beta}(x,v;\delta)$ be given by \eqref{eq:approximatelf} and $\lf_{\alpha,\epsilon}(x,v)$ be given by \eqref{eq:1dlfZZS}. Then for any $0<\alpha_1<\overline{\alpha}<\alpha_2<1$ there exist positive constants $C,C'>0$ with
\begin{equation}\label{eq:lfordering}
    \lf_{\alpha_1,\epsilon}(x,v)\leq C'e^{\overline{\alpha}\pot(x)} \leq C\lf_{\alpha_2,\beta}(x,v;\delta).
\end{equation}
\end{lemma}

\begin{proof}[Proof of Lemma \ref{lem:lfordering}]
Let us first consider $\lf_{\alpha_1,\epsilon}(x,v)$, since $\lvert\phi_\epsilon(s)\rvert \leq  \epsilon\lvert s\rvert/2$ we have
\begin{equation*}
    \lf_{\alpha_1,\epsilon}(x,v) \leq \exp\left(\alpha_1\pot(x)+\frac{\epsilon}{2}\lvert \pot'(x)\rvert\right).
\end{equation*}
By \eqref{eq:growthbounds} there exists $R>0$ such that for any $\lvert x\rvert >R$ we have $\lvert\pot'(x)\rvert \leq 2\epsilon^{-1}(\overline{\alpha}-\alpha_1) \pot(x)$. Therefore for $\lvert x\rvert >R$ we have
\begin{equation*}
    \lf_{\alpha_1,\epsilon}(x,v) \leq \exp\left(\overline{\alpha}\pot(x)
    \right).
\end{equation*}
Setting $C'=\exp(\sup_{\lvert x\rvert \leq R}\lvert\pot'(x)\rvert)$ we have the left hand side of \eqref{eq:lfordering}.

Similarly, we have
\begin{equation*}
    \lf_{\alpha_2,\beta}(x,v;\delta) \geq \exp\left(\alpha_2\pot(x) - \beta\delta_0\lvert\pot'(x)\rvert\right).
\end{equation*}
Using \eqref{eq:growthbounds} for $x$ sufficiently large we have that $\beta\delta_0\lvert\pot'(x)\rvert \leq (\alpha_2-\overline{\alpha})\pot(x)$ and hence the right hand side of \eqref{eq:lfordering} follows.
\end{proof}

\end{appendix}



%
%

\begin{funding}
This work is part of the research programme `Zigzagging through computational barriers' with project number 016.Vidi.189.043, which is financed by the Dutch Research Council (NWO).
\end{funding}



\bibliographystyle{imsart-number} 
\bibliography{bibliography.bib}       


\end{document}